\documentclass[11pt]{amsart}

\usepackage{
	amsmath,
	amsfonts,
	amssymb,
	amsthm,
	amscd,
	comment,
	enumitem,
	etoolbox,
	gensymb,
	mathtools,
	mathdots,
	stmaryrd,
}
\usepackage[usenames,dvipsnames]{xcolor}
\usepackage{todonotes}
\usepackage[all]{xy}
\usepackage[hyphens]{url}
\usepackage[hyphenbreaks]{breakurl}
\usepackage{xparse}
\usepackage{tikz-cd}
\usepackage{tikz, tikz-3dplot, pgfplots}
\usepackage{tkz-graph}

\usepackage[utf8]{inputenc}  
\usepackage{bbm}
\usepackage[colorlinks=true, linkcolor=blue, citecolor=blue, urlcolor=blue, breaklinks=true]{hyperref}

\usepackage[capitalize]{cleveref} 
\crefname{defi}{Definition}{Definitions}
\crefname{exam}{Example}{Examples}
\crefname{egs}{Example}{Examples}
\crefname{lem}{Lemma}{Lemmas}
\crefname{prob}{Proposition}{Propositions}
\crefname{theo}{Theorem}{Theorems}
\crefname{equation}{}{}
\crefname{enumi}{}{}

\newcommand{\bbC}{\mathbb{C}}

\newcommand{\bbZ}{\mathbb{Z}}

\newcommand{\calC}{\mathcal{C}}

\newcommand{\ob}{\mathrm{ob}}
\newcommand{\id}{\mathrm{id}}

\newcommand{\bm}{\begin{bmatrix}}
	\newcommand{\nm}{\end{bmatrix}}

\usepackage{tikz}
\usetikzlibrary{arrows.meta}
\usetikzlibrary{decorations.markings}
\usetikzlibrary{calc}

\tikzset{anchorbase/.style={>=To,baseline={([yshift=-0.5ex]current bounding box.center)}}}
\tikzset{ 
		centerzero/.style={>=To,baseline={([yshift=-0.5ex](#1))}},
		centerzero/.default={0,0}
	}
	\tikzset{wipe/.style={white,line width=3pt}}
	
	\tikzset{gmod/.style={ForestGreen, thick}}
	\tikzset{rmod/.style={red, thick}}
	\tikzset{blmod/.style={blue, thick}}
	\tikzset{bmod/.style={black, thick}}
	\tikzset{gmod>/.style={->, thick, ForestGreen}}
	\tikzset{<gmod/.style={<-, thick, ForestGreen}}
	
	\newcommand\dotlabel[1]{$\scriptstyle{#1}$}

	\newcommand\btoken[4][0px]{
		\filldraw[black] (#2) circle (1.5pt) node[anchor=#3] {\raisebox{#1}{\dotlabel{#4}}}
	}

	\newcommand{\dcap}{
		\begin{tikzpicture}[anchorbase]
			\draw[bmod] (0,0) -- (0,0.125) arc (180:0:0.25) -- (0.5,0);
		\end{tikzpicture}	
	}
	\newcommand{\dcup}{
		\begin{tikzpicture}[anchorbase]
			\draw[bmod] (0,0.125) -- (0,0) arc (180:360:0.25) -- (0.5,0.125);
		\end{tikzpicture}	
	}
	\newtheorem{theo}{Theorem}[section]
	\theoremstyle{definition}
	\newtheorem{defi}[theo]{Definition}
	
	\newtheorem{lem}[theo]{Lemma}
	\newtheorem{theor}[theo]{Theorem}
	\newtheorem{coro}[theo]{Corollary}
	\newtheorem{rmk}[theo]{Remark}
	\newtheorem{exam}[theo]{Example}
	
	\newtheorem{prop}[theo]{Proposition}

	\numberwithin{equation}{section}
	\allowdisplaybreaks

	\setenumerate[1]{label=(\alph*)}
	
	\usepackage{environ}
	\newcommand{\comsize}{3 in}

	\NewEnviron{commeqn}[1][2.5 in]{\renewcommand{\comsize}{#1} \begin{align*} \BODY \end{align*}}
	\graphicspath{{/}}
	
	\newtoggle{comments}
	\newtoggle{details}
	\newtoggle{detailsnote}
	
	\iftoggle{comments}{%
		\newcommand{\scomments}[1]{
			\ \\
			{\color{Violet}
				\textbf{SSS:} #1
			}
			\ \\
		}
	}{%
		\newcommand{\scomments}[1]{}
	}
	
	\iftoggle{details}{%
		\newcommand{\details}[1]{
			\ \\
			{\color{OliveGreen}
				\textbf{Details:} #1
			}
			\\
		}
	}{%
		\newcommand{\details}[1]{}
	}
	
\begin{document}
		
		\title{Categorical Lie-Rinehart modules and Shen-Larsson functors}
		\author{Han Dai}
		\address{School of Mathematics, Hefei University of Technology, Hefei, 230000, Anhui, P. R. China}
		\email{dh2718@mail.ustc.edu.cn}
		\address{Shenzhen International Center for Mathematics, Southern University of Science and Technology, Shenzhen, 518055, Guangdong, P. R. China}
		\author{Vyacheslav Futorny}
		\email{futorny@sustech.edu.cn}
		\author{Huimin Gao}
		\address{Shenzhen International Center for Mathematics, Southern University of Science and Technology, Shenzhen, 518055, Guangdong, P. R. China}
		\email{12431006@mail.sustech.edu.cn}
		
		
		\maketitle
		\thispagestyle{empty}

\begin{abstract}
We develop a categorical framework for Lie-Rinehart monoids and their weak modules in a symmetric monoidal category. Using crossed
homomorphisms, we construct a natural action of the monoidal category of modules over a Lie monoid on the category of weak Lie-Rinehart modules, thereby obtaining categorical versions of the Shen-Larsson functors. We further characterize the conditions under which the category of weak modules admits a monoidal structure and identify the corresponding condition for the associated functors to be strict monoidal. A dual theory for Lie-Rinehart comonoids and weak comodules is developed using cocrossed homomorphisms. Combining the module and comodule constructions, we obtain a bimodule category structure on the category of weak modules. Finally, we specialize the general framework to the symmetric monoidal category of super vector spaces, recovering Lie-Rinehart superalgebras and their associated Shen-Larsson-type constructions. 
\end{abstract}
		
\section{Introduction}
Since the foundational work of B\'enabou \cite{Ben63}, Mac Lane \cite{Mac63,Mac98}, and Kelly \cite{Kel64}, monoidal categories have become fundamental in many areas of algebra and representation theory. 
A basic example arises from Hopf algebras: the category of modules over a Hopf algebra carries a natural monoidal structure induced by the coproduct. 
More generally, many familiar algebraic objects admit intrinsic categorical analogues. 
In particular, Majid \cite{Maj93,Maj94} developed the theory of algebras, coalgebras, bialgebras, and Hopf algebras in braided monoidal categories. 
This point of view makes it possible to separate arguments that depend only on tensor categorical structures from those that are specific to vector spaces, and therefore provides a natural setting for extending classical algebraic constructions to graded, braided, and other monoidal contexts.

Lie-Rinehart algebras form another class of algebraic structures, which was introduced by Rinehart \cite{Rin63}. More explicitly, a Lie-Rinehart algebra consists of a commutative associative algebra $A$, a Lie algebra $L$, an $A$-module structure on $L$, and an anchor map from $L$ to the Lie algebra of derivations of $A$, subject to suitable compatibility conditions. 
Thus, Lie-Rinehart algebras simultaneously encode algebraic and Lie theoretic information. 
Two standard examples are $\big(A,\mathrm{Der}(A)\big)$
for a commutative algebra $A$, and $\big(C^{\infty}(M),\mathrm{Vect}(M)\big)$
for a smooth manifold $M$. 
The latter may be viewed as the algebraic prototype of the tangent Lie algebroid of $M$. 
In this sense, Lie-Rinehart algebras provide a common language for structures appearing in commutative algebra, differential geometry, and representation theory.

As in the case of Lie algebras, one can associate a universal enveloping algebra with a Lie-Rinehart algebra. 
Rinehart \cite{Rin63} and Huebschmann \cite{Hue90} introduced this construction, and subsequent developments clarified the relationship between Lie-Rinehart modules and modules over the corresponding enveloping algebra. For more details about  Lie-Rinehart algebras, see \cite{Hue04,Hue21,BKS24,CLP04,Sar22,BEMS}. 

Representation theory of Lie-Rinehart algebras therefore naturally combines features of both associative algebra and Lie algebra representations. 
A particularly useful notion in this context is that of a \emph{weak module}, in which the associative and Lie actions satisfy a compatibility relation analogous to the Leibniz rule.
An important source of weak modules is closely related to the Shen-Larsson construction. 
The Shen-Larsson functor, originating in the work of Shen \cite{She86} and Larsson \cite{Lar92}, {turns modules over a Lie algebra into representations of vector field Lie algebras}.
Besides connecting different representation categories, this construction has proved useful for producing and studying new classes of irreducible modules. 
{Specifically, the construction combines a representation on one tensor factor with an existing module structure on the other to give a new representation}.

Recently, Pei, Sheng, Tang, and Zhao \cite{PSTZ23} reformulated this phenomenon in terms of Lie-Rinehart algebras and crossed homomorphisms. 
They showed that a crossed homomorphism gives rise to natural Shen-Larsson-type functors and, more conceptually, that the category of weak modules over a Lie-Rinehart algebra carries the structure of a left module category over a suitable monoidal category of Lie algebra representations. 
This interpretation suggests that the Shen-Larsson construction is not confined to ordinary vector spaces. Rather, its constituents including Lie objects, associative algebra objects, module objects, tensor products, and crossed homomorphisms, are all categorical in nature.

The aim of the present paper is to develop this observation systematically. 
Let $\mathcal C$ be a linear symmetric monoidal category. 
We introduce the notion of a \emph{Lie-Rinehart monoid} in $\mathcal C$, in which the commutative algebra and Lie algebra occurring in the classical definition are replaced by a commutative monoid and a Lie monoid, respectively. 
We also introduce the corresponding category ${\bf WMod}^{\mathcal C}_A(L)$ of weak modules. 
The defining compatibility conditions are formulated intrinsically in $\mathcal C$ and can be expressed conveniently by string diagrams. 
When $\mathcal C=\mathcal{V}ec(\mathbb C)$, these notions reduce to the usual Lie-Rinehart algebra and its weak modules.

Our first main result concerns crossed homomorphisms. 
Suppose $(A,L)$ is a Lie-Rinehart monoid, $\mathfrak L$ is a Lie monoid in $\mathcal C$, and $\mathcal H:L\rightarrow \mathfrak L\otimes A$ is a crossed homomorphism. 
We show that every $\mathfrak L$-module $V$ and every weak $(A,L)$-module $M$ canonically determine a new weak $(A,L)$-module on $V\otimes M$. 
More precisely, this construction defines a bifunctor
\[
{\bf Mod}_{\mathcal C}(\mathfrak L)
\times
{\bf WMod}^{\mathcal C}_A(L)
\rightarrow
{\bf WMod}^{\mathcal C}_A(L),
\]
and turns ${\bf WMod}^{\mathcal C}_A(L)$ into a strict left module category over the monoidal category ${\bf Mod}_{\mathcal C}(\mathfrak L)$. 
Equivalently, the crossed homomorphism $\mathcal H$ gives rise to a strict monoidal functor
\[
\mathcal{F}: {\bf Mod}_{\mathcal C}(\mathfrak L)
\rightarrow
\operatorname{End}\big({\bf WMod}^{\mathcal C}_A(L)\big).
\]
This provides a categorical form of the Shen-Larsson construction and recovers the construction of \cite{PSTZ23} in the ordinary vector space case.

We next investigate when the category of weak modules itself admits a monoidal structure. 
Under suitable additional compatibility assumptions on the Lie-Rinehart monoid $(A,L)$, we show that the tensor product in $\mathcal C$ induces a monoidal product on
${\bf WMod}^{\mathcal C}_A(L)$. 
Once this monoidal structure is available, it is natural to ask whether the Shen-Larsson-type endofunctors induced by crossed homomorphisms preserve it. 
We give a necessary and sufficient condition for the corresponding functor
\[
\mathcal F_{\mathcal H}:
{\bf WMod}^{\mathcal C}_A(L)
\rightarrow
{\bf WMod}^{\mathcal C}_A(L)
\]
to be strict monoidal. 
Thus, besides producing new weak modules, crossed homomorphisms can in suitable situations generate monoidal endofunctors of the weak-module category.

We also develop a dual version of the theory. 
In a rigid symmetric monoidal category, we introduce Lie-Rinehart comonoids, weak comodules, and cocrossed homomorphisms. 
A cocrossed homomorphism induces a right action of a monoidal category of Lie comonoid comodules on the category of weak comodules. 
More precisely, the category ${\bf WCMod}^{\mathcal C}_A(L)$
is shown to be a strict right module category over the corresponding category of comodules. 
Combining the monoid and comonoid constructions leads naturally to Lie-Rinehart double monoids and weak double modules. 
For these objects, the left and right actions are compatible and endow the category of weak double modules with a strict bimodule category structure.

Finally, we specialize the general constructions to familiar symmetric monoidal categories. 
For $\mathcal C=\mathcal{V}ec(\mathbb C)$, our framework recovers the corresponding results for ordinary Lie-Rinehart algebras. 
For the symmetric monoidal category
$s\mathcal{V}ec(\mathbb C)$ of $\mathbb Z_2$-graded vector spaces, the categorical definitions yield Lie-Rinehart superalgebras and their weak modules. 
In particular, crossed homomorphisms of Lie superalgebras give rise to Shen-Larsson-type constructions in the super setting. 
{These examples show that the categorical framework is flexible enough to handle the classical and super cases uniformly}.

The paper is organized as follows. 
In Section~\ref{pre}, we recall the basic notions concerning linear symmetric monoidal categories, monoids, comonoids, bimonoids, and Lie monoids. 
We also introduce the graphical notation used throughout the paper. 
In Section~\ref{secoflierinehart}, we define Lie-Rinehart monoids and weak modules and develop their basic properties. 
We then introduce crossed homomorphisms and prove that they induce a left module-category structure on the category of weak modules. 
We further study monoidal structures on weak module categories and characterize when the associated Shen-Larsson-type functors are strict monoidal. 
In the rigid setting, we develop the dual theory of Lie-Rinehart comonoids, weak comodules, and cocrossed homomorphisms, and then combine the two constructions to obtain bimodule categories of weak double modules. 
In Section~\ref{app}, we apply the general theory to vector spaces and super vector spaces. 
In the Appendix, we recall the graphical calculus for monoidal categories used in the proofs.

\textbf{Conventions}. 
Throughout the paper, all vector spaces, associative (super)algebras, coalgebras, bialgebras, and Lie (super)algebras are assumed to be defined over $\mathbb C$.

\section{Monoidal categories}\label{pre}
In this section we review some basic concepts and results concerning monoidal categories. We also introduce the categorical counterparts of algebras, coalgebras, bialgebras, and Lie algebras. For more details, see \cite{Mac98,EGNO15}.
		
Let $\mathcal{C}$ be a category and $\otimes$ be a covariant functor from $\calC \times \calC$ to $\calC,$ i.e., for any $V,W\in \mathrm{ob}(\calC),$ we have the tensor product $V\otimes W \in \mathrm{ob}(\calC)$ of $V$ and $W$. Given morphisms $f: V \rightarrow V'$ and $g: W\rightarrow W'$ of $\calC$, there exists a morphism $f\otimes g: V\otimes W \rightarrow V'\otimes W'$. Moreover, we have 
\begin{align}\label{defoftensor}
	\mathrm{id}_{V\otimes W} = \mathrm{id}_V \otimes \mathrm{id}_W \ \  \mathrm{and} \ \ (f \otimes g) \circ (f'\otimes g')= (f\circ f') \otimes (g\circ g'),   
\end{align}
whenever the compositions are defined. 
An associativity constraint for $\otimes $ is a natural isomorphism 
\[a: \otimes \circ (\otimes \times \mathrm{id}) \rightarrow \otimes \circ ( \mathrm{id} \times \otimes). \]
This means that for any $U,V,W\in \mathrm{ob}(\calC)$, there exists a family of natural isomorphisms
\[a_{U,V,W}: (U\otimes V)\otimes W \rightarrow U\otimes (V\otimes W). \]
Fix an object $I \in \ob(\calC)$. A left unit constraint (resp. a right unit constraint) with respect to $I$ is a natural isomorphism 
\[l:\otimes \circ (I\times \id) \rightarrow \id \ \ (\mathrm{resp.} \ \ r: \otimes \circ (\id \times I) \rightarrow \id).  \]
This means that there exists a family of natural isomorphisms
\[l_V: I\otimes V \rightarrow V \ \ (\mathrm{resp.} \ r_V:V\otimes I \rightarrow V),  \]
for any $V\in \ob(\calC)$. 
		
\begin{defi}
A \emph{monoidal category} $(\calC, \otimes, I, a, l,r )$ is a category $\calC$ equipped with a bifunctor $\otimes: \calC \times \calC \rightarrow \calC,$ with an object $I$, called the unit of the monoidal category, with an associativity constraint $a$, left and right unit constraints $l,r$ with respect to $I$ such that the following diagrams 
\[
\begin{tikzcd}[column sep=1.8em, row sep=3.2ex]
				& \bigl((U\otimes V)\otimes W\bigr)\otimes Z
				\ar[dl, "a_{U,V,W}\otimes\mathrm{id}_Z"']
				\ar[dr, "a_{U\otimes V,W,Z}"] \\
				(U\otimes(V\otimes W))\otimes Z
				\ar[d, "a_{U,V\otimes W,Z}"']
				&& (U\otimes V)\otimes(W\otimes Z)
				\ar[d, "a_{U,V,W\otimes Z}"] \\
				U\otimes\bigl((V\otimes W)\otimes Z\bigr)
				\ar[rr, "\mathrm{id}_U\otimes a_{V,W,Z}"']
		 	&& U\otimes\bigl(V\otimes(W\otimes Z)\bigr)
\end{tikzcd}
\]
and
\[
			\begin{tikzcd}[column sep=3em, row sep=3.5ex]
				(V\otimes I)\otimes W
				\ar[rr, "a_{V,I,W}"]
				\ar[dr, "r_V\otimes\mathrm{id}_W"']
				&& V\otimes(I\otimes W)
				\ar[dl, "\mathrm{id}_V\otimes l_W"] \\
				& V\otimes W
			\end{tikzcd}
\]
commute. Moreover, it is said to be \emph{strict} if the associativity and unit constraints $a,l,r$ are all identities. 
\end{defi}
		
Recall that a braiding in $\calC$ is a family of natural isomorphisms $c_{V,W}: V\otimes W \rightarrow W\otimes V$
such that 
\begin{align}\label{braid1}
	c_{U,V\otimes W}= (\id_V \otimes c_{U,W} ) \circ (c_{U,V}\otimes \id_W)
\end{align}
and
\begin{align}\label{braid2}
	c_{U\otimes V,W}=(c_{U,W}\otimes \id_V) \circ (\id_U \otimes c_{V,W})
\end{align}
for all $U,V,W \in \ob(\calC)$. 
\begin{defi}[Definition $2.1$, \cite{JS93}]
A \emph{braided monoidal category} $(\calC, \otimes, I,a,l,r,c )$ is a monoidal category $\calC$ equipped with a braiding $c$. Moreover, if $c_{W,V} \circ c_{V,W} = \id_{V\otimes W }$ for all $V,W\in \ob(\calC)$, then $\calC$ is called a \emph{symmetric monoidal category}. 
\end{defi}
		
Next, we give several examples of monoidal categories.
		
\begin{exam}
The most basic example is the category $\mathcal{V}ec(\bbC)$ of complex vector spaces. 
Its monoidal structure is given by the usual tensor product of vector spaces over $\mathbb{C}$, and the unit object is the ground field $\mathbb{C}$. 
The associativity and unit constraints are the canonical isomorphisms
 \[(U\otimes V) \otimes W \cong U\otimes (V\otimes W) \ \  \mathrm{and} \ \ \bbC\otimes V \cong V \cong V\otimes \bbC,   \]
for all complex vector spaces $U,V$ and $W$. 
Moreover, $\mathcal{V}ec(\bbC)$ is a symmetric monoidal category with braiding $c_{V,W}: V\otimes W \to W\otimes V, \ v\otimes w \mapsto w \otimes v$ for all $v\in V$ and $w \in W$.
\end{exam}
		
\begin{exam}
A slightly richer example is provided by the category $s\mathcal{V}ec(\bbC)$ of $\bbZ_2$-graded vector spaces, also called super vector spaces. 
An object of this category is a vector space $V=V_{\bar 0}\oplus V_{\bar 1},$ and morphisms are linear maps preserving the $\mathbb{Z}_2$-grading. The unit object is again $\mathbb{C}$, regarded as concentrated in degree $\bar 0$.
For two super vector spaces $V$ and $W$, their tensor product is endowed with the $\mathbb{Z}_2$-grading
\[
V\otimes W=(V\otimes W)_{\bar 0}\oplus (V\otimes W)_{\bar 1},
\]
where
$(V\otimes W)_{\bar 0}
=(V_{\bar 0}\otimes W_{\bar 0})
\oplus
(V_{\bar 1}\otimes W_{\bar 1})
$
and
$(V\otimes W)_{\bar 1}
=(V_{\bar 0}\otimes W_{\bar 1})
\oplus
(V_{\bar 1}\otimes W_{\bar 0}).
$
If $f:V\to V'$ and $g:W\to W'$ are morphisms in $s\mathcal{V}ec(\mathbb{C})$, then their tensor product is given by
\[
(f\otimes g)(v\otimes w)=(-1)^{|g||v|}f(v)\otimes g(w).
\]
The associativity and unit constraints are the usual canonical ones. The braiding is given by the Koszul sign rule
\[
c_{V,W}:V\otimes W\longrightarrow W\otimes V,
\quad
v\otimes w\longmapsto (-1)^{|v||w|}w\otimes v,
\]
for elements $v\in V$ and $w\in W$.
\end{exam}
		
\begin{exam}
Let $L$ be a Lie superalgebra. The category $\mathbf{Mod}(L)$ of  $L$-modules in $s\mathcal{V}ec(\mathbb{C})$ is a symmetric monoidal category. More precisely, if $V$ and $W$ are $L$-modules, then $V\otimes W$ becomes an $L$-module via the diagonal action
\[
x\cdot (v\otimes w)
=
x\cdot v\otimes w
+
(-1)^{|x||v|}v\otimes x\cdot w,
\]
for elements $x\in L, v\in V$ and $w\in W$. The unit object is the trivial one-dimensional $L$-module $\mathbb{C}$. The symmetric braiding is induced from the braiding in $s\mathcal{V}ec(\mathbb{C})$. 
\end{exam}

\begin{exam}
We also recall the standard monoidal structure on cochain complexes. A \emph{differential $\mathbb{Z}$-graded vector space}, or equivalently a cochain complex, is a pair $(V,d_V)$, where
\[
V=\bigoplus_{n\in \mathbb{Z}}V^n
\]
is a $\mathbb{Z}$-graded vector space and $d_V:V\to V$ is a linear map of degree $1$ such that $d_V^2=0$. Thus one has a complex
\[
\cdots
\longrightarrow
V^{n-1}
\overset{d_V}{\longrightarrow}
V^n
\overset{d_V}{\longrightarrow}
V^{n+1}
\longrightarrow
\cdots .
\]
For a nonzero homogeneous element $v \in V^n$, set $|v|=n$.

The category $\mathrm{dg}\mathcal{V}ec(\mathbb{C})$ of differential $\mathbb{Z}$-graded vector spaces is a symmetric monoidal category. A morphism $f:(V,d_V)\longrightarrow (W,d_W)$
is a degree-zero linear map $f:V\to W$ satisfying
\[
d_W\circ f=f\circ d_V.
\]
For two objects $(V,d_V)$ and $(W,d_W)$, their tensor product is defined by $(V,d_V)\otimes (W,d_W)
=
(V\otimes W,d_{V\otimes W}),$
where
\[
d_{V\otimes W}(v\otimes w)
=
d_V(v)\otimes w
+
(-1)^{|v|}v\otimes d_W(w)
\]
for elements $v\in V$ and $w\in W$. With this convention, one has
$d_{V\otimes W}^2=0.$
Indeed, for $v\in V$ and $w\in W$, we have 
\[
\begin{aligned}
	d_{V\otimes W}^2(v\otimes w)
	&=
	d_{V\otimes W}\bigl(d_V(v)\otimes w
	+
	(-1)^{|v|}v\otimes d_W(w)\bigr)  \\
	&=
	d_V^2(v)\otimes w
	+
	(-1)^{|v|+1}d_V(v)\otimes d_W(w)
	+
	(-1)^{|v|}d_V(v)\otimes d_W(w)
	+
	v\otimes d_W^2(w)  \\
	&=0.
\end{aligned}
\]  
The unit object is $\mathbb{C}$, regarded as a complex concentrated in degree $0$ with zero differential. 
The associativity and unit constraints are the usual canonical ones, and the braiding is given by
\[
c_{V,W}:V\otimes W\longrightarrow W\otimes V,
\quad
v\otimes w\longmapsto (-1)^{|v||w|}w\otimes v
\]
for all $v\in V$ and $w \in W$.  For further details, see \cite{Kel06}.
\end{exam}

\begin{exam}
For any category $\mathcal{D}$, the category $\mathrm{End}(\mathcal{D})$ of endofunctors of $\mathcal{D}$ is a strict monoidal category. Its objects are endofunctors
$F:\mathcal{D}\longrightarrow \mathcal{D},$
and its morphisms are natural transformations between such functors. The tensor product is given by composition of endofunctors: $F\otimes G:=F\circ G,$
and the unit object is the identity functor $\mathrm{Id}_{\mathcal{D}}$. 
With these choices, the associativity and unit constraints are identities, and hence $\mathrm{End}(\mathcal{D})$ is strict monoidal. In general, however, this monoidal category is not braided, and hence not symmetric. Indeed, a braiding would require natural isomorphisms $F\circ G \cong G\circ F$
for all endofunctors $F,G$, which do not exist in general.
\end{exam}

By the Mac Lane's strictness theorem, every monoidal category is monoidally equivalent to a strict monoidal category \cite{Mac98}.
Since all constructions below are invariant under monoidal equivalence, we work with strict monoidal categories.

\section{Lie-Rinehart (co)monoids}\label{secoflierinehart}
In this section we introduce a categorical analogue of a Lie-Rinehart (co)monoid, together with the associated notions of a weak (co)module and a (co)crossed homomorphism in $\calC$. Under additional conditions, we further prove that the category of weak (co)modules over a Lie-Rinehart (co)monoid carries the structure of a module category over the monoidal category constructed from a Lie (co)monoid.
\subsection{Basic definitions}
\begin{defi}\label{defiofliemonoid}
Let $(A,m,\eta)$ be a monoid, and let $(L,\mu)$ be a Lie monoid in $\calC$.  A \emph{Lie-Rinehart monoid}
in $\mathcal{C}$ consists of the data $(A,m,\eta,L,\mu,\alpha,\beta),$
where $\beta:A\otimes L\to L$ equips $L$ with the structure of a left $A$-module, and where
$\alpha:L\otimes A\longrightarrow A$
is a morphism in $\mathcal{C}$, called the anchor, satisfying the following
axioms:
\begin{equation} \label{derivation}
\begin{tikzpicture}[anchorbase]
					\draw[bmod] (0.25, 0) -- (0.5, 0.25);
					\draw[bmod] (0, 0) -- (0, 0.25) -- (0.25, 0.5);
					\draw[bmod] (0.75, 0) -- (0.5, 0.25) -- (0.25, 0.5) -- (0.25, 0.75);
					\btoken{0.5, 0.25}{west}{m};
					\btoken{0.25, 0.5}{west}{\alpha};
\end{tikzpicture}
				\ =\
\begin{tikzpicture}[anchorbase]
					\draw[bmod] (0.5, 0) -- (0.25, 0.25);
					\draw[bmod] (0, 0) -- (0.25, 0.25) -- (0.5, 0.5);
					\draw[bmod] (0.75, 0) -- (0.75, 0.25) -- (0.5, 0.5) -- (0.5, 0.75);
					\btoken[2px]{0.5, 0.5}{west}{m};
					\btoken[2px]{0.25, 0.25}{east}{\alpha};
\end{tikzpicture}
				\ +\
\begin{tikzpicture}[anchorbase]
					\draw[bmod] (0, -0.25) -- (0.25, 0) -- (0.5, 0.25);
					\draw[bmod] (0.25, -0.25) -- (0, 0) -- (0, 0.25) -- (0.25, 0.5);
					\draw[bmod] (0.75, -0.25) -- (0.75, 0) -- (0.5, 0.25) -- (0.25, 0.5) -- (0.25, 0.75);
					\btoken{0.5, 0.25}{west}{\alpha};
					\btoken[2px]{0.25, 0.5}{west}{m};
\end{tikzpicture}\ ,
\end{equation}
			
\begin{equation}\label{homofa}
				\
				\begin{tikzpicture}[anchorbase]
					\draw[bmod] (0.5, 0) -- (0.25, 0.25);
					\draw[bmod] (0, 0) -- (0.25, 0.25) -- (0.5, 0.5);
					\draw[bmod] (0.75, 0) -- (0.75, 0.25) -- (0.5, 0.5) -- (0.5, 0.75);
					\btoken[2px]{0.5, 0.5}{west}{\alpha};
					\btoken[2px]{0.25, 0.25}{east}{\beta};
				\end{tikzpicture}
				\ = \ 
				\begin{tikzpicture}[anchorbase]
					\draw[bmod] (0.25, 0) -- (0.5, 0.25);
					\draw[bmod] (0, 0) -- (0, 0.25) -- (0.25, 0.5);
					\draw[bmod] (0.75, 0) -- (0.5, 0.25) -- (0.25, 0.5) -- (0.25, 0.75);
					\btoken{0.5, 0.25}{west}{\alpha};
					\btoken{0.25, 0.5}{west}{m};
				\end{tikzpicture} \ ,
\end{equation}
			
\begin{equation} \label{liehom}
				\begin{tikzpicture}[anchorbase]
					\draw[bmod] (0.25, 0) -- (0.5, 0.25);
					\draw[bmod] (0, 0) -- (0, 0.25) -- (0.25, 0.5);
					\draw[bmod] (0.75, 0) -- (0.5, 0.25) -- (0.25, 0.5) -- (0.25, 0.75);
					\btoken{0.5, 0.25}{west}{\alpha};
					\btoken{0.25, 0.5}{west}{\alpha};
				\end{tikzpicture}
				\ =\
				\begin{tikzpicture}[anchorbase]
					\draw[bmod] (0.5, 0) -- (0.25, 0.25);
					\draw[bmod] (0, 0) -- (0.25, 0.25) -- (0.5, 0.5);
					\draw[bmod] (0.75, 0) -- (0.75, 0.25) -- (0.5, 0.5) -- (0.5, 0.75);
					\btoken[2px]{0.5, 0.5}{west}{\alpha};
					\btoken[2px]{0.25, 0.25}{east}{\mu};
				\end{tikzpicture}
				\ +\
				\begin{tikzpicture}[anchorbase]
					\draw[bmod] (0, -0.25) -- (0.25, 0) -- (0.5, 0.25);
					\draw[bmod] (0.25, -0.25) -- (0, 0) -- (0, 0.25) -- (0.25, 0.5);
					\draw[bmod] (0.75, -0.25) -- (0.75, 0) -- (0.5, 0.25) -- (0.25, 0.5) -- (0.25, 0.75);
					\btoken{0.5, 0.25}{west}{\alpha};
					\btoken[2px]{0.25, 0.5}{west}{\alpha};
				\end{tikzpicture}\ ,
\end{equation}
and 
\begin{equation} \label{compatible1}
				\begin{tikzpicture}[anchorbase]
					\draw[bmod] (0.25, 0) -- (0.5, 0.25);
					\draw[bmod] (0, 0) -- (0, 0.25) -- (0.25, 0.5);
					\draw[bmod] (0.75, 0) -- (0.5, 0.25) -- (0.25, 0.5) -- (0.25, 0.75);
					\btoken{0.5, 0.25}{west}{\beta};
					\btoken{0.25, 0.5}{west}{\mu};
				\end{tikzpicture}
				\ =\
				\begin{tikzpicture}[anchorbase]
					\draw[bmod] (0.5, 0) -- (0.25, 0.25);
					\draw[bmod] (0, 0) -- (0.25, 0.25) -- (0.5, 0.5);
					\draw[bmod] (0.75, 0) -- (0.75, 0.25) -- (0.5, 0.5) -- (0.5, 0.75);
					\btoken[2px]{0.5, 0.5}{west}{\beta};
					\btoken[2px]{0.25, 0.25}{east}{\alpha};
				\end{tikzpicture}
				\ + \
				\begin{tikzpicture}[anchorbase]
					\draw[bmod] (0, -0.25) -- (0.25, 0) -- (0.5, 0.25);
					\draw[bmod] (0.25, -0.25) -- (0, 0) -- (0, 0.25) -- (0.25, 0.5);
					\draw[bmod] (0.75, -0.25) -- (0.75, 0) -- (0.5, 0.25) -- (0.25, 0.5) -- (0.25, 0.75);
					\btoken{0.5, 0.25}{west}{\mu};
					\btoken[2px]{0.25, 0.5}{west}{\beta};
				\end{tikzpicture}\ .
\end{equation}
We abbreviate it as $(A,L,\alpha)$ when unambiguous. Moreover, the monoid $A$ is assumed to be commutative in what follows.
\end{defi}
		
When the monoidal category $\calC$ is specialized to the category $\mathcal{V}ec(\bbC)$ of complex vector spaces, the above notion of a Lie-Rinehart monoid reduces precisely to the classical Lie-Rinehart algebra introduced in \cite{Rin63}. 
		
Next, we present a crossed product construction for Lie-Rinehart monoids.
Starting from a Lie-Rinehart monoid $(A,L,\alpha)$, we keep the commutative
monoid $A$ fixed and replace the Lie part $L$ by $L\otimes A$. The resulting
structure is again a Lie-Rinehart monoid, and may be viewed as the categorical
counterpart of the classical crossed product construction. When
$\mathcal{C}=s\mathcal{V}ec(\mathbb{C})$, this construction specializes to
the example discussed in Example~$2.1$ of \cite{Che95}.
\begin{prop}\label{mainofanchor}
Let $(A, m,\eta, L, \mu, \alpha, \beta )$ be a Lie-Rinehart monoid in $\calC$. Then the tuple $(A,m,\eta,L\otimes A, \mu_A, \alpha_{A}, \beta_A)$ forms a new Lie-Rinehart monoid, where 
\begin{equation*}
\alpha_A= \ 
\begin{tikzpicture}[anchorbase]
	\draw[bmod] (0, -0.25) -- (0.25, 0) -- (0.5, 0.25);
					\draw[bmod] (0.25, -0.25) -- (0, 0) -- (0, 0.25) -- (0.25, 0.5);
					\draw[bmod] (0.75, -0.25) -- (0.75, 0) -- (0.5, 0.25) -- (0.25, 0.5) -- (0.25, 0.75);
					\btoken{0.5, 0.25}{west}{\alpha};
					\btoken[2px]{0.25, 0.5}{west}{\alpha};
\end{tikzpicture} \ ,
				\  \beta_A = \ 
				\begin{tikzpicture}[anchorbase]
					\draw[bmod] (0, 0) -- (0.25, 0.5) -- (0.25, 0.75);
					\draw[bmod] (0.25, 0.5) -- (0.5, 0);
					\draw[bmod] (0.75, 0) -- (0.75, 0.75);
					\btoken{0.25, 0.5}{west}{\beta};
				\end{tikzpicture} 
			\end{equation*}
			and
			\begin{equation*}
				\mu_A= \ 
				\begin{tikzpicture}[anchorbase]
					\draw[bmod] (0, 0) -- (0, 0.5) -- (0.25, .75)-- (0.25, 1);
					\draw[bmod] (0.25, 0) -- (0.25, 0.25)-- (0.75, 0.75)--(0.75,1);
					\draw[bmod] (0.75, 0) -- (0.75, 0.25)--(0.25,.75)--(.25,1);
					\draw[bmod] (1, 0) -- (1, 0.5)--(.75,.75);
					\btoken{0.25, 0.75}{east}{\mu};
					\btoken{0.75, 0.75}{west}{m};
				\end{tikzpicture} \ + \ 
				\begin{tikzpicture}[anchorbase]
					\draw[bmod] (0, 0) -- (0.5, 0.5)--(0.25,0.75)--(0.25,1);
					\draw[bmod] (0.25, 0) -- (0, 0.25)--(0,0.5)--(0.25,0.75)--(0.25,1);
					\draw[bmod] (0.5, 0) -- (0, 0.75)--(0,1);
					\draw[bmod] (0.75, 0) -- (0.75, 0.25)--(0.5,0.5);
					\btoken{0.25, 0.75}{west}{m};
					\btoken{0.5, 0.5}{west}{\alpha};
				\end{tikzpicture} \ - \ 
				\begin{tikzpicture}[anchorbase]
					\draw[bmod] (0, 0) -- (0.5, 0.25)--(0.25,0.5)--(0.5,0.75)--(0.5,1);
					\draw[bmod] (0.5, 0) -- (0, 0.25)--(0.25,0.5);
					\draw[bmod] (0.75, 0) -- (0.75, 0.5)--(0.5,0.75);
					\draw[bmod] (-0.25, 0) -- (-0.25, 1);
					\btoken{0.5, 0.75}{west}{m};
					\btoken{0.25, 0.5}{east}{\alpha};
				\end{tikzpicture}\ . 
			\end{equation*}
\end{prop}
\begin{proof}
The proof is a direct verification of the defining axioms of a
Lie-Rinehart monoid. More precisely, one checks that $\mu_A$ defines a Lie monoid structure on $L\otimes A$, that $\beta_A$ equips $L\otimes A$ with a left $A$-module structure, and that the anchor $\alpha_A$ satisfies the derivation property and the required compatibility conditions with
$\mu_A$ and $\beta_A$.

Each of these identities follows from the corresponding identity for $(A,m,\eta,L,\mu,\alpha,\beta),$
together with the associativity and commutativity of the monoid $A$. In the language of graphical calculus, the required diagrams are obtained by successively applying the defining diagrams of a Lie-Rinehart monoid and then using the coherence of the ambient monoidal category. Hence $(A,m,\eta,L\otimes A,\mu_A,\alpha_A,\beta_A)$ is a Lie-Rinehart monoid.
\end{proof}

The crossed product construction in Proposition~\ref{mainofanchor} can be iterated. Starting from a Lie-Rinehart monoid
$(A,m,\eta,L,\mu,\alpha,\beta),$
we define a sequence of objects and morphisms recursively as follows:		
\begin{align*}
			&L_A^{[0]}=L, \ \  \mu_A^{[0]}=\mu, \ \  \alpha_A^{[0]}=\alpha, \ \ \beta_A^{[0]}=\beta,\\
			&L_A^{[r]}=L_A^{[r-1]}\otimes A, \ \  \mu_A^{[r]}=\left(\mu_A^{[r-1]}\right)_{A}, \ \  \alpha_A^{[r]}=\left(\alpha_A^{[r-1]}\right)_{A}, \ \ \beta_A^{[r]}=\left(\beta_A^{[r-1]}  \right)_A,\ \forall r \geq 1. 
\end{align*}
Here $(\mu_A^{[r-1]})_A$, $(\alpha_A^{[r-1]})_A$, and $(\beta_A^{[r-1]})_A$ denote
the  morphisms obtained from the crossed product construction applied to the Lie-Rinehart monoid
$(A,m,\eta,L_A^{[r-1]},\mu_A^{[r-1]},\alpha_A^{[r-1]},\beta_A^{[r-1]}).$

\begin{coro}
For every $r\in \bbZ_{\geq 0}$, the tuple $(A,m,\eta,L_A^{[r]},\mu_A^{[r]}, \alpha_A^{[r]},\beta_A^{[r]})$ is a  Lie-Rinehart monoid.
\end{coro}

We now introduce morphisms and weak modules for Lie-Rinehart monoids. In what follows, all Lie-Rinehart monoids are assumed to have the same underlying commutative monoid $A$ unless otherwise stated.		
\begin{defi}
Let $(A,L,\alpha)$ and $(A,L',\alpha')$ be Lie-Rinehart monoids in $\calC$. 
A morphism $f:L\rightarrow L'$ is called a \emph{ homomorphism} of Lie-Rinehart monoids if $f$ is a homomorphism of Lie monoids and is compatible with the anchors, namely
\begin{equation*}
				\begin{tikzpicture}[anchorbase]
					\draw[bmod] (0, 0) -- (0, 0.25) -- (0.25, 0.5)--(0.25,0.75);
					\draw[bmod] (0.25, 0.5) -- (0.5, 0.25)--(0.5,0);
					\btoken{0.25, 0.5}{west}{\alpha'};
					\btoken{0, 0.25}{east}{f};
				\end{tikzpicture} \ = \ 
				\begin{tikzpicture}[anchorbase]
					\draw[bmod] (0, 0) -- (0, 0.25) -- (0.25, 0.5)--(0.25,0.75);
					\draw[bmod] (0.25, 0.5) -- (0.5, 0.25)--(0.5,0);
					\btoken{0.25, 0.5}{west}{\alpha};
				\end{tikzpicture} \ .
\end{equation*}
\end{defi}
		
\begin{defi} 
The triple $(M; \rho,\gamma)$ is called a \emph{weak module} over  $(A,m,\eta,L,\mu,\alpha,\beta)$ if it satisfies the following conditions:
\begin{enumerate}
\item [(1)] the pair $(M,\rho)$ is an $A$-module;
				
\item [(2)] the pair $(M,\gamma)$ is an $L$-module;
				
\item [(3)] the following compatible condition is satisfied:
\begin{equation} \label{compatible2}
					\begin{tikzpicture}[anchorbase]
						\draw[bmod] (0.25, 0) -- (0.5, 0.25);
						\draw[bmod] (0, 0) -- (0, 0.25) -- (0.25, 0.5);
						\draw[bmod] (0.75, 0) -- (0.5, 0.25) -- (0.25, 0.5) -- (0.25, 0.75);
						\btoken{0.5, 0.25}{west}{\rho};
						\btoken{0.25, 0.5}{west}{\gamma};
					\end{tikzpicture}
					\ =\
					\begin{tikzpicture}[anchorbase]
						\draw[bmod] (0.5, 0) -- (0.25, 0.25);
						\draw[bmod] (0, 0) -- (0.25, 0.25) -- (0.5, 0.5);
						\draw[bmod] (0.75, 0) -- (0.75, 0.25) -- (0.5, 0.5) -- (0.5, 0.75);
						\btoken[2px]{0.5, 0.5}{west}{\rho};
						\btoken[2px]{0.25, 0.25}{east}{\alpha};
					\end{tikzpicture}
					\ +\
					\begin{tikzpicture}[anchorbase]
						\draw[bmod] (0, -0.25) -- (0.25, 0) -- (0.5, 0.25);
						\draw[bmod] (0.25, -0.25) -- (0, 0) -- (0, 0.25) -- (0.25, 0.5);
						\draw[bmod] (0.75, -0.25) -- (0.75, 0) -- (0.5, 0.25) -- (0.25, 0.5) -- (0.25, 0.75);
						\btoken{0.5, 0.25}{west}{\gamma};
						\btoken[2px]{0.25, 0.5}{west}{\rho};
					\end{tikzpicture}\ .
				\end{equation}   
\end{enumerate}
\end{defi}
\begin{rmk}
The adjective ``weak'' indicates that no $A$-linearity condition is imposed on the $L$-action. More precisely, in the definition above, the morphism $\gamma:L\otimes M\rightarrow M$
is required to define an $L$-module structure on $M$ and to satisfy the Leibniz-type compatibility condition \eqref{compatible2}. However, we do not require $\gamma$ to be $A$-linear with respect to the $A$-module structures on $L$ and $M$. Thus a weak module is weaker than the usual notion of a
Lie-Rinehart module, where the $L$-action is typically required to be compatible with the $A$-module structures in an $A$-linear sense.
\end{rmk}

With the notation introduced in the definition of a Lie monoid,  relations \eqref{derivation} and \eqref{compatible1} imply that the triples $(A;m,\alpha)$ and $(L;\beta,\mu)$ are weak modules over $(A,L,\alpha)$. 
Given two weak modules $(M;\rho,\gamma)$ and $(M';\rho',\gamma')$ of the Lie-Rinehart monoid $(A,L,\alpha)$, a morphism $\phi:M\rightarrow M'$ is called a \emph{homomorphism of Lie-Rinehart weak modules} if $\phi$ is both a homomorphism of $A$-modules and $L$-modules. The composition of homomorphisms of the Lie-Rinehart weak modules is again a homomorphism of Lie-Rinehart weak modules. Consequently, all weak modules over the Lie-Rinehart monoid $(A,L,\alpha)$ form a category, denoted by ${\bf WMod}^{\calC}_A{\big(L\big)}$. 
		
Similar to Proposition~\ref{mainofanchor}, the crossed product construction can also be applied to weak modules. 
\begin{prop}\label{prop:cross}
Let $(M;\rho,\gamma)$ be a weak module over $(A,m,\eta,L, \mu, \alpha, \beta)$. Then the triple $(M;\rho,\gamma_A)$ forms a weak module over $(A,m,\eta,L\otimes A, \mu_A, \alpha_{A}, \beta_A)$,
where 
\begin{equation*}
\gamma_A \ = 
\begin{tikzpicture}[anchorbase]
\draw[bmod] (0, 0) -- (0.5, 0.5) -- (0.25, 0.75)--(0.25,1);
\draw[bmod] (0.5, 0) -- (0, 0.5)--(0.25,0.75);
\draw[bmod] (0.75, 0) -- (0.75, 0.25)--(0.25,0.75);
\btoken{0.25, 0.75}{east}{\rho};
\btoken{0.5, 0.5}{west}{\gamma};
\end{tikzpicture} \ . 
\end{equation*}
\end{prop}
The weak module $(M;\rho,\gamma_A)$ will be called the crossed module associated with $(M;\rho,\gamma)$ over
$(A,m,\eta,L\otimes A,\mu_A,\alpha_A,\beta_A).$ We now iterate the above construction. For any weak module $(M;\rho,\gamma) \in {\bf WMod}^{\calC}_A{\big(L\big)},$
define recursively 
\[\gamma_A^{[0]}=\gamma, \ \ \gamma_A^{[r]}=(\gamma_A^{[r-1]})_A \ \ \text{for all} \  r\in \bbZ_{\geq 1}.    \]
Here $(\gamma_A^{[r-1]})_A$ denotes the action obtained from
Proposition~\ref{prop:cross}, applied to the weak module
$(M;\rho,\gamma_A^{[r-1]})$ over the Lie-Rinehart monoid $
(A,m,\eta,L_A^{[r-1]},\mu_A^{[r-1]},\alpha_A^{[r-1]},\beta_A^{[r-1]}). $
Thus Proposition~\ref{prop:cross} produces, step by step, a sequence of weak modules over the iterated Lie-Rinehart monoids constructed above.
\begin{coro}\label{newweakmodule}
Let $(M;\rho,\gamma)$ be a weak module over $(A,m,\eta,L,\mu,\alpha,\beta)$. 
Then, the triple $(M;\rho,\gamma_A^{[r]})$ forms a weak module over $(A,m,\eta,L_A^{[r]},\mu_A^{[r]},\alpha_A^{[r]},\beta_A^{[r]})$ for every $r\in \bbZ_{\geq 0}.$
\end{coro}

The preceding corollary shows that the assignment
$(M;\rho,\gamma)\longmapsto (M;\rho,\gamma_A^{[r]})$
is compatible with morphisms of weak modules. Hence, for each
$r\in \mathbb{Z}_{\geq 0}$, we obtain the following functor.		
\begin{coro}\label{weakfunctor}
There is a functor 
\begin{align*}
				\mathcal{F}_r:  {\bf WMod}^{\calC}_A{\big(L\big)} &\to {\bf WMod}^{\calC}_A{\big(L_A^{[r]}\big)}\\
				(M;\rho,\gamma) &\mapsto \left(M; \rho, \gamma_A^{[r]}\right),\\
				f:M\to M' &\mapsto f: M \to M',
\end{align*}
where $(M';\rho',\gamma') \in $ ${\bf WMod}^{\calC}_A{\big(L\big)}$. 
\end{coro}

\subsection{Universal enveloping monoids}

In this subsection, we explain how weak modules over a Lie-Rinehart monoid may be realized
as modules over a suitable weak universal enveloping monoid. Throughout this
subsection, we assume that $\mathcal C$ is an additive symmetric monoidal
category admitting tensor monoids and the quotient monoids appearing below.

Let ${\bf Mon}(\mathcal C)$ denote the category whose objects are monoids in $\mathcal C$, and whose morphisms are morphisms in $\mathcal C$ preserving
the multiplication and the unit.
\begin{defi}
Assume that $\mathcal C$ is an additive symmetric monoidal category admitting tensor monoids and the coequalizers appearing below in
${\bf Mon}(\mathcal C)$. 
Let $(A,m,\eta,L,\mu,\alpha,\beta)$
be a Lie-Rinehart monoid in $\mathcal C$. 
Set $X=A\oplus L$ and let $(T(X), m_T, \eta_T)$ be the tensor monoid generated by $X$. 
Denote by
\[
\iota_A:A\rightarrow T(X) \ \text{and} \ 
\iota_L:L\rightarrow T(X)
\]
the canonical morphisms.
Define
\[
R_{\mathrm w}=(A\otimes A)\oplus(L\otimes L)\oplus
(L\otimes A).
\]
We define two morphisms
\[
r_{\mathrm w},s_{\mathrm w}:R_{\mathrm w}\rightarrow T(X)
\]
by specifying their restrictions to the three direct summands.
On $A\otimes A$, set
\[r_A=m_T\circ(\iota_A\otimes\iota_A) \ \text{and} \  s_A=\iota_A\circ m.
\]
On $L\otimes L$, set
\[r_L=m_T\circ(\iota_L\otimes\iota_L) \ \text{and} \ s_L=
m_T\circ(\iota_L\otimes\iota_L)\circ c_{L,L}+\iota_L\circ\mu.
\]
On $L\otimes A$, set
\[r_{\alpha}=m_T\circ(\iota_L\otimes\iota_A) \ \text{and} \ s_{\alpha}=
m_T\circ(\iota_A\otimes\iota_L)\circ c_{L,A}+\iota_A\circ\alpha.
\]
Thus
\[
r_{\mathrm w}=r_A\oplus r_L\oplus r_{\alpha} \ \text{and} \ 
s_{\mathrm w}=s_A\oplus s_L\oplus s_{\alpha}.
\]
By the universal property of the tensor monoid, these morphisms induce
monoid morphisms
\[
\widetilde r_{\mathrm w},\widetilde s_{\mathrm w}:
T(R_{\mathrm w})\rightarrow T(X).
\]
The \emph{weak universal enveloping monoid} of $(A,L,\alpha)$ is defined to
be the coequalizer in ${\bf Mon}(\mathcal C)$:
\[
U^{\mathrm w}_{\mathcal C}(A,L)
=
\operatorname{Coeq}_{{\bf Mon}(\mathcal C)}
\left(
T(R_{\mathrm w})
\overset{\widetilde r_{\mathrm w}}{\underset{\widetilde s_{\mathrm w}}{\rightrightarrows}}T(X)\right).
\]
\end{defi}

\begin{rmk}
Let $B$ be a monoid in $\mathcal C$. To give a monoid morphism $U^{\mathrm w}_{\mathcal C}(A,L)\rightarrow B$
is equivalent to giving morphisms
\[
f_A:A\rightarrow B \ \text{and} \ 
f_L:L\rightarrow B
\]
such that
\[
m_B\circ(f_A\otimes f_A)=f_A\circ m,
\]
\[
m_B\circ(f_L\otimes f_L)=
m_B\circ(f_L\otimes f_L)\circ c_{L,L}
+f_L\circ \mu,
\]
and
\[
m_B\circ(f_L\otimes f_A)=
m_B\circ(f_A\otimes f_L)\circ c_{L,A}
+f_A\circ \alpha.
\]
\end{rmk}

\begin{theor}
Assume that the weak universal enveloping monoid $U^{\mathrm w}_{\mathcal C}(A,L)$
exists. Then there is an equivalence of categories
\[
{\bf Mod}_{\mathcal C}\big(U^{\mathrm w}_{\mathcal C}(A,L)\big)
\simeq
{\bf WMod}^{\mathcal C}_A(L).
\]
\end{theor}
\begin{proof}
Let $N$ be a left $U^{\mathrm w}_{\mathcal C}(A,L)$-module. By restriction
along the canonical morphisms
\[
A\rightarrow U^{\mathrm w}_{\mathcal C}(A,L) \ \text{and} \ 
L\rightarrow U^{\mathrm w}_{\mathcal C}(A,L),
\]
the object $N$ becomes both a left $A$-module and a left $L$-module. The
coequalizer relations imply precisely that $N$ is a weak module over $(A,L,\alpha)$.

Conversely, if $(M;\rho,\gamma)\in {\bf WMod}^{\mathcal C}_A(L),$
then the $A$-action $\rho$ and the $L$-action $\gamma$ define, by the
universal property of the tensor monoid, an action of $T(A\oplus L)$ on $M$.
The $A$-module relation, the $L$-module relation, and the weak compatibility
condition imply that the two induced actions of $T(R_{\mathrm w})$ on $M$
coincide. Therefore the action of $T(A\oplus L)$ factors uniquely through the coequalizer
$U^{\mathrm w}_{\mathcal C}(A,L).$
This construction is functorial and gives an equivalence of categories.
\end{proof}

\subsection{Left module categories over monoidal categories}
\begin{defi}
Let $(A,m,\eta,L,\mu,\alpha,\beta)$ be a  Lie-Rinehart monoid in $\calC$. 
Let $(\mathfrak{L},\mathfrak{u})$ be another Lie monoid in $\calC$. A morphism $\mathcal{H}: L \rightarrow \mathfrak{L}\otimes A$ is called a \emph{crossed homomorphism} if it satisfies the following axiom:
\begin{equation*}
				\begin{tikzpicture}[anchorbase]
					\draw[bmod] (0, 0) -- (0.25, 0.5) -- (0.25, 1)--(0,1.5);
					\draw[bmod] (0.5, 0) -- (0.25, 0.5);
					\draw[bmod] (0.5, 1.5) -- (0.25, 1);
					\btoken{0.25, 1}{west}{\mathcal{H}};
					\btoken{0.25, 0.5}{west}{\mu};
				\end{tikzpicture} \ = \ 
				\begin{tikzpicture}[anchorbase]
					\draw[bmod] (0, 0) -- (0, 1) -- (0.25, 1.25)--(0.25,1.5);
					\draw[bmod] (0.25, 0) --(0.25,0.25)-- (-0.25, 0.75) -- (-0.25, 1.5);
					\draw[bmod] (0.25, 0.25) -- (0.5, 0.5) -- (0.5, 1)--(0.25,1.25);
					\btoken{0.25, 1.25}{west}{\alpha};
					\btoken{0.25, 0.25}{west}{\mathcal{H}};
				\end{tikzpicture}\ - \ 
				\begin{tikzpicture}[anchorbase]
					\draw[bmod] (0, 0) -- (0, 0.25) -- (-0.25, 0.5)--(-0.25,1.5);
					\draw[bmod] (0, 0.25) -- (0.5, 0.75) -- (0.5, 1)--(0.25,1.25)--(0.25,1.5);
					\draw[bmod] (0.5, 0) -- (0.5, 0.25) -- (0, 0.75)--(0,1)--(0.25,1.25);
					\btoken{0.25, 1.25}{west}{\alpha};
					\btoken{0, 0.25}{east}{\mathcal{H}};
				\end{tikzpicture} \ + \ 
				\begin{tikzpicture}[anchorbase]
					\draw[bmod] (0, 0) -- (0, 0.25) -- (-0.25, 0.5)--(-0.25,1)--(0,1.25)--(0,1.5);
					\draw[bmod] (0, 0.25) -- (0.75,1.25)--(0.75,1.5);
					\draw[bmod] (0.75, 0.25) -- (1, 0.5) -- (1, 1)--(0.75,1.25);
					\draw[bmod] (0.75, 0) -- (0.75, 0.25) --(0,1.25);
					\btoken{0, 1.25}{east}{\mathfrak{u}};
					\btoken{0.75, 1.25}{west}{m};
					\btoken{0, 0.25}{east}{\mathcal{H}};
					\btoken{0.75, 0.25}{west}{\mathcal{H}};
				\end{tikzpicture} \ .
\end{equation*}
\end{defi}
		
Crossed homomorphisms provide a mechanism for producing new weak modules from existing ones. More precisely, we shall prove that a crossed homomorphism $\mathcal{H}$, together with a weak module over $(A,L,\alpha)$ and a module over the Lie monoid $\mathfrak{L}$, canonically induces another weak module.
		
\begin{lem}\label{main1}
Let $(A,m,\eta,L,\mu,\alpha,\beta)$ be a Lie-Rinehart monoid in $\calC$, and let $(M;\rho,\gamma)\in {\bf WMod}^{\mathcal C}_A(L).$
Suppose $(\mathfrak L,\mathfrak u)$ is another Lie monoid in $\mathcal C$, $(V,\psi)$ an $\mathfrak L$-module, and 
$\mathcal H:L\longrightarrow \mathfrak L\otimes A$ a crossed homomorphism. Define
\begin{equation}\label{newhom}
				\psi \vartriangleright\rho \ = \ 
				\begin{tikzpicture}[anchorbase]
					\draw[bmod] (0, 0) -- (0, 0.5) -- (0.5, 1)--(0.5,1.25);
					\draw[bmod] (0, 1.25) -- (0,1)--(0.5,0.5)--(0.5,0);
					\draw[bmod] (0.5, 1) -- (1, 0.5) -- (1, 0);
					\btoken{0.5, 1}{west}{\rho};
				\end{tikzpicture} \ \ \ \ \text{and} \ \ \ \
				\psi\vartriangleright \gamma \ =  \ 
				\begin{tikzpicture}[anchorbase]
					\draw[bmod] (0, 0) -- (0, 0.5) -- (0.5, 1)--(0.5,1.25);
					\draw[bmod] (0, 1.25) -- (0,1)--(0.5,0.5)--(0.5,0);
					\draw[bmod] (0.5, 1) -- (1, 0.5) -- (1, 0);
					\btoken{0.5, 1}{west}{\gamma};
				\end{tikzpicture} \ + \ 
				\begin{tikzpicture}[anchorbase]
					\draw[bmod] (0, 0) -- (0, 0.25) -- (-0.25, 0.5)--(-0.25,0.75)--(0,1)--(0,1.25);
					\draw[bmod] (0, 0.25) -- (0.75,1)--(0.75,1.25);
					\draw[bmod] (0.75, 1) -- (1, 0.75) -- (1, 0);
					\draw[bmod] (0.5, 0) -- (0.5, 0.5) --(0,1);
					\btoken{0, 1}{east}{\psi};
					\btoken{0.75, 1}{west}{\rho};
					\btoken{0, 0.25}{east}{\mathcal{H}};
				\end{tikzpicture} \ .
\end{equation}
Then we have  $(V\otimes M ; \psi \vartriangleright \rho , \psi \vartriangleright \gamma ) \in$ ${\bf WMod}^{\calC}_A{\big(L\big)}$.
\end{lem}
\begin{proof}
Thanks to Lemma \ref{tensorofamod}, the pair $(V\otimes M, \psi \vartriangleright \rho)$ forms an $A$-module. It remains to prove that $(V\otimes M, \psi \vartriangleright \gamma )$ is an $L$-module and the triple $(V\otimes M ; \psi \vartriangleright \rho , \psi \vartriangleright \gamma )$ satisfies the compatible condition \eqref{compatible2}. To see the first assertion, it suffices to verify \eqref{fig:ten items}.
\begin{equation}\label{fig:ten items}
				\begin{tikzpicture}[anchorbase]
					\draw[bmod] (0, 0) -- (0, .75) -- (.5, 1.25)--(.5,1.5);
					
					\draw[bmod] (.25, 0) -- (.25, .5) -- (.75, 1)--(.5,1.25);
					\draw[bmod] (.75, 0) -- (.75, .5) -- (0, 1.25)--(0, 1.5);
					\draw[bmod] (1, 0) -- (1, .75)--(.75,1);
					\btoken{.5, 1.25}{west}{\gamma};
					\btoken{.75, 1}{west}{\gamma};
				\end{tikzpicture} \ + \ 
				\begin{tikzpicture}[anchorbase]
					\draw[bmod] (0, 0) -- (0, .75) -- (-.25, 1)--(0,1.25)--(0,1.5);
					\draw[bmod] (0, .75) -- (0.5,1.25)--(.5,1.5);
					\draw[bmod] (.25, 0) -- (.25, .5) -- (.75, 1)--(.5,1.25);
					\draw[bmod] (.75, 0) -- (.75, .5) -- (0, 1.25);
					\draw[bmod] (1, 0) -- (1, .75)--(.75,1);
					\btoken{0, 1.25}{east}{\psi};
					\btoken{.5, 1.25}{west}{\rho};
					\btoken{.75, 1}{west}{\gamma};
					\btoken{0, .75}{east}{\mathcal{H}};
				\end{tikzpicture} \ + \  
				\begin{tikzpicture}[anchorbase]
					\draw[bmod] (0, 0) -- (0, .5) -- (.75, 1.25)--(.75,1.5);
					\draw[bmod] (.5,0) -- (0.5,.25)--(.25,.5)--(.5,.75)--(.5,1.5);
					\draw[bmod] (.5, 0.25) -- (.75, .5) -- (.5, .75);
					\draw[bmod] (.75, .5) -- (.75, .75) -- (1, 1);
					\draw[bmod] (1, 0) -- (1, .25)--(.75,.5);
					\draw[bmod] (1.25, 0) -- (1.25, .75)--(1,1)--(.75,1.25);
					\btoken{0.5, .75}{north}{\psi};
					\btoken{1, 1}{west}{\rho};
					\btoken{.75, 1.25}{west}{\gamma};
					\btoken{0.5, .25}{east}{\mathcal{H}};
				\end{tikzpicture} \ + \ 
				\begin{tikzpicture}[anchorbase]
					\draw[bmod] (0, 0) -- (0, .75) -- (-.25, 1)--(0,1.25)--(0,1.5);
					\draw[bmod] (0,0.75) -- (0.5,1.25)--(.5,1.5);
					\draw[bmod] (.5, 0) -- (.5, .25) -- (.25, .5)--(.5,.75)--(0,1.25);
					\draw[bmod] (.5, 0.25) -- (1, .75) -- (0.5, 1.25);
					\draw[bmod] (1,0)--(1,0.25)--(.5,.75);
					\draw[bmod] (1.25, 0) -- (1.25, .5)--(1,.75);
					\btoken{0, 1.25}{east}{\psi};
					\btoken{.5, 1.25}{west}{\rho};
					\btoken{1, .75}{west}{\rho};
					\btoken{0, .75}{east}{\mathcal{H}};
					\btoken{0.5, .75}{east}{\psi};
					\btoken{0.5, .25}{east}{\mathcal{H}};
				\end{tikzpicture}  \ - \ 
				\begin{tikzpicture}[anchorbase]
					\draw[bmod] (0, 0) -- (0, .5) -- (.25, .75)--(0.25,1)--(0.5,1.25)--(0.5,1.5);
					\draw[bmod] (0.5,0) -- (0.5,.5)--(.25,.75);
					\draw[bmod] (.75, 0) -- (.75, .75) -- (.25, 1.25)--(.25,1.5);
					\draw[bmod] (1, 0) -- (1, .75) -- (0.5, 1.25);
					\btoken{0.5, 1.25}{west}{\gamma};
					\btoken{.25, .75}{east}{\mu};
				\end{tikzpicture} \ - \ 
				\begin{tikzpicture}[anchorbase]
					\draw[bmod] (0, 0) -- (0, .25) -- (.25, .5)--(0.25,.75)--(0,1)--(0.25,1.25)--(0.25,1.5);
					\draw[bmod] (0.25,0.75) -- (0.75,1.25)--(.75,1.5);
					\draw[bmod] (.5, 0) --(.5,.25)-- (.25, .5) -- (.25, .75);
					\draw[bmod] (.75, 0) -- (.75, .75) -- (0.25, 1.25);
					\draw[bmod] (1, 0) -- (1, 1) -- (0.75, 1.25);
					\btoken{0.25, 1.25}{east}{\psi};
					\btoken{.25, .5}{east}{\mu};
					\btoken{.25, .75}{east}{\mathcal{H}};
					\btoken{.75, 1.25}{west}{\rho};
				\end{tikzpicture} 
\end{equation}
\[=\]
\begin{equation*}
				\begin{tikzpicture}[anchorbase]
					\draw[bmod] (0, 0) -- (0, .25) -- (.75, 1)--(.5,1.25)--(.5,1.5);
					\draw[bmod] (0.25, 0) -- (0.25,1)--(.5,1.25);
					\draw[bmod] (.75, 0) -- (.75, .5) -- (0, 1.25)--(0,1.5);
					\draw[bmod] (1, 0) -- (1, .75) -- (.75, 1);
					\btoken{.5, 1.25}{west}{\gamma};
					\btoken{1, .75}{west}{\gamma};
				\end{tikzpicture}\ + \ 
				\begin{tikzpicture}[anchorbase]
					\draw[bmod] (0, 0) -- (1, 1) -- (.75, 1.25)--(.75,1.5);
					\draw[bmod] (.25, 0) -- (.25,.75)--(0,1)--(.25,1.25)--(.25,1.5);
					\draw[bmod] (.25, .75) -- (.75, 1.25);
					\draw[bmod] (1,0) -- (1, 0.5) -- (0.25, 1.25);
					\draw[bmod] (1.25, 0) -- (1.25, .75) -- (1, 1);
					\btoken{.75, 1.25}{west}{\rho};
					\btoken{1, 1}{west}{\gamma};
					\btoken{0.25, 1.25}{east}{\psi};
					\btoken{0.25, .75}{east}{\mathcal{H}}; 
				\end{tikzpicture} \ + \ 
				\begin{tikzpicture}[anchorbase]
					\draw[bmod] (0.75,0)--(0, 0.25) -- (0, .5) -- (.75, 1.25)--(.75,1.5);
					\draw[bmod] (.5,0) -- (0.5,.25)--(.25,.5)--(.5,.75)--(.5,1.5);
					\draw[bmod] (.5, 0.25) -- (.75, .5) -- (.5, .75);
					\draw[bmod] (.75, .5) -- (.75, .75) -- (1, 1);
					\draw[bmod] (1, 0) -- (1, .25)--(.75,.5);
					\draw[bmod] (1.25, 0) -- (1.25, .75)--(1,1)--(.75,1.25);
					\btoken{0.5, .75}{north}{\psi};
					\btoken{1, 1}{west}{\rho};
					\btoken{.75, 1.25}{west}{\gamma};
					\btoken{0.5, .25}{west}{\mathcal{H}};
				\end{tikzpicture} \ + \ 
				\begin{tikzpicture}[anchorbase]
					\draw[bmod] (0.5, 0) -- (0.5,0.25) -- (0.25, .5)--(0.5,.75)--(0,1.25)--(0,1.5);
					\draw[bmod] (0.5, .25) -- (1,.75)--(.5,1.25);
					\draw[bmod] (.75, 0) -- (0, .25)--(0,.75)--(-.25,1)--(0,1.25)--(0,1.5);
					\draw[bmod] (0, .75) -- (0.5, 1.25)--(.5,1.5);
					\draw[bmod] (1, 0) -- (1, .25) -- (0.5, .75);
					\draw[bmod] (1.25, 0) -- (1.25, .5) -- (1, .75);
					\btoken{1, .75}{west}{\rho};
					\btoken{.5, .75}{west}{\psi};
					\btoken{0, 1.25}{east}{\psi};
					\btoken{0, 0.75}{east}{\mathcal{H}}; 
					\btoken{0.5, .25}{west}{\mathcal{H}}; 
					\btoken{.5, 1.25}{west}{\rho};
				\end{tikzpicture} \ .
\end{equation*}
Since $(M,\gamma)$ is an $L$-module, we obtain 
\begin{equation*}
				\begin{tikzpicture}[anchorbase]
					\draw[bmod] (0, 0) -- (0, .75) -- (.5, 1.25)--(.5,1.5);
					
					\draw[bmod] (.25, 0) -- (.25, .5) -- (.75, 1)--(.5,1.25);
					\draw[bmod] (.75, 0) -- (.75, .5) -- (0, 1.25)--(0, 1.5);
					\draw[bmod] (1, 0) -- (1, .75)--(.75,1);
					\btoken{.5, 1.25}{west}{\gamma};
					\btoken{.75, 1}{west}{\gamma};
				\end{tikzpicture} \ = \ 
				\begin{tikzpicture}[anchorbase]
					\draw[bmod] (0, 0) -- (0, .5) -- (.25, .75)--(0.25,1)--(0.5,1.25)--(0.5,1.5);
					\draw[bmod] (0.5,0) -- (0.5,.5)--(.25,.75);
					\draw[bmod] (.75, 0) -- (.75, .75) -- (.25, 1.25)--(.25,1.5);
					\draw[bmod] (1, 0) -- (1, .75) -- (0.5, 1.25);
					\btoken{0.5, 1.25}{west}{\gamma};
					\btoken{.25, .75}{east}{\mu};
				\end{tikzpicture} \ + \ 
				\begin{tikzpicture}[anchorbase]
					\draw[bmod] (0, 0) -- (0, .25) -- (.75, 1)--(.5,1.25)--(.5,1.5);
					\draw[bmod] (0.25, 0) -- (0.25,1)--(.5,1.25);
					\draw[bmod] (.75, 0) -- (.75, .5) -- (0, 1.25)--(0,1.5);
					\draw[bmod] (1, 0) -- (1, .75) -- (.75, 1);
					\btoken{.5, 1.25}{west}{\gamma};
					\btoken{1, .75}{west}{\gamma};
				\end{tikzpicture} \ .     
\end{equation*}
By the condition that $(M;\rho,\gamma)$ is a weak module, we have 
\begin{equation} \label{fig:three}
				\begin{tikzpicture}[anchorbase]
					\draw[bmod] (0, 0) -- (0, .5) -- (.75, 1.25)--(.75,1.5);
					\draw[bmod] (.5,0) -- (0.5,.25)--(.25,.5)--(.5,.75)--(.5,1.5);
					\draw[bmod] (.5, 0.25) -- (.75, .5) -- (.5, .75);
					\draw[bmod] (.75, .5) -- (.75, .75) -- (1, 1);
					\draw[bmod] (1, 0) -- (1, .25)--(.75,.5);
					\draw[bmod] (1.25, 0) -- (1.25, .75)--(1,1)--(.75,1.25);
					
					\btoken{0.5, .75}{north}{\psi};
					\btoken{1, 1}{west}{\rho};
					\btoken{.75, 1.25}{west}{\gamma};
					\btoken{0.5, .25}{east}{\mathcal{H}};
				\end{tikzpicture} \ = \ 
				\begin{tikzpicture}[anchorbase]
					\draw[bmod] (0, 0) -- (0, .75) -- (.5, 1.25)--(.75,1.25)--(1,1.5)--(1,1.75);
					\draw[bmod] (.5,0) -- (0.5,.25)--(.25,.5)--(.25,.75)--(.5,1)--(.5,1.75);
					\draw[bmod] (.5, 0.25) -- (.75, .5) -- (.75, 1.25) -- (1, 1.5)-- (1, 1.75);
					\draw[bmod] (1, 0) -- (1, .5) -- (.5, 1);
					\draw[bmod] (1.5, 0) -- (1.5, 1)--(1,1.5);
					\draw[bmod] (.75, 1.25) -- (1.25, 1.25);
					\btoken{0.5, 1}{north}{\psi};
					\btoken{1, 1.5}{west}{\rho};
					\btoken{1.25, 1.25}{west}{\gamma};
					\btoken{0.5, .25}{east}{\mathcal{H}};
				\end{tikzpicture} \ + \ 
				\begin{tikzpicture}[anchorbase]
					\draw[bmod] (0, 0) -- (0, .5) -- (1, 1.5)--(1,1.75);
					\draw[bmod] (.5,0) -- (0.5,.25)--(.25,.5)--(.5,.75)--(.5,1.75);
					\draw[bmod] (.5, 0.25) -- (1, .75) -- (1, 1) -- (.75, 1.25);
					\draw[bmod] (1, 0) -- (1, .25) -- (.5, .75);
					\draw[bmod] (1.25, 0) -- (1.25, 1.25)--(1,1.5);
					\btoken{0.75, 1.25}{west}{\alpha};
					\btoken{1, 1.5}{west}{\rho};
					\btoken{.5, .75}{west}{\psi};
					\btoken{0.5, .25}{east}{\mathcal{H}};
				\end{tikzpicture}    
\end{equation}
and
\begin{equation}\label{fig:nine}
				\begin{tikzpicture}[anchorbase]
					\draw[bmod] (0.75,0)--(0, 0.25) -- (0, .5) -- (.75, 1.25)--(.75,1.5);
					\draw[bmod] (.5,0) -- (0.5,.25)--(.25,.5)--(.5,.75)--(.5,1.5);
					\draw[bmod] (.5, 0.25) -- (.75, .5) -- (.5, .75);
					\draw[bmod] (.75, .5) -- (.75, .75) -- (1, 1);
					\draw[bmod] (1, 0) -- (1, .25)--(.75,.5);
					\draw[bmod] (1.25, 0) -- (1.25, .75)--(1,1)--(.75,1.25);
					\btoken{0.5, .75}{north}{\psi};
					\btoken{1, 1}{west}{\rho};
					\btoken{.75, 1.25}{west}{\gamma};
					\btoken{0.5, .25}{west}{\mathcal{H}};
				\end{tikzpicture} \ = \ 
				\begin{tikzpicture}[anchorbase]
					\draw[bmod] (0.75, 0) -- (0, .25)--(0,.75) -- (.5, 1.25)--(.75,1.25)--(1,1.5)--(1,1.75);
					\draw[bmod] (.5,0) -- (0.5,.25)--(.25,.5)--(.25,.75)--(.5,1)--(.5,1.75);
					\draw[bmod] (.5, 0.25) -- (.75, .5) -- (.75, 1.25) -- (1, 1.5)-- (1, 1.75);
					\draw[bmod] (1, 0) -- (1, .5) -- (.5, 1);
					\draw[bmod] (1.5, 0) -- (1.5, 1)--(1,1.5);
					\draw[bmod] (.75, 1.25) -- (1.25, 1.25);
					\btoken{0.5, 1}{north}{\psi};
					\btoken{1, 1.5}{west}{\rho};
					\btoken{1.25, 1.25}{west}{\gamma};
					\btoken{0.5, .25}{west}{\mathcal{H}};
				\end{tikzpicture} \ + \ 
				\begin{tikzpicture}[anchorbase]
					\draw[bmod] (.75,0)--(0, 0.25) -- (0, .5) -- (1, 1.5)--(1,1.75);
					\draw[bmod] (.5,0) -- (0.5,.25)--(.25,.5)--(.5,.75)--(.5,1.75);
					\draw[bmod] (.5, 0.25) -- (1, .75) -- (1, 1) -- (.75, 1.25);
					\draw[bmod] (1, 0) -- (1, .25) -- (.5, .75);
					\draw[bmod] (1.25, 0) -- (1.25, 1.25)--(1,1.5);
					\btoken{0.75, 1.25}{west}{\alpha};
					\btoken{1, 1.5}{west}{\rho};
					\btoken{.5, .75}{west}{\psi};
					\btoken{0.5, .25}{west}{\mathcal{H}};
				\end{tikzpicture}  \ . 
\end{equation}
For the braiding $\big(c_{L,L\otimes V}, \id_L, (\psi\otimes \id_A)\circ (\id_{\mathfrak{L}}\otimes c_{A,V})\circ (\mathcal{H}\otimes \id_V)\big),$
its natural isomorphisms give us 
\begin{equation*}
				\begin{tikzpicture}[anchorbase]
					\draw[bmod] (0, 0) -- (0, .75) -- (.5, 1.25)--(.75,1.25)--(1,1.5)--(1,1.75);
					\draw[bmod] (.5,0) -- (0.5,.25)--(.25,.5)--(.25,.75)--(.5,1)--(.5,1.75);
					\draw[bmod] (.5, 0.25) -- (.75, .5) -- (.75, 1.25) -- (1, 1.5)-- (1, 1.75);
					\draw[bmod] (1, 0) -- (1, .5) -- (.5, 1);
					\draw[bmod] (1.5, 0) -- (1.5, 1)--(1,1.5);
					\draw[bmod] (.75, 1.25) -- (1.25, 1.25);
					\btoken{0.5, 1}{north}{\psi};
					\btoken{1, 1.5}{west}{\rho};
					\btoken{1.25, 1.25}{west}{\gamma};
					\btoken{0.5, .25}{east}{\mathcal{H}};
				\end{tikzpicture} \ = \ 
				\begin{tikzpicture}[anchorbase]
					\draw[bmod] (0, 0) -- (1, 1) -- (.75, 1.25)--(.75,1.5);
					\draw[bmod] (.25, 0) -- (.25,.75)--(0,1)--(.25,1.25)--(.25,1.5);
					\draw[bmod] (.25, .75) -- (.75, 1.25);
					\draw[bmod] (1,0) -- (1, 0.5) -- (0.25, 1.25);
					\draw[bmod] (1.25, 0) -- (1.25, .75) -- (1, 1);
					\btoken{.75, 1.25}{west}{\rho};
					\btoken{1, 1}{west}{\gamma};
					\btoken{0.25, 1.25}{east}{\psi};
					\btoken{0.25, .75}{east}{\mathcal{H}}; 
				\end{tikzpicture}  \ \ \text{and} \ \ 
				\begin{tikzpicture}[anchorbase]
					\draw[bmod] (0.75, 0) -- (0, .25)--(0,.75) -- (.5, 1.25)--(.75,1.25)--(1,1.5)--(1,1.75);
					\draw[bmod] (.5,0) -- (0.5,.25)--(.25,.5)--(.25,.75)--(.5,1)--(.5,1.75);
					\draw[bmod] (.5, 0.25) -- (.75, .5) -- (.75, 1.25) -- (1, 1.5)-- (1, 1.75);
					\draw[bmod] (1, 0) -- (1, .5) -- (.5, 1);
					\draw[bmod] (1.5, 0) -- (1.5, 1)--(1,1.5);
					\draw[bmod] (.75, 1.25) -- (1.25, 1.25);
					\btoken{0.5, 1}{north}{\psi};
					\btoken{1, 1.5}{west}{\rho};
					\btoken{1.25, 1.25}{west}{\gamma};
					\btoken{0.5, .25}{west}{\mathcal{H}};
				\end{tikzpicture} \ = \ 
				\begin{tikzpicture}[anchorbase]
					\draw[bmod] (0, 0) -- (0, .75) -- (-.25, 1)--(0,1.25)--(0,1.5);
					\draw[bmod] (0, .75) -- (0.5,1.25)--(.5,1.5);
					\draw[bmod] (.25, 0) -- (.25, 0.25) -- (-.25,.25)--(-.25,.5)--(.25,.5)-- (.75, 1)--(.5,1.25);
					\draw[bmod] (.75, 0) -- (.75, .5) -- (0, 1.25);
					\draw[bmod] (1, 0) -- (1, .75)--(.75,1);
					\btoken{0, 1.25}{east}{\psi};
					\btoken{.5, 1.25}{west}{\rho};
					\btoken{.75, 1}{west}{\gamma};
					\btoken{0, .75}{east}{\mathcal{H}};
				\end{tikzpicture}  \ = \ 
				\begin{tikzpicture}[anchorbase]
					\draw[bmod] (0, 0) -- (0, .75) -- (-.25, 1)--(0,1.25)--(0,1.5);
					\draw[bmod] (0, .75) -- (0.5,1.25)--(.5,1.5);
					\draw[bmod] (.25, 0) -- (.25, .5) -- (.75, 1)--(.5,1.25);
					\draw[bmod] (.75, 0) -- (.75, .5) -- (0, 1.25);
					\draw[bmod] (1, 0) -- (1, .75)--(.75,1);
					\btoken{0, 1.25}{east}{\psi};
					\btoken{.5, 1.25}{west}{\rho};
					\btoken{.75, 1}{west}{\gamma};
					\btoken{0, .75}{east}{\mathcal{H}};
				\end{tikzpicture} \ .
\end{equation*}		
Using the naturality of the braidings $(c_{L,\mathfrak{L}\otimes V}, \id_L,\psi)$ and $(c_{L\otimes A,V},\alpha,\id_V)$, we get 
\begin{equation}\label{fig:three b}
				\begin{tikzpicture}[anchorbase]
					\draw[bmod] (0, 0) -- (0, .5) -- (1, 1.5)--(1,1.75);
					\draw[bmod] (.5,0) -- (0.5,.25)--(.25,.5)--(.5,.75)--(.5,1.75);
					\draw[bmod] (.5, 0.25) -- (1, .75) -- (1, 1) -- (.75, 1.25);
					\draw[bmod] (1, 0) -- (1, .25) -- (.5, .75);
					\draw[bmod] (1.25, 0) -- (1.25, 1.25)--(1,1.5);
					\btoken{0.75, 1.25}{west}{\alpha};
					\btoken{1, 1.5}{west}{\rho};
					\btoken{.5, .75}{west}{\psi};
					\btoken{0.5, .25}{east}{\mathcal{H}};
				\end{tikzpicture} \ = \ 
				\begin{tikzpicture}[anchorbase]
					\draw[bmod] (0, 0) -- (0, .25) -- (1.25, 1.5)--(1.25,1.75);
					\draw[bmod] (.5,0) -- (0.5,.25)--(.25,.5)--(.25,1.25)--(.5,1.5)--(.5,1.75);
					\draw[bmod] (.5, 0.25) -- (1.25, 1) -- (1, 1.25);
					\draw[bmod] (.75, 0) -- (.75, 1.25) -- (.5, 1.5);
					\draw[bmod] (1.5, 0) -- (1.5, 1.25)--(1.25,1.5);
					\btoken{1, 1.25}{west}{\alpha};
					\btoken{1.25, 1.5}{west}{\rho};
					\btoken{.5, 1.5}{west}{\psi};
					\btoken{0.5, .25}{east}{\mathcal{H}};
				\end{tikzpicture} \ + \ 
				\begin{tikzpicture}[anchorbase]
					\draw[bmod] (0, 0) -- (0, .25) -- (1.25, 1.5)--(1.25,1.75);
					\draw[bmod] (.5,0) -- (0.5,.25)--(.25,.5)--(.25,1.25)--(.5,1.5)--(.5,1.75);
					\draw[bmod] (.5, 0.25) -- (1, .75) -- (.75, 1);
					\draw[bmod] (1.25, 0) -- (1.25, .75) -- (.5, 1.5);
					\draw[bmod] (1.5, 0) -- (1.5, 1.25)--(1.25,1.5);
					\btoken{.75, 1}{east}{\alpha};
					\btoken{1.25, 1.5}{west}{\rho};
					\btoken{.5, 1.5}{west}{\psi};
					\btoken{0.5, .25}{east}{\mathcal{H}};
				\end{tikzpicture} \ .
\end{equation}
Applying these together with Lemma \ref{identityofbraid}, we have 
\begin{equation}\label{fig:nine b}
				\begin{tikzpicture}[anchorbase]
					\draw[bmod] (.75,0)--(0, 0.25) -- (0, .5) -- (1, 1.5)--(1,1.75);
					\draw[bmod] (.5,0) -- (0.5,.25)--(.25,.5)--(.5,.75)--(.5,1.75);
					\draw[bmod] (.5, 0.25) -- (1, .75) -- (1, 1) -- (.75, 1.25);
					\draw[bmod] (1, 0) -- (1, .25) -- (.5, .75);
					\draw[bmod] (1.25, 0) -- (1.25, 1.25)--(1,1.5);
					\btoken{0.75, 1.25}{west}{\alpha};
					\btoken{1, 1.5}{west}{\rho};
					\btoken{.5, .75}{west}{\psi};
					\btoken{0.5, .25}{west}{\mathcal{H}};
				\end{tikzpicture} \ = \ 
				\begin{tikzpicture}[anchorbase]
					\draw[bmod] (0,0)--(1, 1) -- (0.75, 1.25) -- (1, 1.5)--(1,1.75);
					\draw[bmod] (.5,0.5) -- (0.25,.75)--(.25,1)--(.5,1.25)--(.5,1.75);
					\draw[bmod] (.5, 0) -- (0, .5) -- (.75, 1.25);
					\draw[bmod] (1, 0) -- (1, .75) -- (.5, 1.25);
					\draw[bmod] (1.25, 0) -- (1.25, 1.25)--(1,1.5);
					\btoken{0.75, 1.25}{west}{\alpha};
					\btoken{1, 1.5}{west}{\rho};
					\btoken{.5, 1.25}{east}{\psi};
					\btoken{0.5, .5}{west}{\mathcal{H}};
				\end{tikzpicture} \ = \ 
				\begin{tikzpicture}[anchorbase]
					\draw[bmod] (0,0)--(.75, .75) -- (0.5, 1) -- (1, 1.5)--(1,1.75);
					\draw[bmod] (.5,0.5) -- (0.25,.75)--(.25,1)--(.5,1.25)--(.5,1.75);
					\draw[bmod] (.5, 0) -- (0, .5) -- (.75, 1.25);
					\draw[bmod] (1, 0) -- (1, .75) -- (.5, 1.25);
					\draw[bmod] (1.25, 0) -- (1.25, 1.25)--(1,1.5);
					\btoken{0.5, 1}{north}{\alpha};
					\btoken{1, 1.5}{west}{\rho};
					\btoken{.5, 1.25}{east}{\psi};
					\btoken{0.5, .5}{west}{\mathcal{H}};
				\end{tikzpicture} \ = \ 
				\begin{tikzpicture}[anchorbase]
					\draw[bmod] (0.25,0)--(.25, .25) -- (0, .5) -- (0, 1)--(.5,1.5)--(.5,1.75);
					\draw[bmod] (.25,0.25) -- (0.75,.75)--(.5,1)--(1,1.5)--(1,1.75);
					\draw[bmod] (.75, 0) -- (.75, .25) -- (.25, .75)--(.5,1);
					\draw[bmod] (1, 0) -- (1, 1) -- (.5, 1.5);
					\draw[bmod] (1.25, 0) -- (1.25, 1.25)--(1,1.5);
					\btoken{0.5, 1}{north}{\alpha};
					\btoken{1, 1.5}{west}{\rho};
					\btoken{.5, 1.5}{east}{\psi};
					\btoken{0.25, .25}{east}{\mathcal{H}};
				\end{tikzpicture} \ .
\end{equation}
Since $(M,\rho)$ is an $A$-module, the naturality of braidings  $(c_{A,\mathfrak{L}\otimes V}, \id_A,\psi)$ and $(c_{A\otimes A,V},m,\id_V)$ give  
\begin{equation}\label{fig:four}
				\begin{tikzpicture}[anchorbase]
					\draw[bmod] (0, 0) -- (0, .75) -- (-.25, 1)--(0,1.25)--(0,1.5);
					\draw[bmod] (0,0.75) -- (0.5,1.25)--(.5,1.5);
					\draw[bmod] (.5, 0) -- (.5, .25) -- (.25, .5)--(.5,.75)--(0,1.25);
					\draw[bmod] (.5, 0.25) -- (1, .75) -- (0.5, 1.25);
					\draw[bmod] (1,0)--(1,0.25)--(.5,.75);
					\draw[bmod] (1.25, 0) -- (1.25, .5)--(1,.75);
					\btoken{0, 1.25}{east}{\psi};
					\btoken{.5, 1.25}{west}{\rho};
					\btoken{1, .75}{west}{\rho};
					\btoken{0, .75}{east}{\mathcal{H}};
					\btoken{0.5, .75}{east}{\psi};
					\btoken{0.5, .25}{east}{\mathcal{H}};
				\end{tikzpicture} \ = \ 
				\begin{tikzpicture}[anchorbase]
					\draw[bmod] (0.25, 0) -- (0.25, .5) -- (0, .75)--(0,1)--(0.25,1.25)--(0.25,1.5);
					\draw[bmod] (0.25,0.5) -- (1,1.25)--(1,1.5);
					\draw[bmod] (.75, 0) -- (.75, .25) -- (.5, .5)--(.75,.75)--(0.25,1.25);
					\draw[bmod] (.75, 0.25) -- (1, .5) -- (1, .75)-- (.75, 1);
					\draw[bmod] (1.25,0)--(1.25,0.25)--(.75,.75);
					\draw[bmod] (1.5, 0) -- (1.5, .75)--(1,1.25);
					\btoken{0.25, 1.25}{east}{\psi};
					\btoken{1, 1.25}{west}{\rho};
					\btoken{.75, 1}{west}{m};
					\btoken{0.25, .5}{east}{\mathcal{H}};
					\btoken{0.75, .75}{north}{\psi};
					\btoken{0.75, .25}{east}{\mathcal{H}};
				\end{tikzpicture} \ = \ 
				\begin{tikzpicture}[anchorbase]
					\draw[bmod] (0.25, 0) -- (0.25, .25) -- (0, .5)--(0,.75)--(0.5,1.25)--(0.5,1.5);
					\draw[bmod] (0.25,0.25) -- (1.25,1.25)--(1.25,1.5);
					\draw[bmod] (.75, 0) -- (.75, .25) -- (.5, .5)--(.5,.75)--(0.75,1)--(0.5,1.25);
					\draw[bmod] (.75, 0.25) -- (1.25, .75) -- (1, 1);
					\draw[bmod] (1.5,0)--(1.5,1)--(1.25,1.25);
					\draw[bmod] (1,0)--(1,.75)--(.75,1);
					\btoken{0.5, 1.25}{east}{\psi};
					\btoken{1.25, 1.25}{west}{\rho};
					\btoken{1, 1}{west}{m};
					\btoken{0.25, .25}{east}{\mathcal{H}};
					\btoken{0.75, 1}{east}{\psi};
					\btoken{0.75, .25}{east}{\mathcal{H}};
				\end{tikzpicture} \ = \ 
				\begin{tikzpicture}[anchorbase]
					\draw[bmod] (0.25, 0) -- (0.25, .25) -- (0, .5)--(0,.75)--(0.5,1.25)--(0.5,1.5);
					\draw[bmod] (0.25,0.25) -- (1.25,1.25)--(1.25,1.5);
					\draw[bmod] (.75, 0) -- (.75, .25) -- (.5, .5)--(.5,.75)--(0.75,1)--(0.5,1.25);
					\draw[bmod] (.75, 0.25) -- (1, .5) -- (.75, .75);
					\draw[bmod] (1.5,0)--(1.5,1)--(1.25,1.25);
					\draw[bmod] (1.25,0)--(1.25,.5)--(.5,1.25);
					\btoken{0.5, 1.25}{east}{\psi};
					\btoken{1.25, 1.25}{west}{\rho};
					\btoken{.75, .75}{north}{m};
					\btoken{0.25, .25}{east}{\mathcal{H}};
					\btoken{0.75, 1}{east}{\psi};
					\btoken{0.75, .25}{west}{\mathcal{H}};
				\end{tikzpicture} \ . 
\end{equation}
Moreover, we have
\begin{equation}\label{fig:ten}
				\begin{tikzpicture}[anchorbase]
	\draw[bmod] (0,-0.5)--(0.5, 0) -- (0.5,0.25) -- (0.25, .5)--(0.5,.75)--(0,1.25)--(0,1.5);
	\draw[bmod] (0.5, .25) -- (1,.75)--(.5,1.25);
	\draw[bmod] (.5,-.5)--(0, 0)--(0,.75)--(-.25,1)--(0,1.25)--(0,1.5);
	\draw[bmod] (0, .75) -- (0.5, 1.25)--(.5,1.5);
	\draw[bmod] (1,-0.5)--(1, 0) -- (1, .25) -- (0.5, .75);
	\draw[bmod] (1.25,-0.5)--(1.25, 0) -- (1.25, .5) -- (1, .75);
	\btoken{1, .75}{west}{\rho};
	\btoken{.5, .75}{west}{\psi};
	\btoken{0, 1.25}{east}{\psi};
	\btoken{0, 0.75}{east}{\mathcal{H}}; 
	\btoken{0.5, .25}{west}{\mathcal{H}}; 
	\btoken{.5, 1.25}{west}{\rho};
\end{tikzpicture} \ = \ 	
\begin{tikzpicture}[anchorbase]
	\draw[bmod] (.75,-0.5)--(0.25, 0) -- (0.25, .5) -- (0, .75)--(0,1)--(0.25,1.25)--(0.25,1.5);
	\draw[bmod] (0.25,0.5) -- (1,1.25)--(1,1.5);
	\draw[bmod] (.25,-0.5)--(.75, 0) -- (.75, .25) -- (.5, .5)--(.75,.75)--(0.25,1.25);
	\draw[bmod] (.75, 0.25) -- (1, .5) -- (1, .75)-- (.75, 1);
	\draw[bmod] (1.25,-.5)--(1.25,0)--(1.25,0.25)--(.75,.75);
	\draw[bmod] (1.5,-0.5)--(1.5, 0) -- (1.5, .75)--(1,1.25);
	\btoken{0.25, 1.25}{east}{\psi};
	\btoken{1, 1.25}{west}{\rho};
	\btoken{.75, 1}{west}{m};
	\btoken{0.25, .5}{east}{\mathcal{H}};
	\btoken{0.75, .75}{north}{\psi};
	\btoken{0.75, .25}{west}{\mathcal{H}};
\end{tikzpicture} \ = \ 
\begin{tikzpicture}[anchorbase]
	\draw[bmod] (.25,0)--(0.75, 0.5) -- (0.75, .75) -- (0.5, 1)--(0.5,1.25)--(0.75,1.5)--(0.5,1.75)--(.5,2);
	\draw[bmod] (0.75,0.75) -- (1.25,1.25)--(1,1.5)--(1.25,1.75)--(1.25,2);
	\draw[bmod] (.75, 0) -- (.25, .5) -- (.25, .75)--(0,1)--(0,1.25)--(.5,1.75);
	\draw[bmod] (.25, 0.75) -- (1.25, 1.75);
	\draw[bmod] (1,0)--(1,1.25)--(.5,1.75);
	\draw[bmod] (1.5, 0) -- (1.5, 1.5)--(1.25,1.75);
	\btoken{0.5, 1.75}{east}{\psi};
	\btoken{1.25, 1.75}{west}{\rho};
	\btoken{1, 1.5}{west}{m};
	\btoken{0.25, .75}{east}{\mathcal{H}};
	\btoken{0.75, 1.5}{south}{\psi};
	\btoken{0.75, .75}{east}{\mathcal{H}};
\end{tikzpicture}\ = \
\begin{tikzpicture}[anchorbase]
	\draw[bmod] (.25,0)--(0.75, 0.5) -- (0.75, .75) -- (0.5, 1)--(0.5,1.25)--(0.75,1.5)--(0.5,1.75)--(.5,2);
	\draw[bmod] (0.75,0.75) -- (1.25,1.25)--(1,1.5)--(1.25,1.75)--(1.25,2);
	\draw[bmod] (.75, 0) -- (.25, .5) -- (.25, .75)--(0,1)--(0,1.25)--(.5,1.75);
	\draw[bmod] (.25, 0.75) -- (1.25, 1.75);
	\draw[bmod] (1,0)--(1,1.25)--(.5,1.75);
	\draw[bmod] (1.5, 0) -- (1.5, 1.5)--(1.25,1.75);
	\btoken{0.5, 1.75}{east}{\psi};
	\btoken{1.25, 1.75}{west}{\rho};
	\btoken{1, 1.5}{west}{m};
	\btoken{0.25, .75}{east}{\mathcal{H}};
	\btoken{0.75, 1.5}{south}{\psi};
	\btoken{0.75, .75}{east}{\mathcal{H}};
\end{tikzpicture}
\end{equation}
\[=\]
\begin{equation*}
\begin{tikzpicture}[anchorbase]
	\draw[bmod] (.75,0)--(0.75, 0.25) -- (0.5, .5) -- (0.5, 1.5)--(0.75,1.75)--(0.5,2)--(0.5,2.25);
	\draw[bmod] (0.75,0.25) -- (1,.5)--(1,1.25)--(.75,1.5)--(1.25,2)--(1.25,2.25);
	\draw[bmod] (1, 0) -- (1, .25) -- (.25, .75)--(0.25,1)--(0,1.25)--(0,1.5)--(.5,2)--(.5,2.25);
	\draw[bmod] (.25, 1) -- (1.25, 2);
	\draw[bmod] (1.25,0)--(1.25,1.25)--(.5,2);
	\draw[bmod] (1.5, 0) -- (1.5, 1.75)--(1.25,2);
	\btoken{0.5, 2}{east}{\psi};
	\btoken{1.25, 2}{west}{\rho};
	\btoken{.75, 1.5}{north}{m};
	\btoken{0.25, 1}{east}{\mathcal{H}};
	\btoken{0.75, 1.75}{south}{\psi};
	\btoken{0.75, .25}{east}{\mathcal{H}};
\end{tikzpicture} \ = \
\begin{tikzpicture}[anchorbase]
	\draw[bmod] (.5,0)--(0.5, 0.25) -- (0, .75) -- (0, 1)--(0.75,1.75)--(0.5,2)--(0.5,2.25);
	\draw[bmod] (0.5,0.25) -- (1,.75)--(1,1.25)--(.75,1.5)--(1.25,2)--(1.25,2.25);
	\draw[bmod] (.75, 0) -- (.75, .25) -- (.25, .75)--(0.25,1)--(0,1.25)--(0,1.5)--(.5,2)--(.5,2.25);
	\draw[bmod] (.25, 1) -- (1.25, 2);
	\draw[bmod] (1.25,0)--(1.25,1.25)--(.5,2);
	\draw[bmod] (1.5, 0) -- (1.5, 1.75)--(1.25,2);
	\btoken{0.5, 2}{east}{\psi};
	\btoken{1.25, 2}{west}{\rho};
	\btoken{.75, 1.5}{north}{m};
	\btoken{0.25, 1}{west}{\mathcal{H}};
	\btoken{0.75, 1.75}{south}{\psi};
	\btoken{0.5, .25}{east}{\mathcal{H}};
\end{tikzpicture} \ = \ 
\begin{tikzpicture}[anchorbase]
	\draw[bmod] (.25,0)--(0.25, 0.25) -- (0, .5) -- (0, 1)--(0.75,1.75)--(0.5,2)--(0.5,2.25);
	\draw[bmod] (0.25,0.25) -- (1,1)--(.75,1.25)--(.75,1.5)--(1.25,2)--(1.25,2.25);
	\draw[bmod] (.75, 0) -- (.75, .25) -- (.25, .75)--(0.25,1.75)--(.5,2);
	\draw[bmod] (.75, .25) -- (1, .5)--(.5,1)--(.75,1.25);
	\draw[bmod] (1.25,0)--(1.25,1.25)--(.5,2);
	\draw[bmod] (1.5, 0) -- (1.5, 1.75)--(1.25,2);
	\btoken{0.5, 2}{east}{\psi};
	\btoken{1.25, 2}{west}{\rho};
	\btoken{.75, 1.25}{north}{m};
	\btoken{0.25, .25}{east}{\mathcal{H}};
	\btoken{0.75, 1.75}{south}{\psi};
	\btoken{0.75, .25}{west}{\mathcal{H}};
\end{tikzpicture} \ = \ 
\begin{tikzpicture}[anchorbase]
	\draw[bmod] (.25,0)--(0.25, 0.25) -- (0, .5) -- (0, 1)--(0.75,1.75)--(0.5,2)--(0.5,2.25);
	\draw[bmod] (0.25,0.25) -- (.5,.5)--(.5,1)--(.75,1.25)--(.75,1.5)--(1.25,2)--(1.25,2.25);
	\draw[bmod] (.75, 0) -- (.75, .25) -- (1, .5)--(1,1)--(.75,1.25);
	\draw[bmod] (.75, .25) -- (.25, .75)--(.25,1.75)--(.5,2);
	\draw[bmod] (1.25,0)--(1.25,1.25)--(.5,2);
	\draw[bmod] (1.5, 0) -- (1.5, 1.75)--(1.25,2);
	\btoken{0.5, 2}{east}{\psi};
	\btoken{1.25, 2}{west}{\rho};
	\btoken{.75, 1.25}{north}{m};
	\btoken{0.25, .25}{east}{\mathcal{H}};
	\btoken{0.75, 1.75}{south}{\psi};
	\btoken{0.75, .25}{west}{\mathcal{H}};
\end{tikzpicture} \ .
\end{equation*}
			
For equation \ref{fig:ten}, the first equality comes from the $A$-module axiom of $(M,\rho)$. The second, third, and fourth equalities follow from the naturality of braidings $(c_{A, \mathfrak{L}\otimes V}, \id_A,\psi), (c_{A\otimes A,V},m,\id_V)$ and $(c_{L,L},\mathcal{H},\id_L)$. Lemma \ref{identityofbraid} implies the fifth equality, the natural isomorphism of the braiding $(c_{L,L},\id_L,\mathcal{H})$ gives the sixth equality. The last equality is due to the commutativity of $A$.
			
Since $(V,\psi)$ is an $\mathfrak{L}$-module, we have 
\begin{equation}\label{fig:four-ten}
\begin{tikzpicture}[anchorbase]
	\draw[bmod] (0, 0) -- (0, .75) -- (-.25, 1)--(0,1.25)--(0,1.5);
	\draw[bmod] (0,0.75) -- (0.5,1.25)--(.5,1.5);
	\draw[bmod] (.5, 0) -- (.5, .25) -- (.25, .5)--(.5,.75)--(0,1.25);
	\draw[bmod] (.5, 0.25) -- (1, .75) -- (0.5, 1.25);
	\draw[bmod] (1,0)--(1,0.25)--(.5,.75);
	\draw[bmod] (1.25, 0) -- (1.25, .5)--(1,.75);
	\btoken{0, 1.25}{east}{\psi};
	\btoken{.5, 1.25}{west}{\rho};
	\btoken{1, .75}{west}{\rho};
	\btoken{0, .75}{east}{\mathcal{H}};
	\btoken{0.5, .75}{west}{\psi};
	\btoken{0.5, .25}{east}{\mathcal{H}};
\end{tikzpicture}  \ - \ 
\begin{tikzpicture}[anchorbase]
	\draw[bmod] (0.5, 0) -- (0.5,0.25) -- (0.25, .5)--(0.5,.75)--(0,1.25)--(0,1.5);
	\draw[bmod] (0.5, .25) -- (1,.75)--(.5,1.25);
	\draw[bmod] (.75, 0) -- (0, .25)--(0,.75)--(-.25,1)--(0,1.25)--(0,1.5);
	\draw[bmod] (0, .75) -- (0.5, 1.25)--(.5,1.5);
	\draw[bmod] (1, 0) -- (1, .25) -- (0.5, .75);
	\draw[bmod] (1.25, 0) -- (1.25, .5) -- (1, .75);
	\btoken{1, .75}{west}{\rho};
	\btoken{.5, .75}{west}{\psi};
	\btoken{0, 1.25}{east}{\psi};
	\btoken{0, 0.75}{east}{\mathcal{H}}; 
	\btoken{0.5, .25}{west}{\mathcal{H}}; 
	\btoken{.5, 1.25}{west}{\rho};
\end{tikzpicture} \ = \ 
\begin{tikzpicture}[anchorbase]
	\draw[bmod] (0.25, 0) -- (0.25,0.5) -- (0, .75)--(0.75,1.5)--(0.75,1.75);
	\draw[bmod] (0.25, .5) -- (1.25,1.5)--(1.25,1.75);
	\draw[bmod] (.75, 0) -- (0.75, .5)--(0.25,1);
	\draw[bmod] (0.75, .5) -- (1, .75)--(.75,1);
	\draw[bmod] (1.25, 0) -- (1.25, 1) -- (0.75, 1.5);
	\draw[bmod] (1.5, 0) -- (1.5, 1.25) -- (1.25, 1.5);
	\btoken{1.25, 1.5}{west}{\rho};
	\btoken{.75, 1.5}{east}{\psi};
	\btoken{0.25, 1}{east}{\mathfrak{u}};
	\btoken{0.25, 0.5}{east}{\mathcal{H}}; 
	\btoken{0.75, .5}{west}{\mathcal{H}}; 
	\btoken{.75, 1}{west}{m};
\end{tikzpicture} \ .
\end{equation}
Note that $\mathcal{H}$ is a crossed homomorphism, the pair $(V\otimes M, \psi \vartriangleright\gamma )$ already becomes an $L$-module by combining the above pieces together. 
			
Finally, we check that $(V\otimes M ; \psi \vartriangleright \rho , \psi \vartriangleright \gamma )$ satisfies the compatible condition \eqref{compatible2}, i.e., 
\begin{equation}\label{fig:enter-label31}
\begin{tikzpicture}[anchorbase]
	\draw[bmod] (0, 0) -- (0,1.25) -- (0.5, 1.75)--(0.5,2);
	\draw[bmod] (0.75, 0) -- (.75,1)--(1,1.25)--(.5,1.75);
	\draw[bmod] (1.25, 0) -- (1.25, 1)--(0.5,1.75);
	\draw[bmod] (1, 0) -- (1, .5)--(0.25,1.25)--(.25,2);
	\btoken{1, 1.25}{west}{\rho};
	\btoken{.5, 1.75}{west}{\gamma};
\end{tikzpicture} \ + \ 
\begin{tikzpicture}[anchorbase]
	\draw[bmod] (0.25, 0) -- (0.25,1) -- (0, 1.25)--(0.5,1.75)--(.5,2);
	\draw[bmod] (0.25, 1) -- (1,1.75)--(1,2);
	\draw[bmod] (1, 0) -- (1, .75)--(1.5,1.25);
	\draw[bmod] (1.5, 0) -- (1.5, .75)--(0.5,1.75);
	\draw[bmod] (1.75, 0) -- (1.75, 1)--(1,1.75);
	\btoken{1, 1.75}{west}{\rho};
	\btoken{1.5, 1.25}{west}{\rho};
	\btoken{.5, 1.75}{east}{\psi};
		\btoken{.25, 1}{east}{\mathcal{H}};
\end{tikzpicture} \ = \ 
\begin{tikzpicture}[anchorbase]
	\draw[bmod] (0, 0) -- (0,.25) -- (1, 1.25)--(0.5,1.75)--(.5,2);
	\draw[bmod] (0.5, 0) -- (.5,.25)--(0,.75)--(0,1.25)--(.5,1.75);
	\draw[bmod] (.75, 0) -- (.75, .5)--(.25,1)--(.25,2);
	\draw[bmod] (1.25, 0) -- (1.25, 1)--(1,1.25);
	\btoken{.5, 1.75}{west}{\rho};
	\btoken{1, 1.25}{west}{\gamma};
\end{tikzpicture} \ + \ 
\begin{tikzpicture}[anchorbase]
	\draw[bmod] (0.5, 0) -- (0.5,.5) -- (.25, .75)--(0.75,1.25)--(.75,2);
	\draw[bmod] (0.5, 0.5) -- (.75,.75)--(0.75,1)--(1.25,1.5)--(1,1.75);
	\draw[bmod] (.75, 0) -- (0, .75)--(1,1.75)--(1,2);
	\draw[bmod] (1, 0) -- (1, 1)--(.75,1.25);
	\draw[bmod] (1.5, 0) -- (1.5, 1.25)--(1.25,1.5);
	\btoken{1, 1.75}{west}{\rho};
	\btoken{1.25, 1.5}{west}{\rho};
	\btoken{.75, 1.25}{north}{\psi};
	\btoken{.5, .5}{west}{\mathcal{H}};
\end{tikzpicture} \ + \ 
\begin{tikzpicture}[anchorbase]
	\draw[bmod] (0, 0) -- (0,1) -- (.75, 1.75)--(0.75,2);
	\draw[bmod] (0.5, 0) -- (.5,1)--(.25,1.25);
	\draw[bmod] (.75, 0) -- (0.75, 1.25)--(.5,1.5)--(.5,2);
	\draw[bmod] (1, 0) -- (1, 1.5)--(.75,1.75);
	\btoken{.75, 1.75}{west}{\rho};
	\btoken{.25, 1.25}{east}{\alpha};
\end{tikzpicture} \ . 
\end{equation}
By the compatible condition of the weak module $(M;\rho,\gamma)$, we obtain  
\begin{equation*}
\begin{tikzpicture}[anchorbase]
	\draw[bmod] (0, 0) -- (0,1.25) -- (0.5, 1.75)--(0.5,2);
	\draw[bmod] (0.75, 0) -- (.75,1)--(1,1.25)--(.5,1.75);
	\draw[bmod] (1.25, 0) -- (1.25, 1)--(0.5,1.75);
	\draw[bmod] (1, 0) -- (1, .5)--(0.25,1.25)--(.25,2);
	\btoken{1, 1.25}{west}{\rho};
	\btoken{.5, 1.75}{west}{\gamma};
\end{tikzpicture} \ + \ 
\begin{tikzpicture}[anchorbase]
	\draw[bmod] (0, 0) -- (0,.25) -- (1, 1.25)--(0.5,1.75)--(.5,2);
	\draw[bmod] (0.5, 0) -- (.5,.25)--(0,.75)--(0,1.25)--(.5,1.75);
	\draw[bmod] (.75, 0) -- (.75, .5)--(.25,1)--(.25,2);
	\draw[bmod] (1.25, 0) -- (1.25, 1)--(1,1.25);
	\btoken{.5, 1.75}{west}{\rho};
	\btoken{1, 1.25}{west}{\gamma};
\end{tikzpicture} \ = \ 
\begin{tikzpicture}[anchorbase]
	\draw[bmod] (0, 0) -- (0,1) -- (.75, 1.75)--(0.75,2);
	\draw[bmod] (0.5, 0) -- (.5,1)--(.25,1.25);
	\draw[bmod] (.75, 0) -- (0.75, 1.25)--(.5,1.5)--(.5,2);
	\draw[bmod] (1, 0) -- (1, 1.5)--(.75,1.75);
	\btoken{.75, 1.75}{west}{\rho};
	\btoken{.25, 1.25}{east}{\alpha};
\end{tikzpicture} \ .
\end{equation*}
Moreover, we have 
\begin{equation}\label{compatible222}
\begin{tikzpicture}[anchorbase]
	\draw[bmod] (0.5, 0) -- (0.5,.5) -- (.25, .75)--(0.75,1.25)--(.75,2);
	\draw[bmod] (0.5, 0.5) -- (.75,.75)--(0.75,1)--(1.25,1.5)--(1,1.75);
	\draw[bmod] (.75, 0) -- (0, .75)--(1,1.75)--(1,2);
	\draw[bmod] (1, 0) -- (1, 1)--(.75,1.25);
	\draw[bmod] (1.5, 0) -- (1.5, 1.25)--(1.25,1.5);
	\btoken{1, 1.75}{west}{\rho};
	\btoken{1.25, 1.5}{west}{\rho};
	\btoken{.75, 1.25}{north}{\psi};
	\btoken{.5, .5}{west}{\mathcal{H}};
\end{tikzpicture} \ = \ 
\begin{tikzpicture}[anchorbase]
	\draw[bmod] (0.5, 0) -- (0.5,.5) -- (.25, .75)--(0.25,1.5)--(.5,1.75)--(.5,2);
	\draw[bmod] (0.5, 0.5) -- (.75,.75)--(0.75,1)--(1.25,1.5)--(1,1.75)--(1,2);
	\draw[bmod] (.75, 0) -- (0, .75)--(1,1.75);
	\draw[bmod] (1.25, 0) -- (1.25, 1)--(.5,1.75);
	\draw[bmod] (1.5, 0) -- (1.5, 1.25)--(1.25,1.5);
	\btoken{1, 1.75}{west}{\rho};
	\btoken{1.25, 1.5}{west}{\rho};
	\btoken{.5, 1.75}{east}{\psi};
	\btoken{.5, .5}{west}{\mathcal{H}};
\end{tikzpicture} \ = \ 
\begin{tikzpicture}[anchorbase]
	\draw[bmod] (0.25, 0) -- (0.25,.5) -- (0, .75)--(0,1.25)--(.5,1.75)--(.5,2);
	\draw[bmod] (0.25, 0.5) -- (1.25,1.5)--(1,1.75)--(1,2);
	\draw[bmod] (.5, 0) -- (0.5, 1.25)--(1,1.75);
	\draw[bmod] (1.25, 0) -- (1.25, 1)--(.5,1.75);
	\draw[bmod] (1.5, 0) -- (1.5, 1.25)--(1.25,1.5);
	\btoken{1, 1.75}{west}{\rho};
	\btoken{1.25, 1.5}{west}{\rho};
	\btoken{.5, 1.75}{east}{\psi};
	\btoken{.25, .5}{east}{\mathcal{H}};
\end{tikzpicture}
\end{equation}
\[ =\]
\begin{equation*}
\begin{tikzpicture}[anchorbase]
	\draw[bmod] (0.25, 0) -- (0.25,.5) -- (0, .75)--(0,1.5)--(.25,1.75)--(.25,2);
	\draw[bmod] (0.25, 0.5) -- (1.25,1.5)--(1,1.75)--(1,2);
	\draw[bmod] (.75, 0) -- (0.75, 1.5)--(1,1.75)--(1,2);
	\draw[bmod] (1, 0) -- (1, .5)--(.5,1)--(.5,1.5)--(.25,1.75);
	\draw[bmod] (1.5, 0) -- (1.5, 1.25)--(1.25,1.5);
	\btoken{1, 1.75}{west}{\rho};
	\btoken{1.25, 1.5}{west}{\rho};
	\btoken{.25, 1.75}{east}{\psi};
	\btoken{.25, .5}{east}{\mathcal{H}};
\end{tikzpicture} \ = \ 
\begin{tikzpicture}[anchorbase]
	\draw[bmod] (0.5, 0) -- (0.5,.5) -- (0, 1)--(0,1.5)--(.25,1.75)--(.25,2);
	\draw[bmod] (0.5, 0.5) -- (1.25,1.25)--(1,1.5)--(1.25,1.75)--(1.25,2);
	\draw[bmod] (1.25, 0) -- (1.25, .75)--(.75,1.25)--(1,1.5)--(1.25,1.75);
	\draw[bmod] (1.5, 0) -- (.5, 1)--(.5,1.5)--(.25,1.75);
	\draw[bmod] (1.75, 0) -- (1.75, 1.25)--(1.25,1.75);
	\btoken{1.25, 1.75}{west}{\rho};
	\btoken{1, 1.5}{west}{m};
	\btoken{.25, 1.75}{east}{\psi};
	\btoken{.5, .5}{east}{\mathcal{H}};
\end{tikzpicture} \ = \ 
\begin{tikzpicture}[anchorbase]
	\draw[bmod] (0.5, 0) -- (0.5,.5) -- (0, 1)--(0,1.5)--(.25,1.75)--(.25,2);
	\draw[bmod] (0.5, 0.5) -- (1.25,1.25)--(1,1.5)--(1.25,1.75)--(1.25,2);
	\draw[bmod] (1.25, 0) -- (1.25, .75)--(.75,1.25)--(1,1.5)--(1.25,1.75);
	\draw[bmod] (1.5, 0) -- (1.5, .25)--(.5,1.25)--(.5,1.5)--(.25,1.75);
	\draw[bmod] (1.75, 0) -- (1.75, 1.25)--(1.25,1.75);
	\btoken{1.25, 1.75}{west}{\rho};
	\btoken{1, 1.5}{west}{m};
	\btoken{.25, 1.75}{east}{\psi};
	\btoken{.5, .5}{east}{\mathcal{H}};
\end{tikzpicture} \ = \ 
\begin{tikzpicture}[anchorbase]
	\draw[bmod] (0.25, 0) -- (0.25,1) -- (0, 1.25)--(0.5,1.75)--(.5,2);
	\draw[bmod] (0.25, 1) -- (1,1.75)--(1,2);
	\draw[bmod] (1, 0) -- (1, .75)--(1.5,1.25);
	\draw[bmod] (1.5, 0) -- (1.5, .75)--(0.5,1.75);
	\draw[bmod] (1.75, 0) -- (1.75, 1)--(1,1.75);
	\btoken{1, 1.75}{west}{\rho};
	\btoken{1.5, 1.25}{west}{\rho};
	\btoken{.5, 1.75}{east}{\psi};
	\btoken{.25, 1}{east}{\mathcal{H}};
\end{tikzpicture} \ ,
\end{equation*}
where the first equality comes from the natural isomorphism of the braiding $(c_{A,\mathfrak{L}\otimes V}, \id_A,\psi)$. 
Thanks to Lemma \ref{identityofbraid}, the next equality holds. The natural isomorphism of the braiding $(c_{A\otimes A,V}, c_{A,A}, \id_V)$ implies the third equality. 
Since $(M,\rho)$ is an $A$-module and $A$ is commutative, we have done it!
\end{proof} 

We next recall the categorical analogues of modules over a semigroup: left module categories, right module categories, and bimodule categories. 
These structures provide a categorical formulation of left modules, right modules, and bimodules, respectively.
		
\begin{defi}
A \emph{left module category} over $\mathcal{C}$ is a category $\mathcal{M}$ equipped with a bifunctor $\vartriangleright : \mathcal{C} \times \mathcal{M} \rightarrow \mathcal{M}$, a natural isomorphism \[a^{\mathcal{M}}_{U,V,Z} :(U\otimes V)\vartriangleright Z \to U\vartriangleright (V \vartriangleright Z),\]  called a left module associativity constraint, and a natural isomorphism \[l^{\mathcal{M}}_Z : I \vartriangleright Z \to Z,\] called a left module unit constraint,  such that the following diagrams
\begin{equation}\label{module1}
				\begin{tikzcd}[sep=large]
					(U \otimes V \otimes W) \vartriangleright Z \ar[r, "a^{\mathcal{M}}_{U,V\otimes W,Z}"] \ar[d, "a^{\mathcal{M}}_{U\otimes V,W,Z}"'] & U \vartriangleright \big((V \otimes W) \vartriangleright Z\big) \ar[d, "\id_U\vartriangleright a^{\mathcal{M}}_{V,W,Z}"] \\
					(U \otimes V) \vartriangleright (W \vartriangleright Z) \ar[r, "a^{\mathcal{M}}_{U,V,W\vartriangleright Z}"']             & U \vartriangleright \big(V \vartriangleright (W \vartriangleright Z)\big)
				\end{tikzcd}
\end{equation}
and
\begin{equation}\label{module2}
				\begin{tikzpicture}[description/.style={fill=white,inner sep=2pt}]
					\matrix (m) [matrix of math nodes, row sep=3em,
					column sep=2.5em, text height=1.5ex, text depth=0.25ex]
					{ 
						(V \otimes I) \vartriangleright Z & & V \vartriangleright (I \vartriangleright Z) \\
						& V \vartriangleright Z & \\
					};
					\path[->,font=\scriptsize]
					(m-1-1) edge node[auto] {$a^{\mathcal{M}}_{V, I, Z}$} (m-1-3)
					(m-1-1) edge node[auto,swap] {$\id_{V\vartriangleright Z} $} (m-2-2)
					(m-1-3) edge node[auto] {$\text{Id}_V \vartriangleright l^{\mathcal{M}}_Z$} (m-2-2);
				\end{tikzpicture}
\end{equation}
are commutative, for all $U,V,W \in \calC, Z\in \mathcal{M}$. Moreover, it is said to be \emph{strict} if $a^{\mathcal{M}}_{U,V,Z}$ and $l^{\mathcal{M}}_{Z}$ are all identities. 
\end{defi}
Similarly, one obtains the notion of a right module category over $\mathcal{C}$ by placing the action of $\mathcal{C}$ on the right. 
We now combine the left and right actions and recall the corresponding notion of a bimodule category. 
\begin{defi}
Let $\mathcal{D}$ be another strict monoidal category. A \emph{bimodule category} over $(\calC,\mathcal{D})$ is a category $\mathcal{M}$ equipped with bifunctors $\vartriangleright: \calC \times \mathcal{M} \to \mathcal{M}$ and $\vartriangleleft:\mathcal{M} \times \mathcal{D} \to \mathcal{M}$, and natural isomorphisms 
\begin{align*}
				a^{\mathcal{M}}_{U,V,Z}& :(U\otimes V)\vartriangleright Z \to U\vartriangleright (V \vartriangleright Z)\\
				l^{\mathcal{M}}_Z&: I \vartriangleright Z \to Z\\
				b^{\mathcal{M}}_{Z,X,Y}&: Z\vartriangleleft(X\otimes Y) \to (Z\vartriangleleft X)\vartriangleleft Y\\
				r^{\mathcal{M}}_Z&:Z\vartriangleleft I \to Z\\
				d^{\mathcal{M}}_{U,Z,X}&:(U\vartriangleright Z)\vartriangleleft X \to U \vartriangleright(Z\vartriangleleft X)
\end{align*}
for all $U,V\in \calC,Z\in \mathcal{M}, X,Y\in \mathcal{D}$, such that $\mathcal{M}$ is a left module category over $\calC$ and a right module category over $\mathcal{D}$, and the following diagrams
\[
			\begin{tikzcd}[column sep=1.8em, row sep=3.2ex]
				& \bigl((U\otimes V)\vartriangleright  Z\bigr)\vartriangleleft X
				\ar[dl, "a^{\mathcal{M}}_{U,V,Z}\vartriangleleft \mathrm{id}_X"']
				\ar[dr, "d^{\mathcal{M}}_{U\otimes V,Z,X}"] \\
				\bigl(U\vartriangleright (V\vartriangleright  Z)\bigr)\vartriangleleft X
				\ar[d, "d^{\mathcal{M}}_{U,V\vartriangleright Z,X}"']
				&& (U\otimes V)\vartriangleright(Z\vartriangleleft X)
				\ar[d, "a^{\mathcal{M}}_{U,V,Z\vartriangleleft X}"] \\
				U\vartriangleright \bigl((V\vartriangleright Z)\vartriangleleft X\bigr)
				\ar[rr, "\mathrm{id}_U\vartriangleright  d^{\mathcal{M}}_{V,Z,X}"']
				&& U\vartriangleright \bigl(V\vartriangleright(Z\vartriangleleft X)\bigr)
			\end{tikzcd}
\]
and
\[
			\begin{tikzcd}[column sep=1.8em, row sep=3.2ex]
				& U\vartriangleright\bigl( Z\vartriangleleft( X\otimes Y)\bigr)
				\ar[dl, "\id_U\vartriangleright b^{\mathcal{M}}_{Z,X,Y}  "'] \\
				U\vartriangleright\bigl((Z\vartriangleleft X) \vartriangleleft Y  \bigr)
				&& (U\vartriangleright Z)\vartriangleleft (X\otimes Y)
				\ar[d, "b^{\mathcal{M}}_{U\vartriangleright Z,X,Y   }"] 
				\ar[ul, "d^{\mathcal{M}}_{U,Z,X\otimes Y}"']\\
				\bigl(U\vartriangleright (Z\vartriangleleft X)  \bigr)\vartriangleleft Y
				\ar[u,"d^{\mathcal{M}}_{U,Z\vartriangleleft X,Y}"]
				&& \bigl( (U\vartriangleright Z)\vartriangleleft  X   \bigr) \vartriangleleft Y
				\ar[ll, "d^{\mathcal{M}}_{U,Z,X} \vartriangleleft \id_Y"']
			\end{tikzcd}
\]
commute for all $U,V\in \calC,Z\in \mathcal{M}, X,Y\in \mathcal{D}$. Moreover, it is said to be \emph{strict} if $a^{\mathcal{M}}_{U,V,Z}, l^{\mathcal{M}}_{Z},b^{\mathcal{M}}_{Z,X,Y},r^{\mathcal{M}}_Z$ and $d^{\mathcal{M}}_{U,Z,X}$ are all identities. 
\end{defi}

It is immediate that $\mathcal{C}$ is both a left and a right module category over itself, with the action given by the tensor product of $\mathcal{C}$.
Consequently, $\mathcal{C}$ is naturally a bimodule category over
$(\mathcal{C},\mathcal{C})$. 
Moreover, by Lemma~\ref{tensorofamod}, for any monoid $A$ in $\mathcal{C}$, the category ${\bf Mod}_{\mathcal{C}}(A)$ of
left $A$-modules carries a natural left module category structure over $\mathcal{C}$.
\begin{rmk}
Let $\mathcal{D}$ be a strict monoidal category. 
A functor $F:\mathcal{C}\rightarrow \mathcal{D}$ is called a \emph{strict monoidal functor} if
\[
	F(I_{\mathcal C})=I_{\mathcal D} \ \ \text{and} \ \ 
	F(V\otimes W)=F(V)\otimes F(W)
\]
for all objects $V,W\in \mathrm{ob}(\mathcal{C})$, and if the analogous equalities hold on morphisms.
A strict left $\mathcal{C}$-module category $\mathcal{M}$ is equivalently the same as a strict monoidal functor $F:\mathcal{C}\rightarrow \mathrm{End}(\mathcal{M}).$
More explicitly, the action
	\[
	\mathcal{C}\times \mathcal{M}\rightarrow \mathcal{M},
	\quad
	(V,M)\mapsto V\otimes M,
	\]
corresponds to the strict monoidal functor defined by $F(V)(M)=V\otimes  M.$
Conversely, any strict monoidal functor
$F:\mathcal{C}\rightarrow \mathrm{End}(\mathcal{M})$
defines a strict left $\mathcal{C}$-module category structure on $\mathcal{M}$ by setting $V\otimes  M:=F(V)(M).$
Thus, module categories categorify the classical fact that a module over a ring is equivalently a representation of that ring on an abelian group.
\end{rmk}

With the preceding preparations, we are now ready to state the main result of this section.
		
\begin{theor}\label{maintheorem}
Let $(A,L,\alpha)$ be a Lie-Rinehart monoid in $\calC$. Suppose  $(\mathfrak{L}, \mathfrak{u})$ is a Lie monoid in $\calC$, and $\mathcal{H}: L \rightarrow \mathfrak{L}\otimes A$  a crossed homomorphism. Then the category ${\bf WMod}^{\calC}_A{\big(L\big)}$ is a strict left module category over the monoidal category $\mathrm{ \bf Mod}_{\calC}(\mathfrak{L})$. Specifically, for all $(V,\psi), (V', \psi')\in $ $\mathrm{ \bf Mod}_{\calC}(\mathfrak{L})$, and $(M;\rho, \gamma), (M'; \rho',\gamma') \in$ ${\bf WMod}^{\calC}_A{\big(L\big)}$,  we have a bifunctor
\begin{align*}
				\mathcal{F}_{\mathcal{H}}: \mathrm{ \bf Mod}_{\calC}(\mathfrak{L}) \times {\bf WMod}^{\calC}_A{\big(L\big)} &\rightarrow {\bf WMod}^{\calC}_A{\big(L\big)} \\
				((V,\psi), (M; \rho,\gamma) ) & \mapsto (V\otimes M; \psi \vartriangleright \rho, \psi \vartriangleright \gamma ).
\end{align*}
On morphisms, it is defined by $\mathcal{F}_{\mathcal{H}}(\varphi, \phi)=\varphi\otimes \phi$, where $\varphi:V\to V'$ is a homomorphism of $\mathfrak{L}$-modules and $\phi:M\to M'$ is a homomorphism of weak modules over $(A,L,\alpha)$. 
\end{theor}
\begin{proof}
By Lemma \ref{main1}, the triple $(V\otimes M; \psi \vartriangleright \rho, \psi \vartriangleright \gamma )$ is indeed a weak module over $(A,L,\alpha)$. 
Thus the assignment $\mathcal{F}_{\mathcal H}$ is well-defined on objects. 
We next verify that $\mathcal{F}_{\mathcal H}$ is well-defined on morphisms.
Let $\varphi:(V,\psi)\rightarrow (V',\psi')$
be a homomorphism of $\mathfrak L$-modules, and let
$\phi:(M;\rho,\gamma)\rightarrow (M';\rho',\gamma')$
be a homomorphism of weak modules over $(A,L,\alpha)$. We need to show that
\[
\varphi\otimes \phi:(V\otimes M;\psi \vartriangleright \rho,\psi\vartriangleright \gamma  )\rightarrow (V'\otimes M';\psi'\vartriangleright \rho',\psi'\vartriangleright \gamma' )
\]
is a homomorphism of weak modules. Since $\phi$ is a homomorphism of $A$-modules, it follows immediately from the definition of $\psi\vartriangleright\rho$ that $\varphi\otimes\phi$ is a homomorphism of $A$-modules. It remains to
verify that it is compatible with the $L$-actions, namely that
\[
(\varphi\otimes \phi)\circ(\psi\vartriangleright \gamma)
=
(\psi'\vartriangleright \gamma')\circ
(\mathrm{id}_L\otimes \varphi\otimes \phi).
\]
This identity is established in 
\begin{equation}\label{tensorofhom}
\begin{tikzpicture}[anchorbase]
	\draw[bmod] (0, 0) -- (0,.25) -- (0.5, .75)--(0.5,1.25);
	\draw[bmod] (0.5, 0) -- (.5,.25)--(0,.75)--(0,1.25);
	\draw[bmod] (.75, 0) -- (.75, .5)--(.5,.75);
	\btoken{0, 1}{east}{\varphi};
	\btoken{.5, 1}{west}{\phi};
	\btoken{.5, .75}{north}{\gamma};
\end{tikzpicture} \ + \
\begin{tikzpicture}[anchorbase]
	\draw[bmod] (0.25, 0) -- (0.25,.25) -- (0, .5)--(0.25,.75)--(.25,1.25);
	\draw[bmod] (0.25, 0.25) -- (.75,.75)--(0.75,1.25);
	\draw[bmod] (.75, 0) -- (.75, .25)--(.25,.75);
	\draw[bmod] (1, 0) -- (1, .5)--(.75,.75);
	\btoken{0.25, 1}{east}{\varphi};
	\btoken{.75, 1}{west}{\phi};
	\btoken{.25, .75}{north}{\psi};
	\btoken{.75, .75}{north}{\rho};
	\btoken{.25, .25}{east}{\mathcal{H}};
\end{tikzpicture} \ = \ 
\begin{tikzpicture}[anchorbase]
	\draw[bmod] (0, 0) -- (0,.5) -- (0.5, 1)--(0.5,1.25);
	\draw[bmod] (0.5, 0) -- (.5,.25)--(0.25,.5)--(0.25,1.25);
	\draw[bmod] (.75, 0) -- (.75, .75)--(.5,1);
	\btoken{0.25, 1}{east}{\varphi};
	\btoken{.75, .5}{west}{\phi};
	\btoken{.5, 1}{west}{\gamma'};
\end{tikzpicture} \ + \ 
\begin{tikzpicture}[anchorbase]
	\draw[bmod] (0.25, 0) -- (0.25,.25) -- (0, .5)--(0,.75)--(.25,1)--(.25,1.25);
	\draw[bmod] (0.25, 0.25) -- (.75,.75)--(0.75,1.25);
	\draw[bmod] (.5, 0) -- (.5, .75)--(.25,1);
	\draw[bmod] (1, 0) -- (1, .5)--(.75,.75);
	\btoken{0.5,.75}{east}{\varphi};
	\btoken{.75, 1}{west}{\phi};
	\btoken{.25, 1}{east}{\psi'};
	\btoken{.75, .75}{north}{\rho};
	\btoken{.25, .25}{east}{\mathcal{H}};
\end{tikzpicture}
\end{equation}
\[ =\]
\begin{equation*}
\begin{tikzpicture}[anchorbase]
	\draw[bmod] (0, 0) -- (0,.5) -- (0.5, 1)--(0.5,1.25);
	\draw[bmod] (0.5, 0) -- (.5,.5)--(0.25,.75)--(0.25,1.25);
	\draw[bmod] (.75, 0) -- (.75, .75)--(.5,1);
	\btoken{0.5, .25}{east}{\varphi};
	\btoken{.75, .5}{west}{\phi};
	\btoken{.5, 1}{west}{\gamma'};
\end{tikzpicture}\ + \ 
\begin{tikzpicture}[anchorbase]
	\draw[bmod] (0.25, 0) -- (0.25,.25) -- (0, .5)--(0,.75)--(.25,1)--(.25,1.25);
	\draw[bmod] (0.25, 0.25) -- (.75,.75)--(0.75,1.25);
	\draw[bmod] (.75, 0) -- (.75, .5)--(.25,1);
	\draw[bmod] (1, 0) -- (1, .5)--(.75,.75);
	\btoken{0.75,.25}{east}{\varphi};
	\btoken{.75, 1}{west}{\phi};
	\btoken{.25, 1}{east}{\psi'};
	\btoken{.75, .75}{west}{\rho};
	\btoken{.25, .25}{east}{\mathcal{H}};
\end{tikzpicture} \ = \ 
\begin{tikzpicture}[anchorbase]
	\draw[bmod] (0, 0) -- (0,.5) -- (0.5, 1)--(0.5,1.25);
	\draw[bmod] (0.5, 0) -- (.5,.5)--(0.25,.75)--(0.25,1.25);
	\draw[bmod] (.75, 0) -- (.75, .75)--(.5,1);
	\btoken{0.5, .25}{east}{\varphi};
	\btoken{.75, .5}{west}{\phi};
	\btoken{.5, 1}{west}{\gamma'};
\end{tikzpicture}\ + \ 
\begin{tikzpicture}[anchorbase]
	\draw[bmod] (0.25, 0) -- (0.25,.5) -- (0, .75)--(0.25,1)--(.25,1.25);
	\draw[bmod] (0.25, 0.5) -- (.75,1)--(0.75,1.25);
	\draw[bmod] (.75, 0) -- (.75, .5)--(.25,1);
	\draw[bmod] (1, 0) -- (1, .75)--(.75,1);
	\btoken{0.75,.25}{east}{\varphi};
	\btoken{1, .25}{west}{\phi};
	\btoken{.25, 1}{east}{\psi'};
	\btoken{.75, 1}{west}{\rho'};
	\btoken{.25, .5}{east}{\mathcal{H}};
\end{tikzpicture} \ , 
\end{equation*}
where the
first equality follows from the fact that $\varphi$ is a homomorphism of
$\mathfrak L$-modules and $\phi$ is a homomorphism of $L$-modules. The second
equality follows from the naturality of the braiding with respect to the
morphisms involved, in particular from the naturality relation for
$(c_{A,V},\mathrm{id}_A,\varphi)$. The last equality follows from the fact
that $\phi$ is a homomorphism of $A$-modules. Hence $\varphi\otimes\phi$ is a
homomorphism of weak modules. Therefore $\mathcal{F}_{\mathcal H}$ is a
bifunctor.

It remains to verify the strict module category axioms. Since the ambient
monoidal category is assumed to be strict, it suffices to show that the
associativity constraint for the action is the identity. Equivalently, for any
$(V,\psi),(W,\theta)\in {\bf Mod}_{\mathcal C}(\mathfrak L)$
and any
$(M;\rho,\gamma)\in {\bf WMod}^{\mathcal C}_A(L),$
we have
\begin{align}
	\label{iden1}
	(\psi\boxplus \theta)\vartriangleright \rho
	&=
	\psi\vartriangleright(\theta\vartriangleright \rho),
	\\
	\label{iden2}
	(\psi\boxplus \theta)\vartriangleright\gamma
	&=
	\psi\vartriangleright(\theta\vartriangleright \gamma).
\end{align}
Here $\psi\boxplus \theta$ denotes the $\mathfrak L$-module structure on
$V\otimes W$ induced by the Lie monoid structure of $\mathfrak L$.

The identity \eqref{iden1} follows directly from the definition of the
$A$-action in \eqref{newhom}. The proof of \eqref{iden2} is obtained by using the
definition of the induced $L$-action, the compatibility of the crossed
homomorphism $\mathcal H$, and the naturality of the braiding, in particular
the naturality relation involving $c_{A,V}$ and $\mathcal H$.
Finally, by the definition of the unit action in \eqref{newhom}, the left unit
constraint is the identity:
\[
l^{\mathcal M}_Z=\mathrm{id}_Z \ \text{for all} \ 
Z\in {\bf WMod}^{\mathcal C}_A(L).
\]
Consequently, ${\bf WMod}^{\mathcal C}_A(L)$ is a strict left module category
over ${\bf Mod}_{\mathcal C}(\mathfrak L)$.
\end{proof}
\begin{rmk}
Let $\mathcal{H} $ and $\mathcal{H}'$ be  crossed homomorphisms from $L$ to $ \mathfrak{L}\otimes A$.
If
\begin{equation*}
\begin{tikzpicture}[anchorbase]
	\draw[bmod] (0.25, 0) -- (0.25,.25) -- (0, .5)--(0.25,.75)--(.25,1);
	\draw[bmod] (0.25, 0.25) -- (.75,.75)--(0.75,1);
	\draw[bmod] (.75, 0) -- (.75, .25)--(.25,.75);
    \draw[bmod] (1, 0) -- (1, .5)--(.75,.75);
	\btoken{0.75, .75}{west}{\rho};
	\btoken{.25, .75}{east}{\psi};
	\btoken{.25, .25}{east}{\mathcal{H}};
\end{tikzpicture} \ = \ 
\begin{tikzpicture}[anchorbase]
	\draw[bmod] (0.25, 0) -- (0.25,.25) -- (0, .5)--(0.25,.75)--(.25,1);
	\draw[bmod] (0.25, 0.25) -- (.75,.75)--(0.75,1);
	\draw[bmod] (.75, 0) -- (.75, .25)--(.25,.75);
    \draw[bmod] (1, 0) -- (1, .5)--(.75,.75);
	\btoken{0.75, .75}{west}{\rho};
	\btoken{.25, .75}{east}{\psi};
	\btoken{.25, .25}{east}{\mathcal{H}'};
\end{tikzpicture}
\end{equation*}
for all $(V,\psi)\in $ $\mathrm{ \bf Mod}_{\calC}(\mathfrak{L})$ and $(M;\rho, \gamma) \in$ ${\bf WMod}^{\calC}_A{\big(L\big)}$, then we have $\mathcal{F}_{\mathcal{H}}= \mathcal{F}_{\mathcal{H}'}$. 
\end{rmk}

As an immediate application of Theorem~\ref{maintheorem} to the iterated Lie-Rinehart monoids constructed above, we obtain the following result.		
\begin{coro}
Let $(A,L,\alpha)$ be a Lie-Rinehart monoid in $\calC$. For every
$r\in \mathbb Z_{\geq 1}$, let $(A,L_A^{[r]},\alpha_A^{[r]})$
be the $r$-th iterated Lie-Rinehart monoid associated with $(A,L,\alpha)$.
Let $(\mathfrak L,\mathfrak u)$ be a Lie monoid in $\mathcal C$, and let
$\mathcal H:L_r^A\longrightarrow \mathfrak L\otimes A$
be a crossed homomorphism. Then the category
${\bf WMod}^{\mathcal C}_A(L_A^{[r]})$
is a strict left module category over the monoidal category
${\bf Mod}_{\mathcal C}(\mathfrak L).$
\end{coro}

Another useful special case is obtained by choosing the trivial Lie monoid.
More precisely, if we take $\mathfrak L$ to be the Lie monoid $(I,{\bf 0})$ in Theorem~\ref{maintheorem}, then we obtain the following corollary.
		
\begin{coro}\label{newfunctor}
Let $(A,L,\alpha)$ be a Lie-Rinehart monoid in $\calC$. Suppose that $\mathcal H:L\rightarrow A\cong I\otimes A$
is a crossed homomorphism with respect to the trivial Lie monoid
$(I,\mathbf 0),$ i.e., 
			\begin{equation*}
				\begin{tikzpicture}[anchorbase]
					\draw[bmod] (0, 0) -- (0, .25) -- (.25, .5)--(.25,1);
					\draw[bmod] (.5,0) -- (0.5,.25)--(.25,.5);
					\btoken{.25, .5}{west}{\mu};
					\btoken{0.25, .75}{east}{\mathcal{H}};
				\end{tikzpicture} \ = \ 
				\begin{tikzpicture}[anchorbase]
					\draw[bmod] (0, 0) -- (0, .5) -- (.25, .75)--(.25,1);
					\draw[bmod] (.5,0) -- (0.5,.5)--(.25,.75);
					\btoken{.25, .75}{east}{\alpha};
					\btoken{0.5, .5}{west}{\mathcal{H}};
				\end{tikzpicture} \ - \ 
				\begin{tikzpicture}[anchorbase]
					\draw[bmod] (0, 0) -- (0.5, .5) -- (.25, .75)--(.25,1);
					\draw[bmod] (.5,0) -- (0,.5)--(.25,.75);
					\btoken{.25, .75}{east}{\alpha};
					\btoken{0.5, .5}{west}{\mathcal{H}};
				\end{tikzpicture}\ . 
\end{equation*}
Then the category ${\bf WMod}^{\mathcal C}_A(L)$
admits a strict left module category structure over
${\bf Mod}_{\mathcal C}(I) \cong \calC.$
\end{coro}
		
Thanks to Theorem~\ref{maintheorem}, the module category structure on ${\bf WMod}^{\mathcal C}_A(L)$ gives rise to several natural functors.
		
\begin{coro}
Let $(A,L,\alpha)$ be a Lie-Rinehart monoid,   $(\mathfrak{L}, \mathfrak{u})$ be a Lie monoid in $\calC$, and let $\mathcal{H}: L \rightarrow \mathfrak{L}\otimes A$ be a crossed homomorphism. Then we have

(1) for every weak module $(M; \rho,\gamma) \in$ ${\bf WMod}^{\calC}_A{\big(L\big)}$, there is a functor 
\begin{align*}
\mathcal{F}_{\mathcal{H}}^{(M;\rho,\gamma)}: \mathrm{ \bf Mod}_{\calC}(\mathfrak{L})  &\rightarrow {\bf WMod}^{\calC}_A{\big(L\big)} \\
(V,\psi) & \mapsto (V\otimes M; \psi \vartriangleright \rho, \psi \vartriangleright \gamma );
\end{align*}

(2) for every $(V,\psi)\in$ {\bf Mod}$_{\calC}(\mathfrak{L})$, there is an endofunctor 
\begin{align*}
\mathcal{F}_{\mathcal{H}}^{(V,\psi)}: {\bf WMod}^{\calC}_A{\big(L\big)}  &\rightarrow {\bf WMod}^{\calC}_A{\big(L\big)} \\
(M;\rho,\gamma) & \mapsto (V\otimes M; \psi \vartriangleright \rho, \psi \vartriangleright \gamma ).
\end{align*}
\end{coro}

\subsection{Monoidal structure on the category ${\bf WMod}^{\calC}_A{\big(L\big)}$ }
Throughout this subsection, we assume that $(A,m,\eta,\Delta,\varepsilon)$ is a bimonoid and $(A,L,\alpha)$ is a Lie-Rinehart monoid in $\calC$. 
By equations \eqref{tensorofasso} and \eqref{tensoroflmodule}, the categories {\bf Mod}$_{\calC}(A)$ and {\bf Mod}$_{\calC}(L)$ are equipped with monoidal structures with respect to the tensor products $\boxtimes$ and $\boxplus$, respectively. Explicitly, for any $(M;\rho,\gamma), (M';\rho',\gamma') \in$ ${\bf WMod}^{\calC}_A{\big(L\big)}$, we have $(M\otimes M', \rho \boxtimes \rho') \in$ {\bf Mod}$_{\calC}(A)$ and $(M\otimes M', \gamma \boxplus \gamma' ) \in$ {\bf Mod}$_{\calC}(L)$. 
		
A natural question is whether these two induced structures are compatible in the sense of weak modules. In other words, one may ask under what additional conditions the object
\[
(M\otimes M';\rho\boxtimes \rho',\gamma\boxplus \gamma')
\]
is again a weak module over $(A,L,\alpha)$. If this holds functorially, then
the category ${\bf WMod}^{\mathcal C}_A(L)$ admits a natural monoidal structure. To answer this question, we impose additional compatibility conditions on the Lie-Rinehart monoid $(A,L,\alpha)$. This leads to the notion of a
Lie-Rinehart bimonoid, which will be introduced below.
		
\begin{defi}
The Lie-Rinehart monoid $(A,L,\alpha)$ is called a \emph{Lie-Rinehart bimonoid} if it satisfies the following additional axioms:
\begin{equation}\label{compatibleofh}
				\begin{tikzpicture}[anchorbase]
					\draw[bmod] (0, 0) -- (0.25, .25) -- (.25, .75)--(0,1);
					\draw[bmod] (.25,.75) -- (0.5,1);
					\draw[bmod] (.5,0) -- (0.25,.25);
					\btoken{.25, .25}{west}{\alpha};
					\btoken{.25, .75}{west}{\Delta};
				\end{tikzpicture} \ = \ 
				\begin{tikzpicture}[anchorbase]
					\draw[bmod] (0, 0) -- (0, .5) -- (.25, .75)--(0.25,1);
					\draw[bmod] (.5,0) -- (0.5,.5)--(.25,.75);
					\draw[bmod] (.5,0.5) -- (0.75,.75)--(.75,1);
					\btoken{.25, .75}{east}{\alpha};
					\btoken{.5, .5}{west}{\Delta};
				\end{tikzpicture} \ + \ 
				\begin{tikzpicture}[anchorbase]
					\draw[bmod] (0, 0) -- (0, .25) -- (.5, .75)--(0.5,1);
					\draw[bmod] (.5,0) -- (0.5,.25)--(0,.75)--(0,1);
					\draw[bmod] (.5,0.25) -- (0.75,.5)--(.5,.75);
					\btoken{.5, .75}{east}{\alpha};
					\btoken{.5, .25}{west}{\Delta};
				\end{tikzpicture} \ \ \text{and} \ \ 
				\begin{tikzpicture}[anchorbase]
					\draw[bmod] (0, 0) -- (0, .25) -- (.25, .5)--(0.25,1);
					\draw[bmod] (.5,0) -- (0.5,.25)--(0.25,.5);;
					\btoken{.25, .5}{west}{\alpha};
					\btoken{.25, 1}{west}{\varepsilon};
				\end{tikzpicture} =0 \ . 
\end{equation}
\end{defi}
		
\begin{lem}\label{comoften}
Assume that $(A,L,\alpha)$ is a Lie-Rinehart bimonoid. Then, for any two weak modules
$(M;\rho,\gamma),\ (M';\rho',\gamma')
\in {\bf WMod}^{\mathcal C}_A(L),$ the triple
$(M\otimes M';\rho\boxtimes \rho',\gamma\boxplus\gamma')$
is again a weak module over $(A,L,\alpha)$.
\end{lem}
\begin{proof}
The proof is given by 
\begin{equation*}
\begin{tikzpicture}[anchorbase]
	\draw[bmod] (0, 0) -- (0,1.25) -- (0.25, 1.5)--(0.25,1.75);
	\draw[bmod] (0.5, 0) -- (.5,.75)--(0.25,1)--(.5,1.25)--(.25,1.5);
	\draw[bmod] (1, 0) -- (1,.75)--(0.5,1.25);
	\draw[bmod] (.5, 0.75) -- (1.25, 1.5)--(1.25,1.75);
	\draw[bmod] (1.5, 0) -- (1.5, 1.25)--(1.25,1.5);
	\btoken{0.25, 1.5}{east}{\gamma};
	\btoken{1.25, 1.5}{west}{\rho'};
	\btoken{.5, 1.25}{west}{\rho};
	\btoken{.5, .75}{east}{\Delta};
\end{tikzpicture} \ + \ 
\begin{tikzpicture}[anchorbase]
	\draw[bmod] (0, 0) -- (0,.75) -- (0.75, 1.5)--(0.75,1.75);
	\draw[bmod] (0.5, 0) -- (.5,.25)--(0.25,.5)--(.25,.75)--(.5,1)--(.5,1.75);
	\draw[bmod] (.5, 0.25) -- (1.25,1)--(0.75,1.5);
	\draw[bmod] (1.5, 0) -- (1.5, .75)--(1.25,1);
	\draw[bmod] (1, 0) -- (1, .5)--(.5,1);
	\btoken{0.75, 1.5}{west}{\gamma'};
	\btoken{1.25, 1}{west}{\rho'};
	\btoken{.5, 1}{west}{\rho};
	\btoken{.5, .25}{east}{\Delta};
\end{tikzpicture}\ = \ 
\begin{tikzpicture}[anchorbase]
	\draw[bmod] (0, 0) -- (0,1) -- (0.5, 1.5)--(0.5,1.75);
	\draw[bmod] (0.5, 0) --(.5,1)--(.25,1.25);
	\draw[bmod] (1, 0) -- (1,1)--(0.5,1.5);
	\draw[bmod] (.5, 1) -- (1, 1.5)--(1,1.75);
	\draw[bmod] (1.25, 0) -- (1.25, 1.25)--(1,1.5);
	\btoken{0.25, 1.25}{north}{\alpha};
	\btoken{1, 1.5}{west}{\rho'};
	\btoken{.5, 1.5}{east}{\rho};
	\btoken{.5, 1}{west}{\Delta};
\end{tikzpicture} \ + \ 
\begin{tikzpicture}[anchorbase]
	\draw[bmod] (0, 0) -- (0,.75) -- (0.5, 1.25)--(0.25,1.5)--(.25,1.75);
	\draw[bmod] (0.25, 0) --(.25,.75)--(0,1)--(0,1.25)--(.25,1.5)--(.25,1.75);
	\draw[bmod] (.25, .75) -- (1, 1.5)--(1,1.75);
	\draw[bmod] (.75, 0) -- (.75,1)--(0.25,1.5);
	\draw[bmod] (1.25, 0) -- (1.25, 1.25)--(1,1.5);
	\btoken{0.5, 1.25}{south}{\gamma};
	\btoken{1, 1.5}{west}{\rho'};
	\btoken{.25, 1.5}{east}{\rho};
	\btoken{.25, .75}{west}{\Delta};
\end{tikzpicture} \ + \ 
\begin{tikzpicture}[anchorbase]
	\draw[bmod] (0, 0) -- (0,.75) -- (0.75, 1.5)--(0.75,1.75);
	\draw[bmod] (0.5, 0) --(.5,.25)--(.25,.5)--(.5,.75)--(.25,1)--(.25,1.75);
	\draw[bmod] (.5, .25) -- (.75, .5)--(.75,1)--(.5,1.25);
	\draw[bmod] (1, 0) -- (1,.25)--(0.5,.75);
	\draw[bmod] (1.25, 0) -- (1.25, 1)--(.75,1.5);
	\btoken{0.5, 1.25}{north}{\alpha};
	\btoken{.75, 1.5}{west}{\rho'};
	\btoken{.5, .75}{east}{\rho};
	\btoken{.5, .25}{west}{\Delta};
\end{tikzpicture} \ + \ 
\begin{tikzpicture}[anchorbase]
	\draw[bmod] (0, 0) -- (0,.5) -- (1, 1.5)--(0.75,1.75)--(.75,2);
	\draw[bmod] (0.5, 0) --(.5,.25)--(.25,.5)--(.5,.75)--(.5,1)--(.25,1.25)--(.25,2);
	\draw[bmod] (.5, .25) -- (1, .75)--(1,1)--(.5,1.5)--(.75,1.75);
	\draw[bmod] (1, 0) -- (1,.25)--(0.5,.75);
	\draw[bmod] (1.25, 0) -- (1.25, 1.25)--(1,1.5);
	\btoken{1, 1.5}{west}{\gamma'};
	\btoken{.75, 1.75}{east}{\rho'};
	\btoken{.5, .75}{west}{\rho};
	\btoken{.5, .25}{west}{\Delta};
\end{tikzpicture}
\end{equation*}
\[ =\]
\begin{equation*}
\begin{tikzpicture}[anchorbase]
	\draw[bmod] (0, 0) -- (0,1) -- (0.5, 1.5)--(0.5,1.75);
	\draw[bmod] (0.5, 0) --(.5,1)--(.25,1.25);
	\draw[bmod] (1, 0) -- (1,1)--(0.5,1.5);
	\draw[bmod] (.5, 1) -- (1, 1.5)--(1,1.75);
	\draw[bmod] (1.25, 0) -- (1.25, 1.25)--(1,1.5);
	\btoken{0.25, 1.25}{north}{\alpha};
	\btoken{1, 1.5}{west}{\rho'};
	\btoken{.5, 1.5}{east}{\rho};
	\btoken{.5, 1}{west}{\Delta};
\end{tikzpicture} \ + \ 
\begin{tikzpicture}[anchorbase]
	\draw[bmod] (0, 0) -- (0,.75) -- (0.5, 1.25)--(0.25,1.5)--(.25,1.75);
	\draw[bmod] (0.25, 0) --(.25,.75)--(0,1)--(0,1.25)--(.25,1.5)--(.25,1.75);
	\draw[bmod] (.25, .75) -- (1, 1.5)--(1,1.75);
	\draw[bmod] (.75, 0) -- (.75,1)--(0.25,1.5);
	\draw[bmod] (1.25, 0) -- (1.25, 1.25)--(1,1.5);
	\btoken{0.5, 1.25}{south}{\gamma};
	\btoken{1, 1.5}{west}{\rho'};
	\btoken{.25, 1.5}{east}{\rho};
	\btoken{.25, .75}{west}{\Delta};
\end{tikzpicture} \ + \ 
\begin{tikzpicture}[anchorbase]
	\draw[bmod] (0, 0) -- (0,.75) -- (0.75, 1.5)--(0.75,1.75);
	\draw[bmod] (0.5, 0) --(.5,.25)--(.25,.5)--(.5,.75)--(.25,1)--(.25,1.75);
	\draw[bmod] (.5, .25) -- (.75, .5)--(.75,1)--(.5,1.25);
	\draw[bmod] (1, 0) -- (1,.25)--(0.5,.75);
	\draw[bmod] (1.25, 0) -- (1.25, 1)--(.75,1.5);
	\btoken{0.5, 1.25}{north}{\alpha};
	\btoken{.75, 1.5}{west}{\rho'};
	\btoken{.5, .75}{east}{\rho};
	\btoken{.5, .25}{west}{\Delta};
\end{tikzpicture} \ + \ 
\begin{tikzpicture}[anchorbase]
	\draw[bmod] (0.25, 0) -- (0.25,.75) -- (0, 1)--(0.5,1.5)--(.5,1.75);
	\draw[bmod] (0.25, 0.75) --(1,1.5)--(1,1.75);
	\draw[bmod] (.75, 0) -- (.75, .75)--(1.25,1.25)--(1,1.5);
	\draw[bmod] (1.25, 0) -- (1.25,.75)--(0.5,1.5); 		
	\draw[bmod] (1.5, 0) -- (1.5, 1)--(1,1.5);
	\btoken{1.25, 1.25}{north}{\gamma'};
	\btoken{1, 1.5}{west}{\rho'};
	\btoken{.5, 1.5}{east}{\rho};
	\btoken{.25, .75}{east}{\Delta};
\end{tikzpicture}
\end{equation*}
\[=\]
\begin{equation*}
\begin{tikzpicture}[anchorbase]
	\draw[bmod] (0, 0) -- (0,1) -- (0.5, 1.5)--(0.5,1.75);
	\draw[bmod] (0.5, 0) --(.5,1)--(.25,1.25);
	\draw[bmod] (1, 0) -- (1,1)--(0.5,1.5);
	\draw[bmod] (.5, 1) -- (1, 1.5)--(1,1.75);
	\draw[bmod] (1.25, 0) -- (1.25, 1.25)--(1,1.5);
	\btoken{0.25, 1.25}{north}{\alpha};
	\btoken{1, 1.5}{west}{\rho'};
	\btoken{.5, 1.5}{east}{\rho};
	\btoken{.5, 1}{west}{\Delta};
\end{tikzpicture} \ + \ 
\begin{tikzpicture}[anchorbase]
	\draw[bmod] (0, 0) -- (1,1) -- (0.5, 1.5)--(0.5,1.75);
	\draw[bmod] (0.25, 0) --(.25,.75)--(0,1)--(.5,1.5);
	\draw[bmod] (.25, .75) -- (1,1.5)--(1,1.75);
	\draw[bmod] (1.25, 0) -- (1.25, .75)--(1,1);
	\draw[bmod] (1.5, 0) -- (1.5, 1)--(1,1.5);
	\btoken{1, 1}{north}{\gamma};
	\btoken{1, 1.5}{west}{\rho'};
	\btoken{.5, 1.5}{east}{\rho};
	\btoken{.25, .75}{east}{\Delta};
\end{tikzpicture}\ + \ 
\begin{tikzpicture}[anchorbase]
	\draw[bmod] (0, 0) -- (0,.25) -- (1.25, 1.5)--(1.25,1.75);
	\draw[bmod] (0.75, 0) --(.75,.25)--(0.25,.75)--(.25,1.25)--(.5,1.5)--(.5,1.75);
	\draw[bmod] (.75, .25) -- (1.25,.75)--(1.25,1)--(1,1.25);
	\draw[bmod] (1.5, 0) -- (1.5, .5)--(.5,1.5);
	\draw[bmod] (1.75, 0) -- (1.75, 1)--(1.25,1.5);
	\btoken{1, 1.25}{south}{\alpha};
	\btoken{1.25, 1.5}{west}{\rho'};
	\btoken{.5, 1.5}{east}{\rho};
	\btoken{.75, .25}{east}{\Delta};
\end{tikzpicture} \ + \ 
\begin{tikzpicture}[anchorbase]
	\draw[bmod] (0.25, 0) -- (0.25,.75) -- (0, 1)--(0.5,1.5)--(.5,1.75);
	\draw[bmod] (0.25, 0.75) --(1,1.5)--(1,1.75);
	\draw[bmod] (.75, 0) -- (.75, .75)--(1.25,1.25)--(1,1.5);
	\draw[bmod] (1.25, 0) -- (1.25,.75)--(0.5,1.5); 		
	\draw[bmod] (1.5, 0) -- (1.5, 1)--(1,1.5);
	\btoken{1.25, 1.25}{north}{\gamma'};
	\btoken{1, 1.5}{west}{\rho'};
	\btoken{.5, 1.5}{east}{\rho};
	\btoken{.25, .75}{east}{\Delta};
\end{tikzpicture}
\end{equation*}
\[ =\]
\begin{equation*}
	\begin{tikzpicture}[anchorbase]
		\draw[bmod] (0, 0) -- (0,1) -- (0.5, 1.5)--(0.5,1.75);
		\draw[bmod] (0.5, 0) --(.5,1)--(.25,1.25);
		\draw[bmod] (1, 0) -- (1,1)--(0.5,1.5);
		\draw[bmod] (.5, 1) -- (1, 1.5)--(1,1.75);
		\draw[bmod] (1.25, 0) -- (1.25, 1.25)--(1,1.5);
		\btoken{0.25, 1.25}{north}{\alpha};
		\btoken{1, 1.5}{west}{\rho'};
		\btoken{.5, 1.5}{east}{\rho};
		\btoken{.5, 1}{west}{\Delta};
	\end{tikzpicture} \ + \ 
	\begin{tikzpicture}[anchorbase]
		\draw[bmod] (0, 0) -- (1,1) -- (0.5, 1.5)--(0.5,1.75);
		\draw[bmod] (0.25, 0) --(.25,.75)--(0,1)--(.5,1.5);
		\draw[bmod] (.25, .75) -- (1,1.5)--(1,1.75);
		\draw[bmod] (1.25, 0) -- (1.25, .75)--(1,1);
		\draw[bmod] (1.5, 0) -- (1.5, 1)--(1,1.5);
		\btoken{1, 1}{north}{\gamma};
		\btoken{1, 1.5}{west}{\rho'};
		\btoken{.5, 1.5}{east}{\rho};
		\btoken{.25, .75}{east}{\Delta};
	\end{tikzpicture}\ + \ 
	\begin{tikzpicture}[anchorbase]
		\draw[bmod] (0, 0) -- (0,.5) -- (1, 1.5)--(1,1.75);
		\draw[bmod] (0.5, 0) --(.5,.25)--(0.25,.5)--(.25,1.25)--(.5,1.5)--(.5,1.75);
		\draw[bmod] (.5, .25) -- (.75,.5)--(.75,.75)--(.5,1);
		\draw[bmod] (1, 0) -- (1, 1)--(.5,1.5);
		\draw[bmod] (1.25, 0) -- (1.25, 1.25)--(1,1.5);
		\btoken{.5, 1}{west}{\alpha};
		\btoken{1, 1.5}{west}{\rho'};
		\btoken{.5, 1.5}{east}{\rho};
		\btoken{.5, .25}{east}{\Delta};
	\end{tikzpicture} \ + \ 
	\begin{tikzpicture}[anchorbase]
		\draw[bmod] (0.25, 0) -- (0.25,.75) -- (0, 1)--(0.5,1.5)--(.5,1.75);
		\draw[bmod] (0.25, 0.75) --(1,1.5)--(1,1.75);
		\draw[bmod] (.75, 0) -- (.75, .75)--(1.25,1.25)--(1,1.5);
		\draw[bmod] (1.25, 0) -- (1.25,.75)--(0.5,1.5); 		
		\draw[bmod] (1.5, 0) -- (1.5, 1)--(1,1.5);
		\btoken{1.25, 1.25}{north}{\gamma'};
		\btoken{1, 1.5}{west}{\rho'};
		\btoken{.5, 1.5}{east}{\rho};
		\btoken{.25, .75}{east}{\Delta};
	\end{tikzpicture}
\end{equation*}
\[=\]
\begin{equation*}
	\begin{tikzpicture}[anchorbase]
		\draw[bmod] (0, 0) -- (0,.25) -- (0.25, .5)--(0.25,.75)--(0,1)--(.5,1.5)--(.5,1.75);
		\draw[bmod] (0.25, 0.75) --(1,1.5)--(1,1.75);
		\draw[bmod] (.5, 0) -- (.5,.25)--(.25,.5);
		\draw[bmod] (1, 0) -- (1, 1)--(.5,1.5);
		\draw[bmod] (1.25, 0) -- (1.25, 1.25)--(1,1.5);
		\btoken{.25, .5}{west}{\alpha};
		\btoken{1, 1.5}{west}{\rho'};
		\btoken{.5, 1.5}{east}{\rho};
		\btoken{.25, .75}{east}{\Delta};
	\end{tikzpicture} \ + 
\begin{tikzpicture}[anchorbase]
	\draw[bmod] (0, 0) -- (1,1) -- (0.5, 1.5)--(0.5,1.75);
	\draw[bmod] (0.25, 0) --(.25,.75)--(0,1)--(.5,1.5);
	\draw[bmod] (.25, .75) -- (1,1.5)--(1,1.75);
	\draw[bmod] (1.25, 0) -- (1.25, .75)--(1,1);
	\draw[bmod] (1.5, 0) -- (1.5, 1)--(1,1.5);
	\btoken{1, 1}{north}{\gamma};
	\btoken{1, 1.5}{west}{\rho'};
	\btoken{.5, 1.5}{east}{\rho};
	\btoken{.25, .75}{east}{\Delta};
\end{tikzpicture} \ + \ 
\begin{tikzpicture}[anchorbase]
	\draw[bmod] (0.25, 0) -- (0.25,.75) -- (0, 1)--(0.5,1.5)--(.5,1.75);
	\draw[bmod] (0.25, 0.75) --(1,1.5)--(1,1.75);
	\draw[bmod] (.75, 0) -- (.75, .75)--(1.25,1.25)--(1,1.5);
	\draw[bmod] (1.25, 0) -- (1.25,.75)--(0.5,1.5); 		
	\draw[bmod] (1.5, 0) -- (1.5, 1)--(1,1.5);
	\btoken{1.25, 1.25}{north}{\gamma'};
	\btoken{1, 1.5}{west}{\rho'};
	\btoken{.5, 1.5}{east}{\rho};
	\btoken{.25, .75}{east}{\Delta};
\end{tikzpicture} \ .
\end{equation*}
The first equality follows from the compatibility conditions for $(M;\rho,\gamma)$ and $(M';\rho',\gamma')$. The natural isomorphism of $$\Big(c_{L,A\otimes M}, \id_L, (\rho\otimes \id_A)\circ (\id_A \otimes c_{A,M})\circ (\Delta\otimes \id_M) \Big)$$ implies the next equality. Applying the natural isomorphism of $(c_{L,A\otimes M}, \id_L,\rho)$ together with Lemma \ref{natof1}, we have the third equality. Finally, the natural isomorphism of $(c_{L\otimes A,M},\alpha,\id_M)$ and the condition \eqref{compatibleofh} ensure last two equality. 
\end{proof}
		
It is clear that $(I,\varepsilon)$ is a left $A$-module and $(I, {\bf 0})$ is an $L$-module, where ${\bf 0}$ denotes the zero morphism. 
		
\begin{lem}\label{comofunit}
Assume that $(A,L,\alpha)$ is a Lie-Rinehart bimonoid. Then the triple  $(I;\varepsilon, \textbf{0}) \in {\bf WMod}^{\calC}_A{\big(L\big)}$.
\end{lem}
\begin{proof}
It remains only to verify the compatibility condition \eqref{compatible2}.
This follows immediately from the defining compatibility axiom of a
Lie-Rinehart bimonoid together with the fact that the $L$-action on $I$ is the zero action. Hence $(I;\varepsilon,{\bf 0})$ is a weak module over $(A,L,\alpha)$.
\end{proof}
Combining Proposition XV. 1.2 in \cite{Kas95} with Lemmas \ref{comoften}--\ref{comofunit}, the following result is obtained. 
\begin{theor}
The category ${\bf WMod}^{\mathcal C}_A(L)$ is a strict monoidal category under the tensor product
$(M;\rho,\gamma)\otimes (M';\rho',\gamma')
=
(M\otimes M';\rho\boxtimes \rho',\gamma\boxplus \gamma')$
with unit object $(I;\varepsilon,{\bf 0})$ if and only if $(A,L,\alpha)$ is a Lie-Rinehart bimonoid.
\end{theor}
		
In what follows, we assume that ${\bf WMod}^{\calC}_A{\big(L\big)}$ is a strict monoidal category. By Corollary~\ref{newfunctor}, for each fixed
object $V\in \mathrm{ob}(\mathcal C),$
we obtain an endofunctor
\[
\mathcal F_{\mathcal H}^{V}:
{\bf WMod}^{\mathcal C}_A(L)
\rightarrow
{\bf WMod}^{\mathcal C}_A(L).
\]
Recall that a functor $F:\calC \to \mathcal{D}$ between two monoidal categories is a \emph{strict monoidal functor} if we have 
\[F(I)=I \ \mathrm{and} \ F(V\otimes W)=F(V)\otimes F(W)  \]
for all $V,W\in \ob(\calC)$.

It is then natural to ask under what conditions this endofunctor is strict monoidal.		
In the present situation, since
$\mathcal F_{\mathcal H}^{V}(I;\varepsilon,{\bf 0})=V,$
the unit condition forces $V=I$. Therefore, among these endofunctors, only $\mathcal F_{\mathcal H}^{I}$
can possibly be strict monoidal. We shall denote this functor simply by $\mathcal F_{\mathcal H}.$ Moreover, for any $(M;\rho,\gamma)\in$ ${\bf WMod}^{\calC}_A{\big(L\big)}$,  we have 
\[ \mathcal{F}_{\mathcal{H}}(\rho)=\rho \ \mathrm{and} \  \mathcal{F}_{\mathcal{H}}(\gamma)=\gamma+\rho\circ (\mathcal{H}\otimes \id_M).\] It remains to determine when this functor is compatible with the tensor
product on ${\bf WMod}^{\mathcal C}_A(L)$. The answer is governed by the compatibility of $\mathcal H$ with the comultiplication of $A$. 
		
\begin{theor}
The functor $\mathcal{F}_{\mathcal{H}}$ is a strict monoidal functor if and only if 
\begin{equation}\label{comoftensor}
\begin{tikzpicture}[anchorbase]
					\draw[bmod] (0.25, 0) -- (0.25, .5) -- (0, .75);
					\draw[bmod] (.25,.5) -- (0.5,.75);
					\btoken{.25, .25}{west}{\mathcal{H}};
					\btoken{.25, .5}{east}{\Delta};
				\end{tikzpicture} \ = \ 
				\begin{tikzpicture}[anchorbase]
					\draw[bmod] (0, 0) -- (0, .75);
					\draw[bmod] (.25,.25) -- (0.25,.75);
					\btoken{0, .25}{east}{\mathcal{H}};
					\btoken{.25, .25}{west}{\eta};
				\end{tikzpicture} \ + \ 
				\begin{tikzpicture}[anchorbase]
					\draw[bmod] (0, 0.25) -- (0, .75);
					\draw[bmod] (.25,0) -- (0.25,.75);
					\btoken{0.25, .25}{west}{\mathcal{H}};
					\btoken{0, .25}{east}{\eta};
				\end{tikzpicture}\ .
\end{equation}
\end{theor}
\begin{proof}
Since $\mathcal{F}_{\mathcal{H}}(I)=I$, it remains to show that $ \mathcal{F}_{\mathcal{H}}(M) \otimes \mathcal{F}_{\mathcal{H}}(M') =\mathcal{F}_{\mathcal{H}}(M\otimes M')$ for any $(M;\rho,\gamma), (M';\rho', \gamma')\in$ ${\bf WMod}^{\calC}_A{\big(L\big)}$. This reduces to proving the equality \[\mathcal{F}_{\mathcal{H}}(\gamma \boxplus \gamma')=\mathcal{F}_{\mathcal{H}}(\gamma) \boxplus \mathcal{F}_{\mathcal{H}}(\gamma').\] The proof is given by
\begin{equation*}
\begin{tikzpicture}[anchorbase]
	\draw[bmod] (0, 0) -- (0,.5) -- (0.25, .75)--(.25,1);
	\draw[bmod] (0.5, 0) --(.5,.5)--(.25,.75);
	\draw[bmod] (.75, 0) -- (.75, 1);
	\btoken{.25, .75}{east}{\gamma};
\end{tikzpicture}\ + \ 
\begin{tikzpicture}[anchorbase]
	\draw[bmod] (0, 0) -- (0,.5) -- (0.25, .75)--(.25,1);
	\draw[bmod] (0.5, 0) --(.5,.5)--(.25,.75);
	\draw[bmod] (.75, 0) -- (.75, 1);
	\btoken{.25, .75}{east}{\rho};
    \btoken{0, .5}{east}{\mathcal{H}};
\end{tikzpicture} \ + \ 
\begin{tikzpicture}[anchorbase]
	\draw[bmod] (0, 0) -- (0,.5) -- (0.25, .75)--(.25,1);
	\draw[bmod] (0.5, 0) --(.5,.5)--(.25,.75);
	\draw[bmod] (.25, 0) -- (.25, .5)--(0,.75)--(0,1);
	\btoken{.25, .75}{west}{\gamma'};
\end{tikzpicture} \ + \ 
\begin{tikzpicture}[anchorbase]
	\draw[bmod] (0.25, 0) -- (0.25,.5) -- (0.5, .75)--(.5,1);
	\draw[bmod] (0.75, 0) --(.75,.5)--(.5,.75);
	\draw[bmod] (.5, 0) -- (0, .5)--(0,1);
	\btoken{.5, .75}{west}{\rho'};
    \btoken{.25, .5}{west}{\mathcal{H}};
\end{tikzpicture}
\end{equation*}
\[=\]
\begin{equation*}
\begin{tikzpicture}[anchorbase]
	\draw[bmod] (0, 0) -- (0,.5) -- (0.25, .75)--(.25,1);
	\draw[bmod] (0.5, 0) --(.5,.5)--(.25,.75);
	\draw[bmod] (.75, 0) -- (.75, 1);
	\btoken{.25, .75}{east}{\gamma};
\end{tikzpicture}\ + \ 
\begin{tikzpicture}[anchorbase]
	\draw[bmod] (0, 0) -- (0,.5) -- (0.25, .75)--(.25,1);
	\draw[bmod] (0.5, 0) --(.5,.5)--(.25,.75);
	\draw[bmod] (.75, 0) -- (.75, 1);
	\btoken{.25, .75}{east}{\rho};
    \btoken{0, .5}{east}{\mathcal{H}};
\end{tikzpicture} \ + \ 
\begin{tikzpicture}[anchorbase]
	\draw[bmod] (0, 0) -- (0,.5) -- (0.25, .75)--(.25,1);
	\draw[bmod] (0.5, 0) --(.5,.5)--(.25,.75);
	\draw[bmod] (.25, 0) -- (.25, .5)--(0,.75)--(0,1);
	\btoken{.25, .75}{west}{\gamma'};
\end{tikzpicture} \ + \ 
\begin{tikzpicture}[anchorbase]
	\draw[bmod] (0, 0) -- (0,.5) -- (0.25, .75)--(.25,1);
	\draw[bmod] (0.5, 0) --(.5,.5)--(.25,.75);
	\draw[bmod] (.25, 0) -- (0.25, .5)--(0,.75)--(0,1);
	\btoken{.25, .75}{west}{\rho'};
    \btoken{0, .5}{east}{\mathcal{H}};
\end{tikzpicture}
\end{equation*}
\[=\]
\begin{equation*}
\begin{tikzpicture}[anchorbase]
	\draw[bmod] (0, 0) -- (0,.5) -- (0.25, .75)--(.25,1);
	\draw[bmod] (0.5, 0) --(.5,.5)--(.25,.75);
	\draw[bmod] (.75, 0) -- (.75, 1);
	\btoken{.25, .75}{east}{\gamma};
\end{tikzpicture} \ + \ 
\begin{tikzpicture}[anchorbase]
	\draw[bmod] (0, 0) -- (0,.5) -- (0.25, .75)--(.25,1);
	\draw[bmod] (0.5, 0) --(.5,.5)--(.25,.75);
	\draw[bmod] (.25, 0) -- (.25, .5)--(0,.75)--(0,1);
	\btoken{.25, .75}{west}{\gamma'};
\end{tikzpicture} \ + \ 
\begin{tikzpicture}[anchorbase]
	\draw[bmod] (0, 0) -- (0,.5) -- (0.25, .75)--(.25,1);
	\draw[bmod] (0.5, 0) --(.5,.5)--(.25,.75);
	\draw[bmod] (.75, 0) -- (.75, .5)--(1,.75)--(1,1);
    \draw[bmod] (1.25, 0) -- (1.25, 0.5)--(1,.75);
	\btoken{.25, .75}{east}{\rho};
    \btoken{0, .5}{east}{\mathcal{H}};
    \btoken{1, .75}{east}{\rho'};
    \btoken{0.75, 0}{west}{\eta};
\end{tikzpicture} \ + \ 
\begin{tikzpicture}[anchorbase]
	\draw[bmod] (0, 0) -- (0,.5) -- (0.25, .75)--(.25,1);
	\draw[bmod] (0.25, 0) --(.25,.25)--(.75,.75)--(0.75,1);
	\draw[bmod] (.75, 0) -- (.75, .25)--(.25,.75);
    \draw[bmod] (1, 0) -- (1, 0.5)--(.75,.75);
	\btoken{.25, .75}{east}{\rho};
    \btoken{0.25, .25}{west}{\mathcal{H}};
    \btoken{.75, .75}{west}{\rho'};
    \btoken{0, 0}{east}{\eta};
\end{tikzpicture}
\end{equation*}
\[=\]
\begin{equation*}
\begin{tikzpicture}[anchorbase]
	\draw[bmod] (0, 0) -- (0,.5) -- (0.25, .75)--(.25,1);
	\draw[bmod] (0.5, 0) --(.5,.5)--(.25,.75);
	\draw[bmod] (.75, 0) -- (.75, 1);
	\btoken{.25, .75}{east}{\gamma};
\end{tikzpicture} \ + \ 
\begin{tikzpicture}[anchorbase]
	\draw[bmod] (0, 0) -- (0,.5) -- (0.25, .75)--(.25,1);
	\draw[bmod] (0.5, 0) --(.5,.5)--(.25,.75);
	\draw[bmod] (.25, 0) -- (.25, .5)--(0,.75)--(0,1);
	\btoken{.25, .75}{west}{\gamma'};
\end{tikzpicture} \ + \ 
\begin{tikzpicture}[anchorbase]
	\draw[bmod] (0, 0) -- (0,.5) -- (0.25, .75)--(.25,1);
	\draw[bmod] (0.25, 0) --(.25,.25)--(.75,.75)--(.75,1);
	\draw[bmod] (.75, 0) -- (.75, .25)--(.25,.75);
    \draw[bmod] (1, 0) -- (1, 0.5)--(.75,.75);
	\btoken{.25, .75}{east}{\rho};
    \btoken{0, .5}{east}{\mathcal{H}};
    \btoken{.75, .75}{west}{\rho'};
    \btoken{0.25, 0}{west}{\eta};
\end{tikzpicture} \ + \ 
\begin{tikzpicture}[anchorbase]
	\draw[bmod] (0, 0) -- (0,.5) -- (0.25, .75)--(.25,1);
	\draw[bmod] (0.25, 0) --(.25,.25)--(.75,.75)--(0.75,1);
	\draw[bmod] (.75, 0) -- (.75, .25)--(.25,.75);
    \draw[bmod] (1, 0) -- (1, 0.5)--(.75,.75);
	\btoken{.25, .75}{east}{\rho};
    \btoken{0.25, .25}{west}{\mathcal{H}};
    \btoken{.75, .75}{west}{\rho'};
    \btoken{0, 0}{east}{\eta};
\end{tikzpicture}
\end{equation*}
\[=\]
\begin{equation*}
\begin{tikzpicture}[anchorbase]
	\draw[bmod] (0, 0) -- (0,.5) -- (0.25, .75)--(.25,1);
	\draw[bmod] (0.5, 0) --(.5,.5)--(.25,.75);
	\draw[bmod] (.75, 0) -- (.75, 1);
	\btoken{.25, .75}{east}{\gamma};
\end{tikzpicture} \ + \ 
\begin{tikzpicture}[anchorbase]
	\draw[bmod] (0, 0) -- (0,.5) -- (0.25, .75)--(.25,1);
	\draw[bmod] (0.5, 0) --(.5,.5)--(.25,.75);
	\draw[bmod] (.25, 0) -- (.25, .5)--(0,.75)--(0,1);
	\btoken{.25, .75}{west}{\gamma'};
\end{tikzpicture} \ + \ 
\begin{tikzpicture}[anchorbase]
	\draw[bmod] (0.25, 0) -- (0.25,.5) -- (0, .75)--(.25,1)--(.25,1.25);
	\draw[bmod] (0.25, 0.5) --(.75,1)--(.75,1.25);
	\draw[bmod] (.75, 0) -- (.75, .5)--(.25,1);
    \draw[bmod] (1, 0) -- (1, 0.75)--(.75,1);
	\btoken{.25, 1}{east}{\rho};
    \btoken{0.25, .25}{east}{\mathcal{H}};
    \btoken{.75, 1}{west}{\rho'};
    \btoken{0.25, 0.5}{west}{\Delta};
\end{tikzpicture} \ ,
\end{equation*}
where the first equality follows immediately from the natural isomorphism of the braiding $(c_{L,M},\mathcal{H},\id_M)$. Since $(M, \rho)$ and $(M', \rho')$ are $A$-modules, the second equality holds. The next equality is due to the fact that $c_{I,M}=\id_M$ and the natural isomorpism of the braiding $(c_{I,M},\eta,\id_M)$. Using the  condition \eqref{comoftensor}, the proof is complete.
\end{proof}
Next, we give some examples about Lie-Rinehart bimonoids. 
\begin{exam}
Let $\mathcal C=\mathcal{V}ec(\mathbb C)$. Let $A$ be a commutative bimonoid and let $\mathfrak g$ be a Lie
algebra. Set $L=\mathfrak g \otimes A$ with $A$-module structure given by multiplication on the first factor, Lie bracket
\[
	[a\otimes x,\ b\otimes y]=ab\otimes [x,y],
\]
and zero anchor $\alpha=0$. Then $(A,L,\alpha)$ is a Lie-Rinehart bimonoid.
\end{exam}

\begin{exam}
Let $\mathcal C=s\mathcal{V}ec(\mathbb C)$ and let
$A=\Lambda(\xi)=\mathbb C\oplus \mathbb C\xi$
be the exterior algebra on one odd generator $\xi$. 
We equip $A$ with the standard super-commutative bialgebra structure given by
\[
\Delta(1)=1\otimes 1 \ \ \text{and} \ \ 
\Delta(\xi)=\xi\otimes 1+1\otimes \xi,
\]
and
\[
	\varepsilon(1)=1,
	\quad
	\varepsilon(\xi)=0.
\]	
Define a derivation $d:A\longrightarrow A$ by $d(1)=0$ and $d(\xi)=\xi.$
Let $L=Ae$ be the free rank-one $A$-module generated by an even element $e$. The Lie-bracket is given by 
\[[ae,be]=ad(b)-(-1)^{|a||b|}bd(a)e   \]
for all $a,b\in A.$ Define the anchor $\alpha:L\otimes A\longrightarrow A$ by
\[
	\alpha(ae\otimes b)=a\,d(b),
	\forall a,b\in A.
\]
Then the additional compatibility condition \eqref{compatibleofh} is satisfied.
\end{exam}

\subsection{Rigid monoidal categories}
In this subsection, we study a dual version of Theorem \ref{maintheorem}.  
\begin{defi}
Fix an object $V\in \ob(\calC)$. An object $V^*\in\ob(\calC)$ is called a \emph{left duality} of $V$, if there are morphisms $b_V:I\to V\otimes V^*$ and $d_V:V^* \otimes V \to I$ such that 
\begin{align}\label{duality}
				(\id_V\otimes d_V)(b_V\otimes \id_V)=\id_V \ \mathrm{and} \ (d_V\otimes \id_{V^*})(\id_{V^*}\otimes b_V)=\id_{V^*}. 
\end{align}  
\end{defi}
The category $\calC$ is called a \emph{left rigid monoidal category} if every object $V\in \ob(\calC)$ has a left duality. 
From now on, we assume that $\calC$ is a left rigid monoidal category. 
Diagrapically, the morphisms $b_V$ and $d_V$ are represented by 
\begin{equation*}
			b_V \ = \ \dcup \ \ \text{and} \ \ d_V \ = \ \dcap \ . 
\end{equation*}
Then equation \eqref{duality} is equivalent to 
\begin{equation*}
			\begin{tikzpicture}[anchorbase]
				\draw[-, thick] (0,0) -- (0,0.6) arc(180:0:0.2) -- (0.4,0.4) arc(180:360:0.2) -- (0.8,1);
			\end{tikzpicture}
			\ =\
			\begin{tikzpicture}[anchorbase]
				\draw[-, thick] (0,0) -- (0,1);
			\end{tikzpicture}\  \ \text{and} \ \ 
			\begin{tikzpicture}[anchorbase]
				\draw[-, thick] (0, 1) -- (0, 0.4) arc(180:360:0.2) -- (0.4, 0.6) arc(180:0:0.2) -- (0.8, 0);
			\end{tikzpicture}
			\ =\
			\begin{tikzpicture}[anchorbase]
				\draw[-, thick] (0,0) -- (0,1);
			\end{tikzpicture} \ . 
\end{equation*}
By the naturality of the braiding, we have 
\begin{equation}\label{dualofu1}
			\begin{tikzpicture}[anchorbase]
				\draw[bmod] (0,0) -- (0,0.25) arc(180:0:0.25) -- (0.5,0);
				\draw[bmod] (0,0.5) -- (0.75,0);
			\end{tikzpicture} \ = \ 
			\begin{tikzpicture}[anchorbase]
				\draw[bmod] (0,0) -- (0,0.125) arc(180:0:0.125) -- (0.25,0);
				\draw[bmod] (0,0.5) -- (0.5,0);
			\end{tikzpicture}  \ \ \text{and} \ \ 
			\begin{tikzpicture}[anchorbase]
				\draw[bmod] (0,0) -- (0,0.25) arc(180:0:0.25) -- (0.5,0);
				\draw[bmod] (-0.25,0) -- (0.5,0.5);
			\end{tikzpicture} \ = \ 
			\begin{tikzpicture}[anchorbase]
				\draw[bmod] (0,0) -- (0,0.125) arc(180:0:0.125) -- (0.25,0);
				\draw[bmod] (-0.25,0) -- (0.25,0.5);
			\end{tikzpicture}   
\end{equation}
and
\begin{equation}\label{dualofu2}
			\begin{tikzpicture}[anchorbase]
				\draw[bmod] (0,0.25) -- (0,0) arc(180:360:0.25) -- (0.5,0.25);
				\draw[bmod] (-0.25,0.25) -- (0.5,-0.25);
			\end{tikzpicture} \ = \ 
			\begin{tikzpicture}[anchorbase]
				\draw[bmod] (0,0.125) -- (0,0) arc(180:360:0.125) -- (0.25,0.125);
				\draw[bmod] (0.25,-0.375) -- (-0.25,0.125);
			\end{tikzpicture}  \ \ \text{and} \ \ 
			\begin{tikzpicture}[anchorbase]
				\draw[bmod] (0,0.25) -- (0,0) arc(180:360:0.25) -- (0.5,0.25);
				\draw[bmod] (0,-0.25) -- (0.75,0.25);
			\end{tikzpicture} \ = \ 
			\begin{tikzpicture}[anchorbase]
				\draw[bmod] (0,0.125) -- (0,0) arc(180:360:0.125) -- (0.25,0.125);
				\draw[bmod] (0,-0.375) -- (0.5,0.125);
			\end{tikzpicture} \ . 
\end{equation}
The following lemma is used frequently: 
			\begin{equation*}
				\begin{tikzpicture}[anchorbase]
					\draw[bmod] (0,0) -- (0,0.25) arc(180:0:0.25) -- (0.5,0);
					\draw[bmod] (0.25,0) -- (0.5,0.5);
				\end{tikzpicture} \ = \ 
				\begin{tikzpicture}[anchorbase]
					\draw[bmod] (0,0) -- (0,0.25) arc(180:0:0.25) -- (0.5,0);
					\draw[bmod] (0,0.5) -- (0.25,0);
				\end{tikzpicture} \ \ \text{and} \ \ 
				\begin{tikzpicture}[anchorbase]
					\draw[bmod] (0,0.25) -- (0,0) arc(180:360:0.25) -- (0.5,0.25);
					\draw[bmod] (0.25,0.25) -- (0.5,-0.25);
				\end{tikzpicture} \ = \ 
				\begin{tikzpicture}[anchorbase]
					\draw[bmod] (0,0.25) -- (0,0) arc(180:360:0.25) -- (0.5,0.25);
					\draw[bmod] (0.25,0.25) -- (0,-0.25);
				\end{tikzpicture}\ .
			\end{equation*}
		
\begin{defi}
Let $f:U \to V$ be a morphism. Define $f^*:V^* \to U^*$ by 
			\begin{equation*}
				\begin{tikzpicture}[anchorbase]
					\draw[-, thick] (0,0) -- (0,0.6) arc(180:0:0.2) -- (0.4,0.4) arc(180:360:0.2) -- (0.8,1);
					\btoken{.4, .5}{west}{f};
				\end{tikzpicture} \ .
			\end{equation*}
\end{defi}
		
For any $V,W\in \ob(\calC)$, define the following morphisms $\lambda_{V,W}:W^*\otimes V^* \to (V\otimes W)^*$ and $\lambda_{V,W}^{-1}:(V\otimes W)^* \to W^* \otimes V^*$:
\begin{align*}
			\lambda_{V,W}&=(d_W \otimes \mathrm{id}_{(V\otimes W)^*}) \circ (\mathrm{id}_{W^*} \otimes d_V \otimes \mathrm{id}_{W\otimes (V\otimes W)^*}) \circ (\mathrm{id}_{W^* \otimes V^*} \otimes b_{V\otimes W}), \\
			\lambda_{V,W}^{-1}&= (d_{V\otimes W}\otimes \mathrm{id}_{W^* \otimes V^*} ) \circ (\mathrm{id}_{(V\otimes W)^*\otimes V} \otimes b_W\otimes \mathrm{id}_{V^*}) \circ (\mathrm{id}_{(V\otimes W)^*} \otimes b_V).
\end{align*}
It is straightforward to verify that $\lambda_{V,W} \circ \lambda_{V,W}^{-1}=\id_{(V\otimes W)^*}$ and $\lambda_{V,W}^{-1}\circ \lambda_{V,W}=\id_{W^* \otimes V^*}$. This implies that $(V\otimes W)^* \cong W^* \otimes V^*.$
Next, we give the dual version of the Lie monoid and Lie-Rinehart monoid in $\calC.$
\begin{defi}
A pair $(L, \delta)$ is called a \emph{Lie comonoid}  in $\calC$ if $L\in \ob(\calC)$ and $\delta : L \rightarrow L\otimes L$ is a morphism such that the following conditions are satisfied:
\[\delta= -c_{L,L} \circ \delta \ \mathrm{and} \ (\id_L\otimes \delta) \circ \delta= (\delta \otimes \id_L) \circ \delta+(c_{L,L} \otimes \id_L) \circ (\id_L \otimes \delta) \circ \delta.   \] 
\end{defi}
It can be easily known that if $(L,\delta)$ is a Lie comonoid in $\calC,$ then $\left(L^*, \delta^{\circledast}  \right)$ is a Lie monoid in $\calC$, where $\delta^{\circledast}=\delta^* \circ \lambda_{L,L}$.
Let $(L,\delta)$ and $(L',\delta')$ be two Lie comonoids in $\calC,$ a morphism $f:L\to L'$ is said to be a homomorphism of Lie comonoids if it satisfies 
\[(f\otimes f) \circ \delta =\delta' \circ f.   \]
\begin{defi}
Let $(L,\delta)$ be a Lie comonoid in $\calC$. A pair $(W,\nu )$ is called a (right) $L$-comodule if $\nu: W \to W\otimes L$ is a morphism such that the following condition is satisfied:
\[ (\nu \otimes \id_L) \circ \nu= (\id_W \otimes \delta ) \circ \nu + (\id_W \otimes c_{L,L})\circ (\nu \otimes \id_L) \circ \nu.  \]
\end{defi}
Again, if $(W,\nu)$ is an $L$-comodule, then $(W^*, \nu^{\circledast})$ is an $L^*$-module, where $\nu^{\circledast}=\nu^* \circ \lambda_{W,L}$. 
In particular, $(L,\delta)$ itself is an $L$-comodule. 
Assume that $(W,\nu)$ and $(W',\nu')$ are $L$-comodules. A morphism $f:W \to W'$ is called a homomorphism of $L$-comodules if it satisfies
\[(f \otimes \id_L) \circ \nu = \nu' \circ f.  \]
Composition of homomorphisms of comodules is also a homomorphism. Thus, all $L$-comodules form a category, denoted by {\bf CMod}$_{\calC}(L)$.
Next, we show that the category {\bf CMod}$_{\calC}(L)$ admits a monoidal structure.
		
For $(M,p), (N,q) \in $ {\bf CMod}$_{\calC}(L)$, define a morphism $p \odot q : M\otimes N \to M\otimes N \otimes L$ as follows:
\begin{equation*}
			p \odot q \ = \
			\begin{tikzpicture}[anchorbase]
				\draw[bmod] (0, 0) -- (0, .75) -- (0, .75);
				\draw[bmod] (.5,0) -- (0.5,.5)--(.25,.75);
				\draw[bmod] (0.5,.5)--(.75,.75);
				\btoken{.5, .5}{west}{q};
			\end{tikzpicture} \ + \ 
			\begin{tikzpicture}[anchorbase]
				\draw[bmod] (0.25, 0) -- (0.25, .25) -- (0, .5)-- (0, .75);
				\draw[bmod] (.25,.25) -- (0.5,.5)--(.5,.75);
				\draw[bmod] (.75,0) -- (0.75,.25)-- (0.25,.75);
				\btoken{.25, .25}{west}{p};
			\end{tikzpicture} \ . 
\end{equation*}
It is easily known that $(M \otimes N, p\odot q) \in$ {\bf CMod}$_{\calC}(L)$. Moreover, the pair $(I,\textbf{0})$ also lies in {\bf CMod}$_{\calC}(L)$. Thus, the category {\bf CMod}$_{\calC}(L)$ forms a monoidal category.
		
\begin{defi}
Let $(A,\Delta,\varepsilon)$ be a comonoid in $\calC.$ The pair $(M, \tau)$ is called a (right) $A$-\emph{comodule}, if $\tau: M \to M\otimes A$ is a morphism such that the following conditions are satisfied:
\[(\id_M \otimes \Delta) \circ \tau = (\tau \otimes \id_A) \circ \tau \ \mathrm{and}\ (\id_M \otimes \varepsilon) \circ\tau = \id_M.  \] 
\end{defi}
In particular, if we define  $\Delta^{\circledast}=  \Delta^* \circ \lambda_{A,A},    \tau^{\circledast}=\tau^* \circ \lambda_{M,A}$, then $(A^*, \Delta^{\circledast}, \varepsilon^*)$ forms a monoid, and the pair  $(M^*,\tau^{\circledast})$ is an $A^*$-module. 
		
\begin{defi}
Let $(A, \Delta, \varepsilon)$ be a cocommutative comonoid in $\calC$. Assume that $(L,\delta)$ is a Lie comonoid in $\calC$.  A \emph{Lie-Rinehart comonoid} in $\calC$ is a triple $(A,L,\mathfrak{a})$ together with an $A$-comodule structure on $L$, denoted by $(L,\mathfrak{b})$, and a morphism  $\mathfrak{a}:A \to A \otimes L$ such that the following conditions are satisfied:
\begin{align*}
				(\id_A \otimes \mathfrak{b}) \circ \mathfrak{a} &= (\mathfrak{a} \otimes \id_A) \circ \Delta,\\
				(\Delta \otimes \id_L) \circ \mathfrak{a} &=(\id_A \otimes \mathfrak{a}) \circ \Delta +(\id_A\otimes c_{L,A}) \circ (\mathfrak{a} \otimes \id_A) \circ \Delta,\\
				(\mathfrak{a} \otimes \id_L) \circ \mathfrak{a} &= (\id_A \otimes \delta) \circ \mathfrak{a}+ (\id_A \otimes c_{L,L}) \circ (\mathfrak{a} \otimes \id_L) \circ \mathfrak{a}, \\
				(\mathfrak{b} \otimes \id_L) \circ \delta &= (\id_L \otimes \mathfrak{a}) \circ \mathfrak{b} + (\id_L \otimes c_{L,A}) \circ (\delta \otimes \id_A) \circ \mathfrak{b}.
\end{align*}
\end{defi}
In particular, the triple $(A^*,L^*,\mathfrak{a}^{\circledast})$ is a Lie-Rinehart monoid, where $\mathfrak{a}^{\circledast}=\mathfrak{a}^* \circ \lambda_{A,L},$ and the $A^*$-module structure on $L^*$ is given by $\mathfrak{b}^{\circledast}=\mathfrak{b}^* \circ \lambda_{L,A}$. 
		
\begin{defi}
Let $(A,L,\mathfrak{a})$ be a Lie-Rinehart comonoid in $\calC$. The triple $(M;p,q)$ is called a weak comodule over $(A,L,\mathfrak{a})$ if 
\begin{itemize}
				\item [(1)] $(M,p)$ is an $A$-comodule
				
				\item [(2)] $(M,q)$ is an $L$-comodule 
\end{itemize}
such that 
\begin{equation*}
				\begin{tikzpicture}[anchorbase]
					\draw[bmod] (0.5, 0) -- (0.5, .25) -- (0, .75)-- (0, .75);
					\draw[bmod] (.25,.5) -- (0.5,.75);
					\draw[bmod] (.5,0.25) -- (0.75,.5)-- (0.75,.75);
					\btoken{.25, .5}{east}{p};
					\btoken{.5, .25}{west}{q};
				\end{tikzpicture}\ = \
				\begin{tikzpicture}[anchorbase]
					\draw[bmod] (0.25, 0) -- (0.25, .25) -- (0, .5)-- (0, .75);
					\draw[bmod] (.25,.25) -- (0.75,.75);
					\draw[bmod] (.5,0.5) -- (0.25,.75);
					\btoken{.25, .25}{west}{p};
					\btoken{.5, .5}{west}{\mathfrak{a}};
				\end{tikzpicture} \ + \   
				\begin{tikzpicture}[anchorbase]
					\draw[bmod] (0.5, 0) -- (0.5, .25) -- (0, .75)-- (0, 1);
					\draw[bmod] (.25,.5) -- (0.75,1);
					\draw[bmod] (.5,0.25) -- (0.75,.5)-- (0.75,.75)-- (0.5,1);
					\btoken{.5, .25}{west}{p};
					\btoken{.25, .5}{east}{q};
				\end{tikzpicture} \ . 
\end{equation*}
\end{defi}

Note that $(M^*;p^{\circledast},q^{\circledast})$ is a weak module over the Lie-Rinehart monoid $(A^*,L^*,\mathfrak{a}^{\circledast})$, where $p^{\circledast}=p^* \circ \lambda_{M,A}$ and $q^{\circledast}=q^* \circ \lambda_{M,L}$.
		
A homomorphism of weak comodules over $(A,L,\mathfrak{a})$ is both a homomorphism of $L$-comodules and a homomorphism of $A$-comodules. Moreover, the composition of homomorphisms of weak comodules is also a homomorphism.  Thus, all weak comodules over the Lie-Rinehart comonoid form a category, denoted by {\bf WCMod}$^{\calC}_A(L)$.
		
\begin{defi}
Let $(A,\Delta, \varepsilon, L,\delta, \mathfrak{a})$ be a Lie-Rinehart comonoid and $(\mathcal{L}, \mathfrak{d})$ be a Lie comonoid in $\calC$. A morphism $\mathcal{J}:A\otimes \mathcal{L} \to L$ is called a \emph{cocrossed homomorphism} if it satisfies 
\begin{equation*}
				\begin{tikzpicture}[anchorbase]
					\draw[bmod] (0, 0) -- (0.25, 0.5) -- (0.25, 1)--(0,1.5);
					\draw[bmod] (0.5, 0) -- (0.25, 0.5);
					\draw[bmod] (0.5, 1.5) -- (0.25, 1);
					\btoken{0.25, 1}{west}{\delta};
					\btoken{0.25, 0.5}{west}{\mathcal{J}};
				\end{tikzpicture} \ = \ 
				\begin{tikzpicture}[anchorbase]
					\draw[bmod] (0.25, 0) -- (0.25, .25) -- (0, .5)--(0,1)--(.25,1.25)--(.25,1.5);
					\draw[bmod] (0.25, 0.
					25) --(0.75,.75)-- (.75, 1.5);
					\draw[bmod] (0.5, 0) -- (0.5, 1) -- (0.25, 1.25);
					\btoken{0.25, 1.25}{east}{\mathcal{J}};
					\btoken{0.25, 0.25}{east}{\mathfrak{a}};
				\end{tikzpicture}\ - \ 
				\begin{tikzpicture}[anchorbase]
					\draw[bmod] (0.25, 0) -- (0.25, 0.25) -- (0, 0.5)--(0,.75)--(.5,1.25)--(.5,1.5);
					\draw[bmod] (0.25, 0.25) -- (0.5, 0.5) -- (0.5, .75)--(0,1.25)--(0,1.5);
					\draw[bmod] (0.75, 0) -- (0.75, 1) -- (0.5, 1.25);
					\btoken{0.5, 1.25}{west}{\mathcal{J}};
					\btoken{0.25, 0.25}{east}{\mathfrak{a}};
				\end{tikzpicture} \ + \ 
				\begin{tikzpicture}[anchorbase]
					\draw[bmod] (0, 0) -- (0, 0.25) -- (-0.25, 0.5)--(-0.25,1)--(0,1.25)--(0,1.5);
					\draw[bmod] (0, 0.25) -- (0.75,1.25)--(0.75,1.5);
					\draw[bmod] (0.75, 0.25) -- (1, 0.5) -- (1, 1)--(0.75,1.25);
					\draw[bmod] (0.75, 0) -- (0.75, 0.25) --(0,1.25);
					\btoken{0, 1.25}{east}{\mathcal{J}};
					\btoken{0.75, 1.25}{west}{\mathcal{J}};
					\btoken{0, 0.25}{east}{\Delta};
					\btoken{0.75, 0.25}{west}{\mathcal{D}};
				\end{tikzpicture} \ .
\end{equation*}
\end{defi}
		
\begin{lem}\label{duallem}
Let $(A,\Delta, \varepsilon, L,\delta, \mathfrak{a})$ be a Lie-Rinehart comonoid and $(\mathcal{L}, \mathfrak{d})$ be a Lie comonoid in $\calC$.  Assume that $(M;p,q)$ is a weak comodule over $(A,L,\mathfrak{a})$. For any $\mathcal{L}$-comodule $(V,\psi)$ and cocrossed homomorphism $\mathcal{J}: A \otimes \mathcal{L} \to L$, define 
\begin{equation}\label{boxdot}
				q\vartriangleleft \psi =  \ 
				\begin{tikzpicture}[anchorbase]
					\draw[bmod] (0.5, 0) -- (0.5, 0.25) -- (0, .75)--(0,1.25);
					\draw[bmod] (0.5, .25) -- (1,.75)--(1,1.25);
					\draw[bmod] (1, 0) -- (1, 0.25) -- (.5, 0.75)--(.5,1.25);
					\btoken{0.5, .25}{west}{q};
				\end{tikzpicture} \ + \ 
				\begin{tikzpicture}[anchorbase]
					\draw[bmod] (0.25, 0) -- (0.25, 0.25) -- (0, 0.5)--(0,1.25);
					\draw[bmod] (0.25, 0.25) -- (1,1)--(1,1.25);
					\draw[bmod] (1, 0) -- (1, 0.25) -- (.5, 0.75)--(.5,1.25);
					\draw[bmod] (1, 0.25) -- (1.25, 0.5) --(1.25,.75)--(1,1);
					\btoken{1, .25}{west}{\psi};
					\btoken{0.25, .25}{east}{p};
					\btoken{1, 1}{west}{\mathcal{J}};
				\end{tikzpicture} \ \ \ \ \text{and} \ \ \ \ 
				p \vartriangleleft \psi  = \ 
				\begin{tikzpicture}[anchorbase]
					\draw[bmod] (0.5, 0) -- (0.5, 0.25) -- (0, .75)--(0,1.25);
					\draw[bmod] (0.5, .25) -- (1,.75)--(1,1.25);
					\draw[bmod] (1, 0) -- (1, 0.25) -- (.5, 0.75)--(.5,1.25);
					\btoken{0.5, .25}{east}{p};
				\end{tikzpicture}  \ .
\end{equation}
Then, the triple $(M\otimes V; p \vartriangleleft \psi, q \vartriangleleft \psi)$ forms a weak comodule over $(A,L,\mathfrak{a})$. 
\end{lem}
\begin{proof}
The proof is similar to Lemma \ref{main1}. 
\end{proof}
		
With the preparations above, we are now in a position to present the dual version of Theorem \ref{maintheorem}. 
		
\begin{theor}\label{rightmodule}
Let $(A,L,\mathfrak{a})$ be a Lie-Rinehart comonoid in $\calC$. For any Lie comonoid $(\mathcal{L}, \mathfrak{d})$ in $\calC$ and cocrossed homomorphism $\mathcal{J}:A\otimes \mathcal{L} \to L$, the category {\bf WCMod}$^{\calC}_A(L)$ is a strict right module category over the monoidal category {\bf CMod}$_{\calC}(\mathcal{L})$. Specifically, for any $(W,\nu), (W',\nu') \in$ {\bf CMod}$_{\calC}(\mathcal{L})$ and $(M;p,q), (M';p',q')\in$ {\bf WCMod}$^{\calC}_A(L)$, we have a bifunctor 
\begin{align*}
				\mathcal{F}_{\mathcal{J}}:   \textbf{WCMod}^{\calC}_A(L) \times \textbf{CMod}_{\calC}(\mathcal{L}) &\to  \textbf{WCMod}^{\calC}_A(L) \\
				\Big( (M;p,q), (W,\nu) \Big) & \mapsto (M\otimes W; p \vartriangleleft \nu, q \vartriangleleft \nu )\\
				(\varphi, \phi) & \mapsto \varphi \otimes \phi,
\end{align*}
where $\varphi: W \to W'$ is a homomorphism of $\mathcal{L}$-comodules, and $\phi:M\to M'$ is a homomorphism of weak comodules of $(A,L)$. 
\end{theor}
\begin{proof}
The proof is similar to Theorem \ref{maintheorem}.
\end{proof}
		
Coupled with Lie-Rinehart monoids and Lie-Rinehart comonoids, we now introduce the concept of Lie-Rinehart bimonoids. 
\begin{defi}
The tuple $(A,L,\alpha,\mathfrak{a})$ is called a  \emph{Lie-Rinehart double monoid} if $(A,L,\alpha)$ is a Lie-Rinehart monoid and $(A,L,\mathfrak{a})$ is a Lie-Rinehart comonoid.
\end{defi}
The tuple $(M;\rho,\gamma;p,q)$ is called a \emph{weak double module} over the Lie-Rinehart double monoid if $(M;\rho,\gamma)\in \textbf{WMod}^{\calC}_A(L)$ and $(M;p,q)\in \textbf{WCMod}^{\calC}_A(L)$.   A homomorphism of weak double modules over the Lie-Rinehart double monoid  $(A,L,\alpha,\mathfrak{a})$ is both a homomorphism of weak modules and a homomorphism of weak comodules. Clearly, all weak double modules over a Lie-Rinehart double monoid form a category, denoted by $\textbf{WDMod}^{\calC}_A(L)$. 
Now we arrive at the main result of this subsection. 
\begin{theor}
Let $(A,L,\alpha,\mathfrak{a})$ be a Lie-Rinehart double monoid. Assume that $(\mathfrak{L},\mathfrak{u})$ is a Lie monoid and $(\mathcal{L},\mathfrak{d})$ is a Lie comonoid. For any crossed homomorphism $\mathcal{H}: L \rightarrow \mathfrak{L}\otimes A$, and cocrossed homomorphism $\mathcal{J}:A\otimes \mathcal{L} \to L$, the category $\textbf{WDMod}^{\calC}_A(L)$ is a strict  bimodule category over  $\Big(\mathrm{ \bf Mod}_{\calC}(\mathfrak{L}), \mathrm{ \bf CMod}_{\calC}(\mathcal{L})\Big)$. 
\end{theor}
		
\begin{proof}
By Theorem \ref{maintheorem} and Theorem \ref{rightmodule}, we obtain that $\textbf{WDMod}^{\calC}_A(L)$ is a left module category over $\mathrm{\bf Mod}^{\calC}_A(L)$ and right module category over $\mathrm{\bf CMod}^{\calC}_A(L)$. The actions are given by 
\begin{align*}
				\mathrm{\bf Mod}_{\calC}( \mathfrak{L}) \times \mathrm{\bf WDMod}^{\calC}_A(L)& \to \mathrm{\bf WDMod}^{\calC}_A(L)\\
				\Big((V,\psi), (M;\rho,\gamma;p,q)      \Big)& \mapsto \Big(V\otimes M; \psi \vartriangleright \rho, \psi \vartriangleright \gamma; \id_V\otimes p, \id_V \otimes q\Big)
\end{align*}
and
\begin{align*}
				\mathrm{\bf WDMod}^{\calC}_A(L) \times \mathrm{\bf CMod}_{\calC}(\mathcal{L}) &\to \mathrm{\bf WDMod}^{\calC}_A(L)\\
				\Big((M;\rho,\gamma;p,q), (W,\nu)   \Big) & \mapsto \Big(M\otimes W; \rho \otimes \id_W, \gamma\otimes \id_W;p\vartriangleleft \nu, q \vartriangleleft \nu \Big).
\end{align*}
To show that the category $\mathrm{\bf WDMod}^{\calC}_A(L)$ is a strict bimodule category, it suffices to prove 
\begin{align*}
				\Big( (\psi\vartriangleright \rho) \otimes \id_W, (\psi \vartriangleright \gamma)\otimes \id_W&, (\id_V\otimes p ) \vartriangleleft \nu, (\id_V \otimes q) \vartriangleleft \nu    \Big) \\
				&=\\
				\Big( \psi \vartriangleright (\rho \otimes \id_W), \psi \vartriangleright (\gamma \otimes \id_W)&,  \id_V\otimes(p \vartriangleleft \nu), \id_V\otimes (q\vartriangleleft \nu)    \Big).
\end{align*} 
However, this equation is deduced from equations \eqref{newhom} and \eqref{boxdot}. 
\end{proof}

\section{Applications}\label{app}
In this section, we work with the symmetric monoidal category $s\mathcal{V}ec(\bbC)$. Under this assumption, the notions of monoids, commutative monoids, Lie monoids and Lie-Rinehart monoids specialize respectively to  unital associative superalgebras,  unital commutative associative superalgebras,  Lie superalgebras and Lie-Rinehart superalgebras over $\bbC$.
Recall that a unital associative algebra $A$ is called a superalgebra if it is a $\bbZ_2$-graded vector space such that $A_i A_j \subset A_{i+j}$ for all $i,j \in \bbZ_2$. A vector space $M$ is called an $A$-module, if $M$ is a $\bbZ_2$-graded vector space such that $A_i \cdot M_j \subset M_{i+j}$. Let $M$ and $M'$ be two $A$-modules. A map $f:M\to M'$ is called a homomorphism of $A$-modules if $f$ is even and $f(a \cdot m)=a\cdot f(m)$ for all $a\in A,m\in M.$
A superalgebra $A$ is called supercommutative if $ab=(-1)^{\mid a \mid \mid b\mid}ba, \ \forall a,b \in A$. For any superalgebra $A$, a linear map $f:A \to A$ of degree $\mid f\mid$ is a \emph{superderivation} provided that
\[f(ab)=f(a)b +(-1)^{\mid a \mid \mid f \mid } af(b)  \]
for all $a,b\in A.$ We denote by $s\mathrm{Der}(A)_i$ the set of all superderivations of degree $i$ on $A$. 
Clearly, the space $s\mathrm{Der}(A)=s\mathrm{Der}(A)_0 \oplus s\mathrm{Der}(A)_1$ admits a natural Lie superalgebra  structure with bracket defined by 
\[[f,g]:=f \circ g -(-1)^{\mid f \mid \mid g \mid }g \circ f, \ \forall f,g\in s\mathrm{Der}(A).\]
In what follows, we assume that $A$ is a supercommutative superalgebra. 
In this case, the Lie superalgebra $s\mathrm{Der}(A)$ becomes an $A$-module via the action 
\[(a\cdot f)(b):=af(b), \  \forall f\in s\mathrm{Der}(A), a,b\in A.\]
		
\begin{defi}[Chapter $2.1$, \cite{Che95}]
Let $(L, [,])$ be a Lie superalgebra. A \emph{Lie-Rinehart superalgebra} is a triple $(A,L,\tilde{\alpha})$ equipped with an $A$-module structure on $L$, and an even map $\tilde{\alpha}: L \to s\mathrm{Der}(A)$ which is simultaneously a homomorphism of Lie superalgebras and a homomorphism of $A$-modules such that 
\begin{align}\label{lierinehartalgebra}
[x, a \cdot y] =(-1)^{\mid a \mid \mid x \mid}a \cdot [x,y] + \tilde{\alpha}(x)(a) \cdot y, \ \forall x,y,\in L, \ a\in A.   
\end{align}
\end{defi}
\begin{exam}
The triple $(A, s\mathrm{Der}(A),\id_{s\mathrm{Der}(A)} )$ forms a Lie-Rinehart superalgebra. 
\end{exam}
\begin{exam}
Let $M$ be an $A$-module. A first-order differential operator on $M$ is a pair $(D, \sigma)$, where $D:M\to M$ is a linear map of order $| D|$, and $\sigma:=\sigma_D \in s\mathrm{Der}(A)$ satisfying
\[D(a\cdot m)=(-1)^{\mid a \mid \mid D \mid } a \cdot D(m) + \sigma(a) \cdot m,   \]
for all $a\in A,m\in M.$ Denote by $\mathcal{D}(M)_i=\{(D,\sigma) \mid  \ | D | = | \sigma | =i \}$. Then $\mathcal{D}(M):= \mathcal{D}(M)_0 \oplus \mathcal{D}(M)_1$ forms a Lie superalgebra if we define
\[[(D_1,\sigma_1), (D_2,\sigma_2) ]:= \Big(D_1 \circ D_2 - (-1)^{\mid D_1 \mid \mid D_2 \mid } D_2 \circ D_1, \sigma_1 \circ \sigma_2 -(-1)^{\mid \sigma_1 \mid \mid \sigma_2 \mid} \sigma_2 \sigma_1 \Big),     \]
where $(D_1,\sigma_1), (D_2,\sigma_2) \in \mathcal{D}(M).$ Define a map 
\begin{align*}
\mathrm{pr}: \mathcal{D}(M) &\to s\mathrm{Der}(A) \\
(D,\sigma) &\mapsto \sigma.
\end{align*}
Then, the triple $(A, \mathcal{D}(M), \mathrm{pr})$ is a Lie-Rinehart superalgebra. 
\end{exam}

We now establish the relationship between our definition of the Lie-Rinehart monoid in the monoidal category $s\mathcal{V}ec(\bbC)$ and the classical notion of the Lie-Rinehart superalgebra. 
\begin{lem}\label{liemonoid}
Lie-Rinehart monoids in $s\mathcal{V}ec(\bbC)$ are precisely Lie-Rinehart superalgebras.
\end{lem}
\begin{proof}
Let $(A,L,\alpha)$ be a Lie-Rinehart monoid in $s\mathcal{V}ec(\bbC)$. Define an even map $\tilde{\alpha}:L \rightarrow \mathrm{End}(A)$ by 
\[\tilde{\alpha}(x)(a):= \alpha( x\otimes a), \ \forall x\in L,a\in A. \]
By equation \eqref{derivation}, we have $\tilde{\alpha}(x) \in s\mathrm{Der}(A)_{|x|}$.  
Equations \eqref{homofa} and \eqref{liehom} imply that $\tilde{\alpha}$ is a homomorphism of $A$-modules and Lie superalgebras, respectively. The condition \eqref{lierinehartalgebra} holds due to the equation \eqref{compatible1}. Thus, $(A,L,\tilde{\alpha})$ forms a Lie-Rinehart superalgebra. 
\end{proof}
		
Applying Theorem \ref{maintheorem} to the case where $\calC=s\mathcal{V}ec(\bbC)$,  we obtain the following result. 
\begin{theor}\label{superver}
Let $(A,L,\alpha)$ be a Lie-Rinehart superalgebra and $\mathfrak{L}$ be a Lie superalgebra. Then any even crossed homomorphism $\mathcal{H}:L\rightarrow \mathfrak{L}\otimes A$ induces a left module category structure of the category of weak representations ${\bf WMod}_A(L)$ over the monoidal category ${\bf Mod}(\mathfrak{L})$:
\[F_{\mathcal{H}}: {\bf Mod}(\mathfrak{L}) \times {\bf WMod}_A(L) \rightarrow {\bf WMod}_A(L).   \]
\end{theor}
		
\begin{rmk}
When a Lie-Rinehart superalgebra is specialized to an ordinary Lie-Rinehart algebra, we recover   \cite[Theorem $3.26$]{PSTZ23}.
\end{rmk}
		
As noted in \cite{PSTZ23}, it is generally difficult to construct non-trivial crossed homomorphisms for a given Lie-Rinehart superalgebra. In what follows, we present several examples of crossed homomorphisms for Lie-Rinehart superalgebras in some special cases.
In particular, when the Lie superalgebra $\mathfrak{L}$ is taken to be  trivial, any crossed homomorphism $\mathcal{H}: L \to A$ reduces to an element of $\mathrm{sDer}(L,A)_0$. 
For $L=s\mathrm{Der}(A),$ we have the following result.
\begin{coro}
Any $\mathcal{H} \in s\mathrm{Der}(s\mathrm{Der}(A),A)_0$ induces a functor 
\begin{align*}
\mathcal{F}_{\mathcal{H}}: {\bf WMod}_A(s\mathrm{Der}(A)) &\to {\bf WMod}_A(s\mathrm{Der}(A))\\
(M;\rho,\gamma) & \mapsto \Big(M;\rho, \rho \circ (\mathcal{H}\otimes \id_A) +\gamma \Big).
\end{align*} 
\end{coro}
Let $P_m^{\pm}=\bbC[t_1^{\pm 1},t_2^{\pm 1},\dots,t_m^{\pm 1}]$ denote the Laurent polynomial ring, and let $\Lambda_n=\bbC[\theta_1, \theta_2,\dots, \theta_n]$ be the exterior algebra, characterized by the relations $\theta_i\theta_j=-\theta_j \theta_i$ for all $ i,j=1,2,\dots,n$. Note that $\Lambda_n$ is finite-dimensional. Set $A_{m,n}= P_m^{\pm}\otimes \Lambda_n$. Then $A_{m,n}$ carries a natural superalgebra structure with the $\bbZ_2$-grading defined by  $|t_i^{\pm 1}|=0 $ and $ |\theta_j|=1 $ for all $  i=1,2,\dots,m,\ j=1,2,\dots,n.$ The Lie superalgebra of superderivations $\mathrm{sDer}(A_{m,n})$ is the \emph{Witt superalgebra} $\mathcal{W}(m,n ).$  
		
For $\boldsymbol{\alpha}=(\alpha_1,\alpha_2,\dots,\alpha_m)\in \bbZ^{m}$ and $I=\{i_1,i_2,\dots, i_k\} \subset \{1,2,\dots,n\}$, let $t^{\boldsymbol{\alpha}}=t_1^{\alpha_1}  t_2^{\alpha_2}\cdots t_m^{\alpha_m},$ and ${\theta}_{I}=\theta_{i_1} \theta_{i_2}\cdots \theta_{i_k}.$ Moreover, we set $\theta_{\emptyset}=1.$ For any $f,g\in A_{m,n},$ we define the following superderivations $f\frac{\partial}{\partial t_i}$ and $f\frac{\partial}{\partial \theta_j}$:
\[ f\frac{\partial}{\partial t_i}(g)=f\frac{\partial g}{\partial t_i}, \        f\frac{\partial}{\partial \theta_j}(g)=f\frac{\partial g}{\partial \theta_j}. \]
Then the Witt superalgebra $\mathcal{W}(m,n)$ has a basis as follows:
\[ \Big\{t^{\boldsymbol{\alpha}} \theta_{I} \frac{\partial}{\partial t_i}, \ t^{\boldsymbol{\alpha}} \theta_{I} \frac{\partial}{\partial \theta_j} \ | \ \boldsymbol{\alpha} \in \bbZ^{m}, I\subset  \{1,2,\dots,n\}, i \in \{1,2,\dots,m\}, j \in \{ 1,2, \dots,n\}  \Big\}        \]
with the bracket defined by 
\[ [t^{\boldsymbol{\alpha}_1} \theta_{I_1} \partial_1, t^{\boldsymbol{\alpha}_2} \theta_{I_2} \partial_2 ]= t^{\boldsymbol{\alpha}_1} \theta_{I_1} \partial_1(t^{\boldsymbol{\alpha}_2} \theta_{I_2}) \partial_2 -(-1)^{(|I_1|+|\partial_1|)(|I_2|+|\partial_2|)}     t^{\boldsymbol{\alpha}_2} \theta_{I_2} \partial_2 (t^{\boldsymbol{\alpha}_1} \theta_{I_1}) \partial_1,  \]
where $\partial_1, \  \partial_2 \in \left\{\frac{\partial}{\partial t_1}, \frac{\partial}{\partial t_2}, \dots, \frac{\partial}{\partial t_m}, \frac{\partial}{\partial \theta_1}, \frac{\partial}{\partial \theta_2}, \dots, \frac{\partial}{\partial \theta_n}\right\}.$
		
Let $\mathfrak{gl}(m,n)$ be the general linear Lie superalgebra, with basis $\{ E_{ij} \mid 1 \leq i,j \leq m+n \}$. 
Define a map 
\begin{align*}
\mathcal{H}: \mathcal{W}(m,n) &\rightarrow \mathfrak{gl}(m,n)\otimes A_{m,n} \\
t^{\boldsymbol{\alpha}}\theta_I\frac{\partial}{\partial t_i} & \mapsto \sum\limits_{s=1}^m E_{si}\otimes \frac{\partial}{\partial t_s}(t^{\boldsymbol{\alpha}}\theta_I) + (-1)^{\mid I \mid -1} \sum\limits_{l=1}^n E_{m+l,i}\otimes \frac{\partial}{\partial \theta_l}(t^{\boldsymbol{\alpha}}\theta_I),\\
t^{\boldsymbol{\alpha}}\theta_I\frac{\partial}{\partial \theta_j}& \mapsto \sum\limits_{s=1}^m E_{s,m+j}\otimes \frac{\partial}{\partial t_s}(t^{\boldsymbol{\alpha}}\theta_I) +(-1)^{\mid I \mid -1} \sum\limits_{l=1}^n E_{m+l,m+j}\otimes \frac{\partial}{\partial \theta_l}(t^{\boldsymbol{\alpha}}\theta_I). 
\end{align*}
According to \cite{LX23}, the map $\mathcal{H}$ is a crossed homomorphism.
		
Let $A_m$ denote the $m$-th Weyl algebra, generated by $x_i^{\pm 1}, y_i $ for all $i=1,2,\dots, m$, subject to the relations:
\begin{align*}
x_ix_i^{-1}&=x_i^{-1}x_i=1, \ \forall i=1,2,\dots,m,\\
x_ix_j&=x_jx_i, \ y_iy_j=y_jy_i, \ \forall i,j=1,2,\dots,m,\\
y_ix_i&=x_iy_i+1, x_iy_j=y_jx_i,\ \forall 1 \leq i\neq j\leq m.
\end{align*}
Let $C_n$ denote the $n$-th Clifford algebra, generated by $\psi_i,\xi_i $ for all $ i=1,2,\dots,n,$ with the following relations:
\begin{align*}
			\xi_i\xi_j+\xi_j\xi_i&=0, \ \forall i,j=1,2,\dots,n,\\
			\psi_i\psi_j+\psi_j\psi_i&=0, \ \forall i,j=1,2,\dots,n,\\
			\psi_i\xi_i+\xi_i\psi_i&=1, \ \forall i=1,2,\dots,n,\\
			\xi_i\psi_j+\psi_j\xi_i&=0, \ \forall 1 \leq i\neq j \leq n.
\end{align*}
Set $K_{m,n}=A_m\otimes C_n.$ We call $K_{m,n}$ the Clifford-Weyl superalgebra, where the $\bbZ_2$-grading is given by  $|x_{i}^{\pm 1}|=|y_i|=0, \ |\xi_j|=|\psi_j|=1 $ for all $ i=1,2,\dots,m, j=1,2, \dots,n.$
Let $(P,\rho)$ be a representation of the Clifford-Weyl superalgebra $K_{m,n}$. Then the pair $(P,\rho|_{\mathcal{W}(m,n)})$ is also a representation of $\mathcal{W}(m,n)$. Moreover, we can show that $(P;\rho|_{A_{m,n}}, \rho|_{\mathcal{W}(m,n)} )$ is a weak representation of the Lie-Rinehart superalgebra $\left(A_{m,n}, \mathcal{W}(m,n), \id_{\mathcal{W}(m,n)}\right)$.  
Applying Theorem \ref{superver}, we get
\begin{coro}
We have a functor 
\begin{align*}
				F^{\rho}_{\mathcal{H}}: {\bf Mod}(\mathfrak{gl}(m,n)) &\rightarrow {\bf WMod}_{A_{m,n}}(\mathcal{W}(m,n))\\
				(V,\theta) & \mapsto \Big(V\otimes P; \theta\vartriangleright\rho|_{A_{m,n}}, \theta\vartriangleright \rho|_{\mathcal{W}(m,n)} \Big).
\end{align*}
\end{coro}
\begin{rmk}
This functor is an analogue of the Shen-Larsson functor for $(\mathcal{W}(m,n),\mathfrak{gl}(m,n))$.
\end{rmk}

\section{Appendix}
For convenience, we recall the method of graphical calculus (or string diagrams) in the monoidal category \cite{KT95,TV17,Kas95}. In what follows, we work in a symmetric linear monoidal category $\calC$, i.e., $\mathrm{Hom}_{\calC}(X,Y)$ is a vector space for all objects $X,Y\in \ob(\calC)$, the composition of morphisms and tensor product of morphisms are bilinear. 
		
We represent a morphism $f: V \rightarrow W$ in $\calC$ by a dot with two vertical lines upwards as follows:
\begin{equation*}\label{repoff}
			\begin{tikzpicture}[anchorbase]
				\draw[bmod] (0, 0) -- (0, .5);
				\btoken{0, .25}{east}{f};
			\end{tikzpicture} \ . 
\end{equation*}
		
In particular, we omit the vertical line connecting to the unit object $I$ and represent the identity map by a vertical line. 
The picture for the composition of $f:U\rightarrow V$ and $g:V\rightarrow W$ is obtained by putting the picture of $g$ on top of the picture of $f$:    
		\begin{equation*}\label{compositionoffg}
			\begin{tikzpicture}[anchorbase]
				\draw[bmod] (0, 0) -- (0, 1);
				\btoken{0, .5}{east}{g \circ f};
			\end{tikzpicture} \ = \ 
			\begin{tikzpicture}[anchorbase]
				\draw[bmod] (0, 0) -- (0, 1);
				\btoken{0, .25}{east}{ f};
				\btoken{0, .75}{east}{ g};
			\end{tikzpicture} \ . 
		\end{equation*}
The tensor product of two morphisms $f:U \rightarrow V$ and $h:W\rightarrow Z$ is represented by circles placed side by side:
		\begin{equation*}\label{tensoroffh}
			\begin{tikzpicture}[anchorbase]
				\draw[bmod] (0, 0) -- (0, 1);
				\btoken{0, .5}{east}{f \otimes h};
			\end{tikzpicture} \ = \ 
			\begin{tikzpicture}[anchorbase]
				\draw[bmod] (0, 0) -- (0, 1);
				\draw[bmod] (0.5, 0) -- (0.5, 1);
				\btoken{0, .5}{east}{ f};
				\btoken{0.5, .5}{west}{h };
			\end{tikzpicture} \ . 
		\end{equation*}
Thus, by equation \eqref{defoftensor}, we have
		\begin{equation*}\label{pull}
			\begin{tikzpicture}[anchorbase]
				\draw[bmod] (0, 0) -- (0, 1);
				\draw[bmod] (0.5, 0) -- (0.5, 1);
				\btoken{0, .5}{east}{ f};
				\btoken{0.5, .5}{west}{h };
			\end{tikzpicture} \ = \ 
			\begin{tikzpicture}[anchorbase]
				\draw[bmod] (0, 0) -- (0, 1);
				\draw[bmod] (0.5, 0) -- (0.5, 1);
				\btoken{0, .25}{east}{ f};
				\btoken{0.5, .75}{west}{h };
			\end{tikzpicture} \ = \ 
			\begin{tikzpicture}[anchorbase]
				\draw[bmod] (0, 0) -- (0, 1);
				\draw[bmod] (0.5, 0) -- (0.5, 1);
				\btoken{0, .75}{east}{f};
				\btoken{0.5, .25}{west}{h};
			\end{tikzpicture} \ . 
		\end{equation*}
		
This leads to the following ``partial isotopy principle": for any figure presenting a morphism of $\calC$, the part of the figure lying to the left (or to the right) of a vertical line may be pushed up or down without changing the corresponding morphism in $\calC$. We shall use this principle frequently and without any further explanation in the sequel.
		
For any $V,W\in \ob(\calC)$, the braiding $c_{V,W}$ is represented by
		$\begin{tikzpicture}[anchorbase]
			\draw[bmod] (0, 0) -- (0.5, .5);
			\draw[bmod] (0.5, 0) -- (0, .5);
		\end{tikzpicture}\ .$
		Then equations \eqref{braid1}--\eqref{braid2} are represented by
		$\begin{tikzpicture}[anchorbase]
			\draw[bmod] (0, 0) -- (0.5, .5);
			\draw[bmod] (0.5, 0) -- (0, .5);
		\end{tikzpicture} \ = \ 
		\begin{tikzpicture}[anchorbase]
			\draw[bmod] (0, 0) -- (0.25, .5);
			\draw[bmod] (0.5, 0) -- (0, .5);
			\draw[bmod] (0.25, 0) -- (0.5, .5);
		\end{tikzpicture} \ \ \text{and} \ \ 
		\begin{tikzpicture}[anchorbase]
			\draw[bmod] (0, 0) -- (0.5, .5);
			\draw[bmod] (0.5, 0) -- (0, .5);
		\end{tikzpicture} \ = \ 
		\begin{tikzpicture}[anchorbase]
			\draw[bmod] (0, 0) -- (0.5, .5);
			\draw[bmod] (0.25, 0) -- (0, .5);
			\draw[bmod] (0.5, 0) -- (0.25, .5);
		\end{tikzpicture} \ . $
The symmetry of the braiding is represented by 
		\begin{equation*}
			\begin{tikzpicture}[anchorbase]
				\draw[bmod] (0, 0) -- (0.5, .5);
				\draw[bmod] (0.5, 0) -- (0, .5);
				\draw[bmod] (0, 0) -- (0.5, -.5);
				\draw[bmod] (0.5, 0) -- (0, -.5);
			\end{tikzpicture} \ = \ 
			\begin{tikzpicture}[anchorbase]
				\draw[bmod] (0, 0) -- (0, 1);
				\draw[bmod] (0.5, 0) -- (0.5, 1);
			\end{tikzpicture} \ . 
		\end{equation*}
		
For morphisms $f:V\rightarrow V'$ and $g:W\rightarrow W'$, the natural isomorphisms of the braiding are represented by 
		\begin{equation*}\label{natofbraid}
			\begin{tikzpicture}[anchorbase]
				\draw[bmod] (0, 0) -- (0, .5)--(.5,1);
				\draw[bmod] (0.5, 0) -- (0.5, .5)--(0,1);
				\btoken{0, .25}{east}{f};
				\btoken{0.5, .25}{west}{g};
			\end{tikzpicture} \ = \ 
			\begin{tikzpicture}[anchorbase]
				\draw[bmod] (0, 0) -- (0.5, .5)--(.5,1);
				\draw[bmod] (0.5, 0) -- (0, .5)--(0,1);
				\btoken{0.5, .75}{west}{f};
				\btoken{0, .75}{east}{g};
			\end{tikzpicture} \ . 
		\end{equation*}
		
We denote this equation by $(c_{V,W},f,g)$.
The following lemmas are frequently used in our proof.
\begin{lem}\label{identityofbraid}
For any morphism $f:V\rightarrow W\otimes U$ and $Z\in \ob(\calC)$, we have
			\begin{equation*}
				\begin{tikzpicture}[anchorbase]
					\draw[bmod] (0.5, 0) -- (0.5, .5)--(0,1);
					\draw[bmod] (0.5, 0.5) -- (1, 1);
					\draw[bmod] (0.75, 0) -- (.25, .5)--(.25,1);
					\btoken{0.5, .5}{west}{f};
				\end{tikzpicture} \ = \ 
				\begin{tikzpicture}[anchorbase]
					\draw[bmod] (0.5, 0) -- (0.5, .5)--(0,1);
					\draw[bmod] (0.5, 0.5) -- (1, 1);
					\draw[bmod] (0.75, 0) -- (.75, 1);
					\btoken{0.5, .5}{east}{f};
				\end{tikzpicture} \ \ \text{and} \ \ 
				\begin{tikzpicture}[anchorbase]
					\draw[bmod] (0.5, 0) -- (0.5, .5)--(0,1);
					\draw[bmod] (0.5, 0.5) -- (1, 1);
					\draw[bmod] (0.25, 0) -- (.75, .5)--(.75,1);
					\btoken{0.5, .5}{east}{f};
				\end{tikzpicture} \ = \ 
				\begin{tikzpicture}[anchorbase]
					\draw[bmod] (0.5, 0) -- (0.5, .5)--(0,1);
					\draw[bmod] (0.5, 0.5) -- (1, 1);
					\draw[bmod] (0.25, 0) -- (.25, 1);
					\btoken{0.5, .5}{west}{f};
				\end{tikzpicture}\ .
			\end{equation*}
\end{lem}
\begin{proof}
We just prove the first equality. By the natural isomorphism of the braiding $(c_{V,Z},f,\id_Z)$ and the symmetry of the braiding, it can be obtained. 
\end{proof}
		
\begin{lem}\label{natof1}
For any morphisms $f:V\rightarrow W\otimes U$ and $g:X\otimes Y \rightarrow Z$, we have 
			\begin{equation*}
				\begin{tikzpicture}[anchorbase]
					\draw[bmod] (0.25, 0) -- (0.5, .25)--(0.5,.5)--(.25,.75)--(.25,1);
					\draw[bmod] (0.5, 0) -- (.25, .25)--(.25,.5)--(0,.75)--(0,1);
					\draw[bmod] (0.25, 0.5) -- (.5, .75)--(.5,1);
					\draw[bmod] (0.75, 0) -- (.5, .25);
					\btoken{0.25, .5}{east}{f};
					\btoken{0.5, .25}{west}{g};
				\end{tikzpicture} \ = \ 
				\begin{tikzpicture}[anchorbase]
					\draw[bmod] (0, 0) -- (0, .25)--(0.5,.75)--(.5,1);
					\draw[bmod] (0.5, 0) -- (.5, .25)--(0,.75)--(0,1);
					\draw[bmod] (0.5, 0.25) -- (1, .75)--(1,1);
					\draw[bmod] (1, 0) -- (1, .25)--(.5,.75);
					\btoken{0.5, .25}{east}{f};
					\btoken{0.5, .75}{west}{g};
				\end{tikzpicture} \ . 
			\end{equation*}
\end{lem}
\begin{proof}
The proof is obtained by the natural isomorphism of the braiding $(c_{U,X\otimes Y}, \id_U,g)$, and Lemma \ref{identityofbraid}. 
\end{proof}
		
We now recall the diagrammatic definitions of monoids, comonoids, bimonoids, and Lie monoids in $\calC$. For more details, we refer to \cite{EGNO15,TV17,Maj93,Maj94}.
\begin{defi}
The triple $(A,m,\eta)$ is called a \emph{monoid} in $\calC$ if $A\in \ob(\calC)$, $m: A\otimes A \rightarrow A$ and $\eta: I \rightarrow A $ are morphisms such that the following axioms are satisfied:
			\begin{equation*}\label{defiofasso}
				\begin{tikzpicture}[anchorbase]
					\draw[bmod] (0, 0) -- (0.5, .5)--(0.5,.75);
					\draw[bmod] (0.5, 0) -- (.25, .25);
					\draw[bmod] (0.75, 0) -- (.75, .25)--(.5,.5);
					\btoken{0.25, .25}{east}{m};
					\btoken{0.5, .5}{west}{m};
				\end{tikzpicture}\ = \ 
				\begin{tikzpicture}[anchorbase]
					\draw[bmod] (0, 0) -- (0, .25)--(0.25,.5)--(0.25,.75);
					\draw[bmod] (0.25, 0) -- (.5, .25)--(.25,.5);
					\draw[bmod] (0.75, 0) -- (.5, .25);
					\btoken{0.25,.5}{east}{m};
					\btoken{0.5,.25}{west}{m};
				\end{tikzpicture} \ \ \text{and} \ \ 
				\begin{tikzpicture}[anchorbase]
					\draw[bmod] (0, 0) -- (0, .25)--(0.25,.5)--(0.25,.75);
					\draw[bmod] (0.5, 0) -- (.5, .25)--(.25,.5);
					\btoken{0.25,.5}{west}{m};
					\btoken{0,0}{east}{\eta};
				\end{tikzpicture} \ = \ 
				\begin{tikzpicture}[anchorbase]
					\draw[bmod] (0, 0) -- (0, .25)--(0.25,.5)--(0.25,.75);
					\draw[bmod] (0.5, 0) -- (.5, .25)--(.25,.5);
					\btoken{0.25,.5}{east}{m};
					\btoken{0.5,0}{west}{\eta};
				\end{tikzpicture} \ = \ 
				\begin{tikzpicture}[anchorbase]
					\draw[bmod] (0, 0) -- (0, .75);
				\end{tikzpicture} \ .
			\end{equation*}
\end{defi}
		
Moreover, the monoid $(A,m,\eta)$ in $\calC$ is called \emph{commutative} if the multiplication satisfies the following axiom:
		\begin{equation*}
			\begin{tikzpicture}[anchorbase]
				\draw[bmod] (0, 0) -- (0.5, .5)--(0.25,.75)--(0.25,1);
				\draw[bmod] (0.5, 0) -- (0, .5)--(.25,.75);
				\btoken{0.25,.75}{east}{m};
			\end{tikzpicture} \ = \ 
			\begin{tikzpicture}[anchorbase]
				\draw[bmod] (0, 0) -- (0, .5)--(0.25,.75)--(0.25,1);
				\draw[bmod] (0.5, 0) -- (0.5, .5)--(.25,.75);
				\btoken{0.25,.75}{west}{m};
			\end{tikzpicture} \ . 
		\end{equation*}
		
Given two monoids $(A,m,\eta)$ and $(A',m',\eta')$, a morphism $f:A \rightarrow A'$ is said to be a \emph{homomorphism of monoids} if it satisfies the following axioms:
		\begin{equation*}
			\begin{tikzpicture}[anchorbase]
				\draw[bmod] (0, 0) -- (0, .25)--(0.25,.5)--(0.25,1);
				\draw[bmod] (0.5, 0) -- (0.5, .25)--(.25,.5);
				\btoken{0.25,.5}{west}{m};
				\btoken{0.25,.75}{west}{f};
			\end{tikzpicture} \ = \ 
			\begin{tikzpicture}[anchorbase]
				\draw[bmod] (0, 0) -- (0, .5)--(0.25,.75)--(0.25,1);
				\draw[bmod] (0.5, 0) -- (0.5, .5)--(.25,.75);
				\btoken{0.25,.75}{west}{m'};
				\btoken{0,.25}{east}{f};
				\btoken{0.5,.25}{west}{f};
			\end{tikzpicture} \ \ \text{and} \ \ 
			\begin{tikzpicture}[anchorbase]
				\draw[bmod] (0, 0) -- (0, 1);
				\btoken{0,.75}{west}{f};
				\btoken{0,0}{west}{\eta};
			\end{tikzpicture} \ = \ 
			\begin{tikzpicture}[anchorbase]
				\draw[bmod] (0, 0) -- (0, 1);
				\btoken{0,0}{west}{\eta'};
			\end{tikzpicture} \ . 
		\end{equation*}
		
\begin{defi}
Let $(A,m,\eta)$ be a monoid in $\calC$.  The pair $(V, \beta)$ is called an \emph{$A$-module} if $V\in \ob(\calC)$ and $\beta:A\otimes V \rightarrow V$ is a morphism such that the following axioms are satisfied: 
			\begin{equation*}
				\begin{tikzpicture}[anchorbase]
					\draw[bmod] (0, 0) -- (0.5, .5)--(0.5,.75);
					\draw[bmod] (0.5, 0) -- (.25, .25);
					\draw[bmod] (0.75, 0) -- (.75, .25)--(.5,.5);
					\btoken{0.25, .25}{east}{m};
					\btoken{0.5, .5}{west}{\beta};
				\end{tikzpicture}\ = \ 
				\begin{tikzpicture}[anchorbase]
					\draw[bmod] (0, 0) -- (0, .25)--(0.25,.5)--(0.25,.75);
					\draw[bmod] (0.25, 0) -- (.5, .25)--(.25,.5);
					\draw[bmod] (0.75, 0) -- (.5, .25);
					\btoken{0.25,.5}{east}{\beta};
					\btoken{0.5,.25}{west}{\beta};
				\end{tikzpicture} \ \ \text{and} \ \ 
				\begin{tikzpicture}[anchorbase]
					\draw[bmod] (0, 0) -- (0, .25)--(0.25,.5)--(0.25,.75);
					\draw[bmod] (0.5, 0) -- (.5, .25)--(.25,.5);
					\btoken{0.25,.5}{west}{\beta};
					\btoken{0,0}{east}{\eta};
				\end{tikzpicture} \  = \ 
				\begin{tikzpicture}[anchorbase]
					\draw[bmod] (0, 0) -- (0, .75);
				\end{tikzpicture} \ .
			\end{equation*}
\end{defi}
		
Note that $(A,m)$ is an $A$-module. The following lemma is immediate.
\begin{lem}\label{tensorofamod}
Let $(A,m,\eta)$ be a monoid in $\calC$. Given an $A$-module $(M,\beta)$, then  $(V\otimes M, \beta')$ is an $A$-module for any $V\in \ob(\calC)$, where the morphism $\beta':A\otimes V\otimes M \rightarrow V\otimes M$ is defined by
			\begin{equation*}
				\beta' \ = \ 
				\begin{tikzpicture}[anchorbase]
					\draw[bmod] (0, 0) -- (0.25, .25)--(0.25,.5);
					\draw[bmod] (0.5, 0) -- (.25, .25);
					\draw[bmod] (0.25, 0) -- (0, .25)-- (0, .5);
					\btoken{0.25,0.25}{west}{\eta};
				\end{tikzpicture}\ .
			\end{equation*} 
			In particular, the pair $\big(V\otimes A, (\id_V\otimes m)\circ (c_{A,V}\otimes \id_A)\big)$ is an $A$-module for any $V\in \ob(\calC)$.
\end{lem}
		
Assume that $(V, \beta)$ and $(V',\beta')$ are two $A$-modules. A morphism $\varphi:V \rightarrow V'$ is called a \emph{homomorphism of $A$-modules} if the following axiom is satisfied: 
		\begin{equation*}
			\begin{tikzpicture}[anchorbase]
				\draw[bmod] (0, 0) -- (0.25, .25)--(0.25,.75);
				\draw[bmod] (0.5, 0) -- (.25, .25);
				\btoken{0.25,0.5}{west}{\varphi};
				\btoken{0.25,0.25}{west}{\beta};
			\end{tikzpicture} \ = \ 
			\begin{tikzpicture}[anchorbase]
				\draw[bmod] (0, 0) -- (0, .25)--(0.25,.5)--(0.25,.75);
				\draw[bmod] (0.5, 0) -- (.5, .25)--(.25,.5);
				\btoken{0.5,0.25}{west}{\varphi};
				\btoken{0.25,0.5}{east}{\beta'};
			\end{tikzpicture}\ . 
		\end{equation*}
		
It is straightforward to verify that the composition of homomorphisms of $A$-modules remains a homomorphism. Consequently, all $A$-modules and their homomorphisms form a category {\bf Mod}$_\calC(A)$. 
		
\begin{defi}
The triple $(C,\Delta,\varepsilon)$ is called a \emph{comonoid} in $\calC$ if $C\in \ob(\calC)$, $\Delta: C \rightarrow C\otimes C$ and $\varepsilon: C \rightarrow I $ are morphisms such that the following axioms are satisfied:
			\begin{equation*}
				\begin{tikzpicture}[anchorbase]
					\draw[bmod] (0.5, 0)--(.5,.25) -- (0,.75);
					\draw[bmod] (0.25, 0.5) -- (.5, .75);
					\draw[bmod] (0.5, 0.25) -- (.75, .5)--(.75,.75);
					\btoken{0.5,0.25}{west}{\Delta};
					\btoken{0.25,0.5}{east}{\Delta};
				\end{tikzpicture} \ = \ 
				\begin{tikzpicture}[anchorbase]
					\draw[bmod] (0.25, 0)--(.25,.25) -- (0,.5)-- (0,.75);
					\draw[bmod] (0.25, 0.25) -- (.75, .75);
					\draw[bmod] (0.5, 0.5) -- (.25, .75);
					\btoken{0.25,0.25}{west}{\Delta};
					\btoken{0.5,0.5}{west}{\Delta};
				\end{tikzpicture} \ \ \text{and} \ \ 
				\begin{tikzpicture}[anchorbase]
					\draw[bmod] (0.25, 0)--(.25,.25) -- (0,.5)-- (0,.75);
					\draw[bmod] (0.25, 0.25) -- (.5, .5)-- (.5, .75);
					\btoken{0.25,0.25}{east}{\Delta};
					\btoken{0.5,0.75}{west}{\varepsilon};
				\end{tikzpicture} \ = \ 
				\begin{tikzpicture}[anchorbase]
					\draw[bmod] (0.25, 0)--(.25,.25) -- (0,.5)-- (0,.75);
					\draw[bmod] (0.25, 0.25) -- (.5, .5)-- (.5, .75);
					\btoken{0.25,0.25}{west}{\Delta};
					\btoken{0,0.75}{east}{\varepsilon};
				\end{tikzpicture} \ = \ 
				\begin{tikzpicture}[anchorbase]
					\draw[bmod] (0, 0)--(0,.75);
				\end{tikzpicture} \ . 
			\end{equation*}
\end{defi}
		
\begin{defi}
The tuple $(H,m,\eta,\Delta,\varepsilon)$ is called a bimonoid if $H\in \ob(\calC)$, $(H,m,\eta)$ is a monoid and $(H,\Delta,\varepsilon)$ is a comonoid such that the following axioms are satisfied:
			\begin{equation*}
				\begin{tikzpicture}[anchorbase]
					\draw[bmod] (0, 0)--(.25,.25) -- (0.25,.75)-- (0,1);
					\draw[bmod] (0.25, 0.75) -- (.5, 1);
					\draw[bmod] (0.5, 0) -- (.25, .25);
					\btoken{0.25,0.75}{west}{\Delta};
					\btoken{0.25,0.25}{east}{m};
				\end{tikzpicture}\ = \
				\begin{tikzpicture}[anchorbase]
					\draw[bmod] (0.25, 0)--(.25,.25) -- (0,.5)-- (0.25,.75)--(.25,1);
					\draw[bmod] (0.25, 0.25) -- (.75, .75)--(.75,1);
					\draw[bmod] (0.75, 0) -- (.75, .25)--(.25,.75);
					\draw[bmod] (0.75, 0.25) -- (1, .5)--(.75,.75);
					\btoken{0.25,0.25}{east}{\Delta};
					\btoken{0.75,0.25}{west}{\Delta};
					\btoken{0.25,0.75}{east}{m};
					\btoken{0.75,0.75}{west}{m};
				\end{tikzpicture}\ , \ 
				\begin{tikzpicture}[anchorbase]
					\draw[bmod] (0.25, 0)--(.25,.75) -- (0,1)-- (0,1);
					\draw[bmod] (0.25, 0.75) -- (.5, 1);
					\btoken{0.25,0}{west}{\eta};
					\btoken{0.25,0.75}{east}{\Delta};
				\end{tikzpicture} \ = \ 
				\begin{tikzpicture}[anchorbase]
					\draw[bmod] (0, 0)--(0,1);
					\draw[bmod] (0.25, 0) -- (.25, 1);
					\btoken{0.25,0}{west}{\eta};
					\btoken{0,0}{east}{\eta};
				\end{tikzpicture} \ , \ 
				\begin{tikzpicture}[anchorbase]
					\draw[bmod] (0, 0)--(.25,.25) -- (0.25,1);
					\draw[bmod] (0.5, 0) -- (.25, .25);
					\btoken{0.25,1}{west}{\varepsilon};
					\btoken{0.25,.25}{east}{m};
				\end{tikzpicture} \ = \ 
				\begin{tikzpicture}[anchorbase]
					\draw[bmod] (0, 0)--(0,1);
					\draw[bmod] (0.25, 0) -- (.25, 1);
					\btoken{0.25,1}{west}{\varepsilon};
					\btoken{0,1}{east}{\varepsilon};
				\end{tikzpicture} \ \ \text{and} \ \ 
				\begin{tikzpicture}[anchorbase]
					\draw[bmod] (0, 0)--(0,1);
					\btoken{0,0}{west}{\eta};
					\btoken{0,1}{west}{\varepsilon};
				\end{tikzpicture} \ = \ \mathrm{id}_I \ . 
			\end{equation*}
\end{defi}
		
Let $(H,m,\eta,\Delta,\varepsilon)$ be a bimonoid in $\calC.$ As established in \cite{Maj93,Maj94}, the category {\bf Mod}$_{\calC}(H)$ of all left $H$-modules admits a monoidal structure. Explicitly, given $(M,\rho), (M',\rho')\in $ {\bf Mod}$_{\calC}(H)$, 
there is a morphism from $H\otimes (M\otimes M')$ to $M\otimes M'$ as follows:
		\begin{equation}\label{tensorofasso}
			\rho \boxtimes \rho' \ = \ 
			\begin{tikzpicture}[anchorbase]
				\draw[bmod] (0.25, 0)--(0.25,.25)--(0, 0.5)--(0.25,.75)--(0.25,1);
				\draw[bmod] (0.25, 0.25) -- (.75, .75)-- (.75, 1);
				\draw[bmod] (0.75, 0) -- (.75, .25)-- (.25, .75);
				\draw[bmod] (1, 0) -- (1, .5)-- (.75, .75);
				\btoken{0.25,.25}{west}{\Delta};
				\btoken{0.25,.75}{east}{\rho};
				\btoken{0.75,.75}{west}{\rho'};
			\end{tikzpicture}\ .
		\end{equation}
		
Then $(M\otimes M', \rho\boxtimes \rho')\in$ {\bf Mod}$_{\calC}(H)$. Moreover, the unit object is $(I,\varepsilon)$, where $\varepsilon: H \to I$ equips $I$ with the trivial module structure. 
		
Recall that a pair $(L, \mu)$ is called a \emph{Lie monoid} in $\calC$ if $L\in \ob(\calC)$ and $\mu : L\otimes L \rightarrow L$ is a morphism satisfying certain additional conditions. Furthermore, one may define not only homomorphisms of Lie monoids, but also modules over $L$ and homomorphisms between such $L$-modules.  For  precise definitions, we refer to Definitions  3.1 and 3.3 in \cite{Sam25}.
		
\begin{exam}\label{exoflie}
Let $(A,m,\eta)$ be a monoid in $\calC$. Then $(A,\mu_-)$ forms a new Lie monoid in $\calC$, where
			\begin{equation*}
				\mu_- \ = \ 
				\begin{tikzpicture}[anchorbase]
					\draw[bmod] (0, 0)--(0,.25)--(0.25, 0.5)--(0.25,.75);
					\draw[bmod] (0.5, 0) -- (.5, .25)-- (.25, .5);
					\btoken{0.25,.5}{west}{m};
				\end{tikzpicture} \ - \ 
				\begin{tikzpicture}[anchorbase]
					\draw[bmod] (0, 0)--(0.5,.25)--(0.25, 0.5)--(0.25,.75);
					\draw[bmod] (0.5, 0) -- (0, .25)-- (.25, .5);
					\btoken{0.25,.5}{west}{m};
				\end{tikzpicture}\ . 
			\end{equation*}
\end{exam}

\begin{rmk}
As noted in {\cite[Theorem 1.27]{AM10}}, the definition of Lie monoids can be formulated in any linear braided monoidal category, not necessarily symmetric. However, the above Example \ref{exoflie} fails at this level of generality. For this and other reasons, we restrict the consideration of Lie monoids to the context of linear symmetric monoidal categories. 
\end{rmk}
		
All $L$-modules form a category ${\bf Mod}_{\calC}(L)$, since a composition of homomorphisms is a homomorphism. By Proposition $3.23$ in \cite{Sam25}, the category {\bf Mod}$_{\calC}(L)$ forms a monoidal category.
For $(V, \psi), (W,\gamma)\in $ {\bf Mod}$_{\calC}(L)$, define
		\begin{equation}\label{tensoroflmodule}
			\psi \boxplus \gamma \ = \ 
			\begin{tikzpicture}[anchorbase]
				\draw[bmod] (0, 0)--(0.5,.25)--(0.25, 0.5)--(0.25,.75);
				\draw[bmod] (0.5, 0) -- (0, .25)-- (.25, .5);
				\btoken{0.25,.5}{west}{m};
			\end{tikzpicture}\ .
		\end{equation}
		
Then the pair $(V\otimes W, \psi \boxplus  \gamma) \in$  {\bf Mod}$_{\calC}(L)$.
The unit object is $(I,{\bf 0})$, where ${\bf 0}$ denotes the zero morphism.

\section{\bf Acknowledgments }
The authors are grateful to   N. Andruskiewitsch, I. Angiono, A. Makhlouf and A. Savage for useful discussions. The second author was supported by NSFC No. 12350710787.


\begin{thebibliography}{99}
	\bibitem{AM10} M. Aguiar, S. Mahajan, Monoidal functors, species and Hopf algebras, CRM Monogr. Ser. 29, American Mathematical Society, Providence, RI, 2010.
	
	\bibitem{Ben63} J. B\'enabou, Cat\'egories avec multiplication, 
	C. R. Acad. Sci. Paris,  {\bf 256} (1963), 1887--1890.
	
	\bibitem{BKS24} X. Bekaert, N. Kowalzig, P. Saracco, Universal enveloping algebras of Lie-Rinehart algebras: crossed products, connections, and curvature, Lett. Math. Phys.,  {\bf 114}  (2024),  73 pp.
	
	
	
	
	
	
	\bibitem{BEMS} S.  Bouarroudj, Q.  Ehret, A. Makhlouf, N. Shyntas, The Superization of Hochschild's Lemma and Restricted Lie-Rinehart Superalgebras, arXiv:2511.18372. 
	
	\bibitem{CLP04}  J. M. Casas, M.  Ladra, T. Pirashvili, Crossed modules for Lie-Rinehart algebras,
	J. Algebra, {\bf 274} (2004),  192--201.
	
	\bibitem{Che95} S. Chemla, Operations for modules on Lie-Rinehart superalgebras, Manuscripta Math., {\bf 87} (1995), 199--223.
	
	\bibitem{EGNO15} P. Etingof, S. Gelaki, D. Nikshych, and V. Ostrik, Tensor categories, Math. Surveys Monogr. 205, American Mathematical Society, Providence, RI, 2015.
	
	
	\bibitem{Hue90} J. Huebschmann, Poisson cohomology and quantization, J. Reine Angew. Math., {\bf 408} (1990), 57--113.
	
	\bibitem{Hue04} J. Huebschmann, Lie-Rinehart algebras, descent, and quantization, in Galois Theory, Hopf Algebras, and Semiabelian Categories, vol. 43 of Fields Inst. Commun., Amer. Math. Soc., Providence, RI, 2004, 295--316.
	
	\bibitem{Hue21} J. Huebschmann, On the history of Lie brackets, crossed modules, and Lie-Rinehart algebras,
	J. Geom. Mech., {\bf 13} (2021), 385--402.
	
	
	
	\bibitem{JS93} A. Joyal, R. Street, Braided tensor categories, Adv. Math., {\bf 102} (1993), 20--78.
	
	\bibitem{Kas95} C. Kassel, Quantum groups, Graduate Texts in Mathematics 155, Springer, New York, 1995.
	\bibitem{KT95} C. Kassel,  V. Turaev,  Double construction for monoidal categories, Acta Math., {\bf 175} (1995),  1--48.	
	
	\bibitem{Kel06} B. Keller, On differential graded categories, 151--190,  International Congress of Mathematicians II, Z\"urich: European Mathematical Society, 2006.
	
	\bibitem{Kel64} G. M. Kelly, On MacLane's conditions for coherence of natural associativities, commutativities, etc, J. Algebra, {\bf 1} (1964), 397--402.
	
	
	
	\bibitem{Lar92} T. A. Larsson, Conformal fields: A class of representations of $\mathrm{Vect}(N)$, Int. J. Mod. Phys. A,  {\bf 7} (1992), 6493--6508. 
	
	\bibitem{LX23} R. L\"{u}, Y. Xue, Bounded weight modules over the Lie superalgebra of Cartan $W$-type, Algebr. Represent. Theory, {\bf 26} (2023),  763--781.
	
	
	
	
	
	\bibitem{Mac63} S. MacLane, Natural associativity and commutativity, Rice Univ. Stud., {\bf 49} (1963), 28--46.
	
	\bibitem{Mac98} S. MacLane,  \textit{Category Theory for the Working Mathematician},  Springer-Verlag, New York, second edition, 1998.
	
	\bibitem{Maj93} S. Majid, Braided groups, J. Pure Appl. Algebra, {\bf 86} (1993), 187--221.
	
	\bibitem{Maj94} S. Majid, Algebras and Hopf algebras in braided categories, In: ``Advances in Hopf Algebras", Marcel Dekker Lecture Notes in Pure and Applied Mathematics, vol. 158, pp. 55--105 (1994).
	
	\bibitem{PSTZ23} Y. Pei,  Y.  Sheng,  R. Tang, K.  Zhao,  Actions of monoidal categories and representations of Cartan type Lie algebras, 
	J. Inst. Math. Jussieu, {\bf 22}  (2023),  2367--2402.
	
	\bibitem{Rin63} G. S. Rinehart, Differential forms on general commutative algebras, Trans. Amer. Math. Soc., {\bf 108} (1963), 195--222.
	
	\bibitem{Sam25} S. Samchuck-Schnarch, Towards interpolating categories for equivariant map algebras, arXiv: 2504.21163v1. 
	
	\bibitem{Sar22} P. Saracco, Universal enveloping algebras of Lie-Rinehart algebras as a left adjoint functor, Mediterr. J. Math., {\bf 19} (2022), 19 pp. 
	
	\bibitem{She86} G. Shen, Graded modules of graded Lie algebras of Cartan type. I. Mixed products of modules, Sci. Sinica Ser. A, {\bf 29} (1986), 570--581. 
	
	
	\bibitem{TV17} V. Turaev, A. Virelizier, Monoidal Categories and Topological Field Theory, Progress in Mathematics, vol. 322, Springer, Cham, 2017.
	
\end{thebibliography}
	\end{document}